\documentclass[11pt]{article}
\usepackage{graphicx,afterpage,times}
\usepackage{lscape}
\usepackage{dsfont}
\usepackage{bm}
\usepackage{wrapfig}
\usepackage{soul}
\usepackage{anyfontsize}
\usepackage{subfig}
\usepackage{algorithm2e}
\RestyleAlgo{ruled}
\usepackage{enumitem}
\usepackage{cancel}
\usepackage[dvipsnames]{xcolor}
\usepackage{empheq}\usepackage{mathtools}
\usepackage[top=0.5in, bottom=0.8in, left=0.5in, right=0.5in]{geometry}
\usepackage{mathrsfs,color}
\usepackage[normalem]{ulem}
\usepackage{amsfonts}\usepackage{multirow}
\usepackage{amssymb}\usepackage{caption,comment}
\usepackage{amsmath}\usepackage{float}
\usepackage{relsize}
\usepackage{amssymb,amsthm}
\usepackage[toc,page,header,title]{appendix}
\usepackage{graphicx,afterpage,cancel}
\usepackage{epstopdf}
\usepackage{mdframed}
\newcommand{\rmd}{\mathrm{d}}

\usepackage[pdftex,colorlinks=true]{hyperref}
\hypersetup{CJKbookmarks,%
bookmarksnumbered,%
colorlinks,%
linkcolor=black,%
citecolor=blue,%
plainpages=false,%
pdfstartview=FitH}
\numberwithin{equation}{section}
\numberwithin{figure}{section}

\newcommand\tabcaption{\def\@captype{table}\caption}

\newtheoremstyle{theorem}%
  {3pt}
  {3pt}
  {}
  {}
  {\bfseries\color{red}}
  {\textcolor{red}{.}}
  {.5em}
  {}
\theoremstyle{theorem}
\newtheorem{thm}{Theorem}[section]

\newtheoremstyle{remark}%
  {3pt}
  {3pt}
  {}
  {}
  {\bfseries\color{magenta}}
  {\textcolor{magenta}{.}}
  {.5em}
  {}
\theoremstyle{remark}
\newtheorem{rmk}[thm]{Remark}

\newtheoremstyle{lemma}%
  {3pt}
  {3pt}
  {}
  {}
  {\bfseries\color{blue}}
  {\textcolor{blue}{.}}
  {.5em}
  {\thmname{#1} \thmnumber{#2}\thmnote{ (#3)}}
\theoremstyle{lemma}
\newtheorem{lem}[thm]{Lemma}

\newtheoremstyle{definition}%
  {3pt}
  {3pt}
  {}
  {}
  {\bfseries\color{green}}
  {\textcolor{green}{.}}
  {.5em}
  {\thmname{#1} \thmnumber{#2}\thmnote{ (#3)}}
\theoremstyle{definition}
\newtheorem{defn}[thm]{Definition}

\definecolor{orange}{RGB}{255,127,0}

\newcommand{\magenta}[1]{\textcolor{magenta}{#1}}

\newcommand{\ml}{\boldsymbol{\Lambda}}
\newcommand{\ms}{\boldsymbol{\Sigma}}
\newcommand{\mr}[1]{\mathbf{R}_{\text{#1}}}
\newcommand{\mrt}[1]{\mathbf{R}_{\text{#1}}(t)}

\newcommand{\zeros}{\mathbf{0}}
\newcommand{\vx}{\mathbf{x}}
\newcommand{\vy}{\mathbf{y}}
\newcommand{\ve}{\boldsymbol{\epsilon}}
\newcommand{\vz}{\mathbf{z}}

\newcommand{\vh}{\mathbf{h}}
\newcommand{\vf}{\mathbf{f}}
\newcommand{\vxt}{\mathbf{x}(t)}
\newcommand{\vyt}{\mathbf{y}(t)}
\newcommand{\vzt}{\mathbf{z}(t)}

\newcommand{\vm}[1]{\boldsymbol{\mu}_{\text{#1}}}
\newcommand{\vmt}[1]{\boldsymbol{\mu}_{\text{#1}}(t)}
\newcommand{\smooth}[1]{\overleftarrow{#1}}
\newcommand{\dt}{\Delta t}
\newcommand{\vw}{\mathbf{W}}
\newcommand{\vwt}[1]{\mathbf{W}_{#1}(t)}
\newcommand{\va}{\boldsymbol{\alpha}}
\newcommand{\ma}{\mathbf{A}}
\newcommand{\mb}{\mathbf{B}}
\newcommand{\mc}{\boldsymbol{\Gamma}}
\newcommand{\cF}{\mathcal{F}}
\newcommand{\pp}{\mathbb{P}}
\newcommand{\rr}{\mathbb{R}}
\newcommand{\tran}{\mathtt{T}}
\newcommand{\ee}[1]{\mathbb{E}\left[#1\right]}
\newcommand{\nf}{\normalfont{f}}
\newcommand{\ns}{\normalfont{s}}

\makeatletter
\renewcommand*{\@cite@ofmt}{\bfseries\hbox}

\makeatother

\allowdisplaybreaks

\title{A Martingale-Free Introduction to Conditional Gaussian Nonlinear Systems}
\author{Marios Andreou and Nan Chen}
\date{\today}
\begin{document}

\maketitle\tableofcontents

\begin{abstract}
The Conditional Gaussian Nonlinear System (CGNS) is a broad class of nonlinear stochastic dynamical systems. Given the trajectories for a subset of state variables, the remaining follow a Gaussian distribution. Despite the conditionally linear structure, the CGNS exhibits strong nonlinearity, thus capturing many non-Gaussian characteristics observed in nature through its joint and marginal distributions. Desirably, it enjoys closed analytic formulae for the time evolution of its conditional Gaussian statistics, which facilitate the study of data assimilation and other related topics. In this paper, we develop a martingale-free approach to improve the understanding of CGNSs. This methodology provides a tractable approach to proving the time evolution of the conditional statistics by deriving results through time discretization schemes, with the continuous-time regime obtained via a formal limiting process as the discretization time-step vanishes. This discretized approach further allows for developing analytic formulae for optimal posterior sampling of unobserved state variables with correlated noise. These tools are particularly valuable for studying extreme events and intermittency and apply to high-dimensional systems. Moreover, the approach improves the understanding of different sampling methods in characterizing uncertainty. The effectiveness of the framework is demonstrated through a physics-constrained, triad-interaction climate model with cubic nonlinearity and state-dependent cross-interacting noise.

\end{abstract}

\section{Introduction} \label{sec:1}

Complex turbulent nonlinear dynamical systems (CTNDSs) are prevalent in various fields such as geoscience, engineering, and material science \cite{majda2016introduction, majda2006nonlinear, strogatz2018nonlinear, deisboeck2007complex, kitano2001foundations, sheard2008principles, wilcox1988multiscale}. These systems are characterized by high-dimensional state spaces, multiscale structures, and chaotic dynamics with positive Lyapunov exponents \cite{dijkstra2013nonlinear, palmer1993nonlinear, majda2012physics, sapsis2013statistically, majda2012filtering, harlim2014ensemble}. They often exhibit non-Gaussian features, such as intermittent extreme events and skewed probability density functions \cite{cousins2014quantification, farazmand2019extreme, denny2009prediction, mohamad2015probabilistic}. Key mathematical challenges include understanding their qualitative properties \cite{majda2006nonlinear, strogatz2018nonlinear, wiggins2003introduction}, forecasting, state estimation \cite{majda2012filtering, lahoz2010data, evensen2009data, law2015data}, uncertainty quantification (UQ), causal dependence \cite{bossomaier2016introduction, kim2017causation, almomani2020entropic, elinger2020information, chen2023causality}, developing reduced-order models \cite{brunton2016discovering, ahmed2021closures, chekroun2017data, lin2021data, mou2021data, peherstorfer2015dynamic, hijazi2020data, smarra2018data, chen2024minimum}, and accurate prediction of their response to perturbations \cite{majda2010quantifying, ragone2016new, lucarini2017predicting, andreou2024statistical}.

Due to the lack of perfect knowledge of nature and potential computational limits, CTNDSs are often approximated with unavoidable errors \cite{majda2012lessons, palmer2001nonlinear, orrell2001model, hu2010ensemble, benner2015survey}. State estimation is crucial for parameter estimation, prediction, and control. However, model errors in turbulent systems can be easily amplified in time. Therefore, data assimilation (DA) is widely used in practice to mitigate such errors. DA combines model information with available observations to improve state estimation. For turbulent systems, Bayesian inference is often used in DA to provide a statistical estimate known as the posterior distribution, with the model defining the so-called prior distribution while the observations construe the likelihood of the data. In many applications, observations are limited to only a subset of variables \cite{majda2012filtering, kalnay2003atmospheric, lahoz2010data, evensen2009data, law2015data}. This creates additional challenges in DA, especially when inferring the unobserved variables.

Depending on the use of observations, DA methods can often be divided into two categories: filtering and smoothing. Filtering uses past observations up to the present, making it essential for initializing real-time forecast models. Smoothing, which incorporates both past and future data, provides more accurate offline estimates, particularly in reanalysis for climate and oceanography \cite{rauch1965maximum, chen2020efficient, sarkka2023bayesian, uppala2005era}. For linear systems with additive Gaussian noise, the celebrated Kalman filters and Rauch-Tung-Striebel smoothers offer closed-form solutions for the posterior distributions \cite{rauch1965maximum, kalman1961new, bucy1987filtering, bierman1977factorization}. Beyond improving state estimation, the optimal filter and smoother statistics enable effective sampling of trajectories, offering insights beyond the point-wise estimates. Since these sampled trajectories incorporate the information from the observations, they are often more accurate than those simulated from imperfect models. These state estimation and sampling methods are widely used in practice to understand dynamical properties, develop surrogate models, assist in building stochastic parameterizations, and quantify uncertainty and model error.

Closed-form solutions for DA are rarely available for general CTNDSs due to their nonlinear dynamics and non-Gaussian statistics, especially for sampling and estimating the latent dynamics. As a result, numerical methods such as the ensemble Kalman filter (EnKF) and particle filter have been developed for nonlinear or non-Gaussian systems \cite{evensen2009data, law2015data, anderson2001ensemble, delmoral1997nonlinear, liu1998sequential}, along with their smoother counterparts \cite{evensen2000ensemble, kitagawa1996monte}. These methods are widely used across various scientific disciplines and are essential for ensemble forecasting in operational systems \cite{molteni1996ecmwf, palmer2019ecmwf, toth1997ensemble}. However, simulating high-dimensional CTNDSs is computationally expensive, limiting the number of ensemble members or particles that can be used, which can introduce a nontrivial bias and instability in the state estimates (such as finite-time blowup of the
solution) \cite{gottwald2013mechanism, harlim2010catastrophic}, with the number of samples needed to address these stability concerns increasing at an exponential rate \cite{snyder2008obstacles}. Techniques like noise inflation, localization, and resampling help address these issues, but they are often empirical and lack natural integration into the state estimation framework \cite{anderson2001ensemble, chen2020predicting, hol2006resampling, greybush2011balance}. Furthermore, as inherently approximate in nature, the careful study of the accuracy of these pro tem approaches is warranted. Thus, closed-form analytic methods are preferred for enhancing not only the computational efficiency and
numerical behavior of probabilistic state estimation, but also in promoting the accurate capture of the salient non-Gaussian features of the dynamics, e.g., intermittent extreme events, and also in facilitating the theoretical study of the model error, as well as the uncertainty in the estimated state.

Instead of modifying DA schemes directly, computational challenges in state estimation can be addressed by developing approximate models with analytic solutions for posterior distributions while retaining key system characteristics. Linearizing strongly nonlinear systems can introduce significant errors in the mean's and higher-order moments' time evolution, but a special class of systems, known as conditional Gaussian nonlinear systems (CGNSs), offers a better alternative \cite{chen2022conditional, liptser2001statisticsI, liptser2001statisticsII, chen2016filtering, chen2018conditional}. Despite the intrinsic nonlinearity and non-Gaussian statistics, the conditional distribution of the unobserved variables given one time series of the observed variables is Gaussian. This conditional Gaussianity provides closed-form solutions for the optimal filter and smoother statistics, which facilitate efficient and accurate state estimation without empirical tuning. The CGNS framework has been successfully applied in DA, prediction, and machine learning, offering analytic tools for performance assessment and error quantification \cite{chen2022conditional, chen2024cgnsde, chen2014predicting, chen2014information}.

This paper aims to introduce the CGNS framework without relying on martingale theory \cite{liptser2001statisticsII} to establish the conditional Gaussianity of the posterior distribution or to derive optimal nonlinear state estimation equations. This approach unifies discrete- and continuous-time formulations, which are often treated separately, by showing that the discrete system can converge to its continuous counterpart under appropriate consistency and stability conditions. In addition to state estimation through filtering and smoothing, we explore the conditional sampling of the unobserved variables when correlated noise appears. These sampled trajectories, similar to ensemble members in the ensemble DA, preserve temporal dependencies. The temporal correlation distinguishes them from time series formed by independent posterior samples at different instants. The sampling schemes are derived using forward (filter-based) and backward (smoother-based) methods, with closed-form solutions provided for both. This facilitates systematically comparing these methods in capturing nonlinear and non-Gaussian features associated with the unobserved states, especially for intermittent behaviors and extreme events. The sampled trajectories further highlight the limitations of treating the posterior mean estimates as surrogates of the hidden variables. Finally, we discuss a hierarchy that emerges in the fluctuation component of the sample-to-sample uncertainty, which decreases as more information from the observable process is incorporated, depending on whether the filter- or smoother-based sampling methodology is being used.

The remainder of this work is organized as follows. Section \ref{sec:2} presents the CGNS modeling framework and derives the optimal nonlinear filter and smoother state estimation equations using a martingale-free approach. In Section \ref{sec:3}, we construct the filter-based forward and smoother-based backward sampling procedures for the hidden process and demonstrate their first- and second-order moment consistency with the Gaussian statistics of the optimal posterior distributions from Section \ref{sec:2}. We also compare their dynamical properties in terms of their uncertainty, specifically through their damping and fluctuation constituents, both together and with those from the unconditional forward run of the system. An illustrative numerical experiment in Section \ref{sec:4} highlights the differences between the optimal nonlinear filter and smoother, focusing on effective state estimation and sampling of the hidden intermittent time series. The article concludes in Section \ref{sec:5}. The Appendix includes the details of the proofs of most theorems in this work.

\section{The Conditional Gaussian Nonlinear System Modeling Framework} \label{sec:2}

\subsection{Basic Concepts} \label{sec:2.1}

Throughout this paper, boldface letters are exclusively used to denote vectors for mathematical clarity. Lowercase boldface letters represent column vectors, while uppercase boldface letters denote matrices. The only exception is $\vw$ (with subscripts or superscripts), which denotes a Wiener process. Although this always corresponds to a column vector in this work, we use an uppercase letter due to literary convention.

Let $t$ denote the time variable, with $t \in [0, T]$, where $T > 0$ may be infinite. Let $(\Omega, \mathcal{F}, \mathbb{P})$ be a complete probability space, and let $\{\mathcal{F}_t\}_{t \in [0, T]}$ be a filtration of sub-$\sigma$-algebras of $(\Omega, \mathcal{F})$. We assume this filtration is augmented (i.e., complete and right-continuous), forming the stochastic basis $(\Omega, \mathcal{F}, \{\mathcal{F}_t\}_{t \in [0, T]}, \mathbb{P})$. For any filtration, there exists a smallest augmented filtration refining $\{\mathcal{F}_t\}_{t \in [0, T]}$ (known as its completion), so this is without loss of generality. We denote the partially observable $(S, \mathcal{A})$-valued stochastic process as $(\mathbf{x}(t, \omega), \mathbf{y}(t, \omega))$ for $t \in [0, T]$ and $\omega \in \Omega$. Here, $\mathbf{x}$ is the observable component, while $\mathbf{y}$ is the unobservable component. The theory that follows can be applied mutatis mutandis to any partially observable stochastic process that takes values over a measurable space $(S, \mathcal{A})$, where $S$ is a separable Hilbert space (finite-dimensional or not) over a complete scalar ground field, and $\mathcal{A}$ is a $\sigma$-algebra on $S$. For this work, we consider complex-valued finite-dimensional processes with respect to the usual Euclidean inner product. Thus, we set $S = \mathbb{C}^{k+l}$ and $\mathcal{A} = \mathcal{B}_{\mathbb{C}^{k+l}} \equiv \mathcal{B}_{\mathbb{R}^{2(k+l)}}$, with $\mathbf{x}$ as a $k$-dimensional vector and $\mathbf{y}$ as an $l$-dimensional vector. Here, $\mathcal{B}_{\mathbb{R}^{2(k+l)}}$ is the Borel $\sigma$-algebra of $\mathbb{R}^{2(k+l)}$, since $\mathrm{dim}_{\mathbb{R}}(\mathbb{C}^{k+l}) = 2(k+l)$. We assume that
\begin{equation*}
(\mathbf{x}, \mathbf{y}) = \left((x_1(t), \ldots, x_k(t), y_1(t), \ldots, y_l(t)), \mathcal{F}_t\right), \quad t \in [0, T],
\end{equation*}
indicating that the partially observable random process is (jointly) adapted (or non-anticipative) to the filtration $\{\mathcal{F}_t\}_{t \in [0, T]}$. This means that for all times $t \in [0, T]$, the random vector defined by $\left(\mathbf{x}(t, \cdot)^\tran, \mathbf{y}(t, \cdot)^\tran\right)^\tran: \Omega \to \mathbb{C}^{k+l}$ is an $(\mathcal{F}_t; \mathcal{A})$-measurable function. Specifically, this implies that the natural filtration $\mathcal{F}$ with respect to $\{\mathbf{x}(s)\}_{s \leq t}$, which is the sub-$\sigma$-algebra generated by the observable processes for times $s \leq t$, is defined as
\begin{equation*}
\mathcal{F}_t^{\mathbf{x}} := \sigma\left(\{ \mathbf{x}(s) \}_{s \leq t}\right) = \left\{ \mathbf{x}(s)^{-1}[A] = \mathbf{x}(s, \cdot)^{-1}[A]: A \in \mathcal{A}, \, s \leq t \right\}.
\end{equation*}
This satisfies $\mathcal{F}_t^{\mathbf{x}} \subseteq \mathcal{F}_t$, since $\mathbf{x}$ is adapted to the filtration $\{\mathcal{F}_t\}_{t \in [0, T]}$ by construction. By definition, $\mathcal{F}_t^{\mathbf{x}}$ is the smallest such filtration. We refer to this natural filtration as the observable $\sigma$-algebra (at time $t$) for the remainder of this work.

It can be shown that, at each time instant $t\in[0,T]$, the optimal estimate in the mean-square sense, which minimizes the expected squared norm error of some $\cF_t$-measurable function $h(t,\vxt,\vyt)$ of $t$, $\vx$, and $\vy$, based on the observations up to time $t$, $\{\mathbf{x}(s)\}_{s\leq t}$, is exactly its conditional expectation conditioned on the observable $\sigma$-algebra at time $t$, or $\ee{\vh\big|\cF_t^\vx}=\ee{\vh(t,\vxt,\vyt)\big|\cF_t^\vx}$. This is known as the a-posteriori mean, and its optimality rests on the tacit assumption that $\ee{\left\|\vh(t,\vxt,\vyt)\right\|_2^2}$ is finite, where $\left\|\boldsymbol{\cdot}\right\|_2$ denotes the usual Euclidean norm over $\mathbb{C}^{\mathrm{dim}(\mathbf{h})}$ \cite{liptser2001statisticsI, liptser2001statisticsII}. This stochastic functional is to be understood as the up to a $\pp$-null set unique integrable $\mathbb{C}^{\mathrm{dim}(\mathbf{h})}$-valued and $\cF_t^{\vx}$-measurable random vector satisfying
\begin{equation*}
    \int_{F} \ee{\vh\big|\cF_t^\vx} \rmd\pp=\int_{F} \vh \rmd\pp, \quad \forall F\in \cF_t^\vx,
\end{equation*}
with the existence and uniqueness (up to $\pp$-null sets) of the a-posteriori mean being provided by the Hilbert projection theorem. This means that the random vector such that the residual $\vh-\ee{\vh\big|\cF_t^\vx}$ is orthogonal to the indicator functions $\mathds{1}_F$ under the Euclidean inner product over $L^2\left(\Omega, \cF_t, \pp; \mathbb{C}^{\mathrm{dim}(\mathbf{h})}\right)$, or the Lebesgue quotient space of square-integrable processes taking values in $\mathbb{C}^{\mathrm{dim}(\mathbf{h})}$, modulo the adapted processes with a squared null integral $\forall F\in\cF_t^\vx$ w.r.t.\ $\rmd\pp$, i.e., $\int_F \|\mathbf{b}(t)\|^2\rmd\pp=0$. Usually $\vh$ is mainly a function of the unobserved process $\vy$, and in this work we exclusively use $\vh=\vy$ to recover the optimal filter and smoother conditional Gaussian statistics of the hidden process when conditioning on the observations.

The goal in optimal state estimation is to characterize the a-posteriori mean by a system of stochastic differential equations (SDEs) called the {optimal nonlinear filter (state estimation) equations} or {optimal nonlinear smoother (state estimation) equations}, depending on whether we condition up to the current observation or the whole observational period (in the case where an observed time series exists a-priori), respectively. In general, without any special assumptions on the structure of the processes $\vh$ and $\vx$,  $\ee{\vh\big|\cF_t^\vx}$ is difficult to determine. However, under the assumption that their components are of a specific type, then we can characterize the a-posteriori mean through a system of analytic equations. Specifically, in the most general setting, such a system of analytic equations can be recovered under certain regularity assumptions, which include the $\vh=(\vh(t,\vxt,\vyt),\cF_t)$ process, $t\leq T$, being represented as
\begin{equation} \label{eq:assumptionh}
    \vh(t,\vxt,\vyt)=\mathbf{h}(0,\vx(0),\vy(0))+\int_0^t\boldsymbol{\xi}(s,\vx(s),\vy(s))\rmd s+\vz(t,\vxt,\vyt),
\end{equation}
where $\vz=(\vz(t,\vxt,\vyt),\cF_t)$ is a martingale and $\boldsymbol{\xi}=(\boldsymbol{\xi}(t,\vxt,\vyt),\cF_t)$ is a nonanticipative process such that it is $\pp$-a.s.\ (almost surely) in an appropriate $L^2$ Bochner space, i.e., $\pp\left(\int_0^T\left\|\boldsymbol{\xi}(s,\vx(s),\vy(s))\right\|_2\rmd s<+\infty\right)=1$, while the observable process is assumed to be an Itô process of the diffusive type, i.e.,
\begin{equation} \label{eq:assumptionx}
    \vxt=\mathbf{x}(0)+\int_0^t\mathbf{a}(s,\vx(s))\rmd s+\int_0^t\mathbf{B}(s,\mathbf{x}(s))\rmd \mathbf{W}(s),
\end{equation}
where $\mathbf{W}=(\vwt{},\cF_t)$ is a $d$-dimensional complex-valued Wiener process, and the processes $\mathbf{a}(t,\vxt)\in\mathbb{C}^k$ and $\mathbf{B}(t,\vxt)\in\mathbb{C}^{k\times d}$ are nonanticipative (at least for almost all times $t\in[0,T]$), with the former being (Lebesgue) integrable while the latter is $\mathbf{W}$-integrable, $\pp$-a.s \cite{liptser2001statisticsI}. Unfortunately, under this general formulation, one encounters the inherent difficulty of needing to know the higher-order conditional moments of the desired stochastic functional $\vh$, specifically the ratios between them (in the tesnor sense), as to be able to arrive to a closed system of explicitly solvable equations for the a-posteriori mean (or covariance).

Nonetheless, for a broad class of nonlinear systems, that being the class of CGNSs, the challenge of non-closedness can be effectively addressed to yield a closed system of optimal nonlinear posterior state estimation equations, as discussed in Chapter 12 of Liptser \& Shiryayev \cite{liptser2001statisticsII}. The CGNS framework is particularly powerful in this context, as it provides a closed system of equations for the a-posteriori mean, stemming from the conditional Gaussianity of the posterior distributions. This property allows for the expression of higher-order conditional moments solely in terms of the lower-order ones, specifically through the conditional mean and covariance tensor, thanks to the conditional Gaussian structure.

\subsection{Conditionally Gaussian Nonlinear Systems} \label{sec:2.2}

In its most general form, a conditional Gaussian system of processes consists of two diffusion-type processes defined by the following system of stochastic differentials given in Itô form
\cite{liptser2001statisticsII, chen2018conditional, chen2016filtering2}:
\begin{align} 
    \rmd\vxt &= \left[\ml^\vx(t,\vx)\vyt+\vf^\vx(t,\vx)\right]\rmd t+\ms_1^\vx(t,\vx)\rmd \vwt{1}+\ms_2^\vx(t,\vx)\rmd \vwt{2}, \label{eq:condgauss1}\\
    \rmd\vyt &= \left[\ml^\vy(t,\vx)\vyt+\vf^\vy(t,\vx)\right]\rmd t+\ms_1^\vy(t,\vx)\rmd \vwt{1}+\ms_2^\vy(t,\vx)\rmd \vwt{2}, \label{eq:condgauss2}
\end{align}
where 
\begin{equation*}
    \vw_1=\big((W_{11}(t),\ldots,W_{1d}(t)),\cF_t\big) \quad\text{and}\quad \vw_2=\big((W_{21}(t),\ldots,W_{2r}(t)),\cF_t\big),
\end{equation*}
are two mutually independent complex-valued Wiener processes (i.e., both their real and imaginary parts are mutually independent real-valued Wiener processes) and almost every path of $\vx$ and $\vy$ is in $C^0([0,T];\mathbb{C}^k)$ and $C^0([0,T];\mathbb{C}^l)$, respectively. The elements of the vector- and matrix-valued functions of multiplicative factors ($\ml^\vx,\ml^\vy$), forcings ($\vf^\vx,\vf^\vy$), and noise feedbacks ($\ms_1^\vx,\ms_2^\vx,\ms_1^\vy,\ms_2^\vy$) are assumed to be nonanticipative (adapted) functionals over the measurable time-function cylinder
\begin{equation*}
    (C^{0,k}_T,\mathscr{B}^k_T):=\big([0,T]\times C^0([0,T];\mathbb{C}^k), \mathcal{B}([0,T])\otimes
 \mathcal{B}(C^0([0,T];\mathbb{C}^k))\big),
\end{equation*}
where $\otimes$ denotes the tensor-product $\sigma$-algebra on the underlying product space, i.e.,
\begin{equation*}
    \mathscr{B}^k_T=\mathcal{B}([0,T])\otimes\mathcal{B}(C^0([0,T];\mathbb{C}^k))=\sigma\left(\big\{A\times B: A\in\mathcal{B}([0,T]), B\in\mathcal{B}(C^0([0,T];\mathbb{C}^k))\big\}\right),
\end{equation*}
with $\mathcal{B}(C^0([0,T];\mathbb{C}^k))$ being a $\sigma$-algebra generated by some topology on the space of continuous functions from $[0,T]$ to $\mathbb{C}^k$, $C^0([0,T];\mathbb{C}^k)$ (e.g., the Isbell topology or the $\sigma$-algebra generated by equipping the space of continuous functions with the compact-open topology, which simplifies to the topology of compact convergence, thus making the space Polish and reducing $\mathcal{B}(C^0([0,T];\mathbb{C}^k))$ to the $\sigma$-algebra of Borelian subsets \cite{dimaio1998topologies, escardo2001topologies}). This means that if $g(t,\vx)$ denotes a standard element of these functionals, for each $t\in[0,T]$, it is $(\cF_t^\vx;\mathbb{C})$-measurable. Observe how this is a weaker assumption imposed on the model parameters compared to progressive joint measurability, i.e., where each one of their standard elements is instead $(\mathcal{B}([0,t])\otimes\cF_t^\vx;\mathbb{C})$-measurable for every $t\in[0,T]$ \cite{karatzas2014brownian}. It is important to emphasize here the fact that in a CGNS, the unobservable component $\vy$ enters into \eqref{eq:condgauss1}--\eqref{eq:condgauss2} in a purely linear manner, whereas the observable process $\vx$ can enter into the coefficients of both equations in any $\cF^\vx_t$-measurable way, by allowing the model components $\ml^\vx$, $\ml^\vy$, $\vf^\vx$, $\vf^\vy$, $\ms^\vx_1$, $\ms^\vx_2$, $\ms^\vy_1$, and $\ms^\vy_2$ to be nonlinear functionals of the known components of the state variable.

\subsection{Significance and Broad Applications of CGNSs} \label{sec:2.3}
A variety of CTNDSs can be represented within the CGNS framework. Notable examples include physics-constrained nonlinear stochastic models, such as noisy Lorenz models \cite{lorenz1963deterministic, lorenz1984formulation}, low-order representations of Charney-DeVore flows \cite{olbers2001gallery}, and models capturing topographic mean flow interactions \cite{harlim2014ensemble, ferrari2003seasonal}. Additionally, stochastically coupled reaction-diffusion models are prevalent in fields like neuroscience and ecology, including models like the stochastically coupled FitzHugh-Nagumo \cite{treutlein1985noise, lindner2004effects, treutlein1985noise, lindner2004effects}, stochastically forced predator-prey models \cite{sun2010rich, sadhu2018stochastic}, and SIR epidemic models \cite{gray2011stochastic, kim2013sir}. Furthermore, multiscale models for geophysical flows, such as the Boussinesq equations with added noise \cite{majda2003introduction, castaing1989scaling, kelliher2011boundary} and the stochastically forced rotating shallow water equation \cite{chen2023uncertainty}, also fall within this framework \cite{chen2018conditional}. This modeling approach has been applied to describe many natural phenomena as well, such as the Madden-Julian oscillation and Arctic sea ice dynamics \cite{chen2014predicting, chen2022efficient}.

In addition to modeling a plethora of complex phenomena, the CGNS framework and its closed analytic DA formulae have been utilized to answer many theoretical questions and design computational strategies. Specifically, the classical Kalman-Bucy filter \cite{kalman1961new} serves as a fundamental special case of a CGNS DA scheme. But beyond this, the CGNSs has been utilized to develop an exact nonlinear Lagrangian DA algorithm that advances rigorous analysis to understand model errors and uncertainties in state recovery \cite{chen2014information, chen2015noisy, chen2016model}. This includes the usage of the framework not just for high-dimensional state estimation of ocean flows, but also in determining the optimal launching locations for Lagrangian tracers under uncertainty \cite{chen2024launching, bolt2024causation}.  Furthermore, the analytically solvable DA schemes have been applied for state estimation and the prediction of intermittent time series associated with monsoons and other climatic events \cite{chen2016filtering2, chen2018predicting}. Moreover, the effective DA procedures have contributed to developing rapid algorithms for addressing high-dimensional Fokker-Planck equations \cite{chen2017beating, chen2018efficient}.

Importantly, the underlying principles of designing CGNSs have been extended to much broader applications. Examples include the development of forecasting models within the realm of dynamic stochastic superresolution \cite{branicki2013dynamic, keating2012new}, the creation of stochastic superparameterizations for geophysical turbulence \cite{majda2009mathematical, grooms2014stochastic, majda2014new}, and the design of efficient multiscale DA strategies \cite{majda2014blended, deng2024lemda}. It has also been used for the development of efficient data-driven multiscale reduced-order, as well as deep learning,
modeling frameworks for CTNDSs \cite{mou2023efficient, chen2024cgkn}.

The CGNS framework differs significantly from other closure modeling approaches, such as linear regression models \cite{freedman2009statistical, yan2009linear} and projection methods like the Galerkin proper orthogonal decomposition \cite{holmes2012turbulence, hasselmann1988pips, kwasniok1996reduction}. These methods typically create surrogate models that are more computationally efficient than the original, addressing truncation errors through closure terms. In contrast, closure models for the large-scale variable $\vx$ do not incorporate the small-scale variables $\vy$, relying instead on "diagnostic'' terms that parameterize the nonlinear coupling between slow and fast variables. These terms often involve past values of $\vx$ and are derived from the Mori-Zwanzig (MZ) formalism from statistical mechanics \cite{mori1965transport, zwanzig2001nonequilibrium, chorin2002optimal, chorin2009stochastic}. The CGNS, however, employs a simplified prognostic equation for the unresolved small-scale variables $\vy$, allowing them to influence the dynamics only linearly, while permitting nonlinear dependencies in the coefficients of drift and diffusion.

A treatise on the broad applications of the CGNS framework in the prediction, state estimation, and uncertainty quantification of multiscale nonlinear stochastic systems is provided in Chen and Majda \cite{chen2018conditional}.

\subsection{Preliminaries} \label{sec:2.4}
Let the following denote the standard elements of the vector- and matrix-valued functionals appearing in \eqref{eq:condgauss1}--\eqref{eq:condgauss2} as model parameters:
\begin{gather*}
        \vf^\vx(t,\vx):=(f^\vx_1(t,x),\ldots,f^\vx_k(t,x))^\mathtt{T},\quad\vf^\vy(t,\vx):=(f^\vy_1(t,x),\ldots,f^\vy_l(t,x))^\mathtt{T},\\
        \ml^\vx(t,\vx):=\left(\Lambda^\vx_{ij}(t,\vx)\right)_{k\times l},\quad \ms_1^\vx(t,\vx):=\left(\Sigma^{\vx,1}_{ij}(t,\vx)\right)_{k\times d},\quad \ms_2^\vx(t,\vx):=\left(\Sigma^{\vx,2}_{ij}(t,\vx)\right)_{k\times r},\\
    \ml^\vy(t,\vx):=\left(\Lambda^\vy_{ij}(t,\vx)\right)_{l\times l},\quad\ms_1^\vy(t,\vx):=\left(\Sigma^{\vy,1}_{ij}(t,\vx)\right)_{l\times d},\quad\ms_2^\vy(t,\vx):=\left(\Sigma^{\vy,2}_{ij}(t,\vx)\right)_{l\times r}.
\end{gather*}
As to be able to obtain the main results which define the CGNS framework and its potency in DA through closed-form expressions for the posterior statistics, a set of sufficient regularity conditions needs to be assumed a-priori. For the following conditions, each respective pair of indices $i$ and $j$ takes all admissible values and $\vz$ is a $k$-dimensional function in $C^0([0,T];\mathbb{C}^k)$:
\begin{enumerate}[label=\magenta{\textbf{(\arabic*)}}]
    \item The multiplicative factor matrix and forcing vector in the unobservable process are Lebesgue integrable, while all noise feedback matrices and multiplicative factor matrix and forcing vector in the observable process are $\vw$-integrable (in the Itô sense), or equivalently, by Itô isometry,
    \begin{align*}
        \int_0^T \Big[&|f^\vy_i(t,\vz)|+|\Lambda^\vy_{ij}(t,\vz)|+\big| f^\vx_i(t,\vz)\big|^2+\big|\Lambda^\vx_{ij}(t,\vz)\big|^2 \\
        & +\big|\Sigma^{\vx,1}_{ij}(t,\vz)\big|^2+\big|\Sigma^{\vx,2}_{ij}(t,\vz)\big|^2+\big|\Sigma^{\vy,1}_{ij}(t,\vz)\big|^2+\big|\Sigma^{\vy,2}_{ij}(t,\vz)\big|^2\Big]\rmd t<+\infty.
    \end{align*}
    This ensures the existence of the integrals in \eqref{eq:condgauss1}--\eqref{eq:condgauss2} \cite{liptser2001statisticsI, liptser2001statisticsII}.

    \item $\displaystyle |\Lambda^\vx_{ij}(t,\vz)|,|\Lambda^\vy_{ij}(t,\vz)|\leq L_1$ for some (uniform) $L_1>0$, $\forall t\in[0,T]$.
    
    \item If $g(t,\vz)$ denotes any element of the multiplicative factors in the drift dynamics, $\ml^\vx$ and $\ml^\vy$, or of the noise feedback matrices, $\ms_m^\vx(t,\vz)$ and $\ms_m^\vy(t,\vz)$ for $m=1,2$, and $K(s)$ is a nondecreasing right-continuous function taking values in $[0,1]$, then there $\exists L_2,L_3,L_4,L_5>0$ such that for any $\vz$ and $\mathbf{w}$ $k$-dimensional functions in $C^0([0,T];\mathbb{C}^k)$ we have
    \begin{gather*}
        |g(t,\vz)-g(t,\mathbf{w})|^2\leq L_2\int_0^t \left\|\vz(s)-\mathbf{w}(s)\right\|_2^2\rmd K(s)+L_3\left\|\vzt-\mathbf{w}(t)\right\|_2^2, \ \forall t\in[0,T],\\
        |g(t,\vz)|^2\leq L_4\int_0^t (1+\left\|\vz(s)\right\|_2^2)\rmd K(s)+L_5(1+\left\|\vzt\right\|_2^2), \ \forall t\in[0,T].
    \end{gather*}
    By the properties of $K$, the integrals are to be understood in the Lebesgue–Stieltjes integration sense. The first inequality acts as a generalized or weak global Lipschitz condition, while the second one establishes an at most linear growth condition, both in terms of the spatial component (which can be weakened to local boundedness or local integrability of the drift components and to Sobolev diffusion elements for existence of a unique local strong solution \cite{zhang2005strong, veretennikov1981strong}). 

    \item $\ee{\|\vx(0)\|_2^{2}+\|\vy(0)\|_2^{2}}<+\infty$. This also implies $\pp$-a.s.\ finiteness of the initial conditions by Markov's inequality.

    \begin{itemize}
        \item[$\blacktriangleright$] Assumptions \textbf{\magenta{(1)}}--\textbf{\magenta{(4)}} are sufficient in establishing the existence and uniqueness (both in the strong (pathwise) and weak (in law) sense) of a strong solution to the CGNS of SDEs that is continuous in $t$, continuously depends on the initial distributions and additional model parameters, and satisfies $\underset{t\in[0,T]}{\sup}\left\{\ee{\|\vx(t)\|_2^{2}+\|\vy(t)\|_2^{2}}\right\}<+\infty$ \cite{arnold1974stochastic, mao2008stochastic, stroock1997multidimensional, evans2012introduction, oksendal2003stochastic}. A proof of this specific fact is given in Theorem 12.4 of Liptser \& Shiryaev \cite{liptser2001statisticsII} for when $k=l=1$, which can be easily adapted to the multi-dimensional case for arbitrary $k,l\in\mathbb{N}$ (see also \cite{kolodziej1980state}). These results also extend to the existence and uniqueness of solutions, that are also continuous in $t$ and continuously depend on the initial data and additional model parameters, to the optimal nonlinear filter and smoother state estimation equations, as well as associated sampling SDEs, which are provided in Sections \ref{sec:2.5} and \ref{sec:3.1}--\ref{sec:3.2}, respectively (see note accompanying assumptions \textbf{\magenta{(5)}}--\textbf{\magenta{(8)}}). Note that the globality of the Lipschitz conditions in \textbf{\magenta{(3)}} is unavoidable, since from classical SDE theory, while local Lipschitzness (in space, uniformly in time) guarantees strong uniqueness, it can only ensure existence up to the "blow-up time" of the local solution \cite{karatzas2014brownian}. Notably, in the case where \textbf{\magenta{(3)}} holds solely for the observable diffusion components $\ms^\vx_m$, with $m=1,2$, then this is sufficient for the existence and uniqueness of continuous solutions to the optimal nonlinear filter and smoother state estimation equations (and associated sampling formulae), granted that there preexists a strong local solution to the CGNS of equations, as it leads to the unique strong solution of an auxiliary SDE which is necessary in establishing a certain equivalence of probability measures and to give a well-defined Bayes formula \cite{kolodziej1980state}.
    \end{itemize}

    \item The sum of the Gramians (with respect to rows) of the noise coefficient matrices in the observable process are uniformly nonsingular, i.e., the elements of the inverse of 
    \begin{equation*}
        (\ms^\vx\circ\ms^\vx)(t,\vz)=\ms_1^\vx(t,\vz)\ms_1^\vx(t,\vz)^\dagger+\ms_2^\vx(t,\vz)\ms_2^\vx(t,\vz)^\dagger,
    \end{equation*}
    are uniformly bounded in $[0,T]$, where $\boldsymbol{\cdot}^\dagger$ denotes the Hermitian transpose operator. This ensures the nondegeneracy of stochastic measures associated with $\vx$, as well as the invertibility of the observable Gramian over time.

    \item $\displaystyle \int_0^T \ee{|\Lambda^\vx_{ij}(t,\vzt)y_j(t)|}\rmd t< +\infty$.

    \item $\displaystyle \ee{|y_j(t)|}<+\infty, \ t\in[0,T]$.

    \item For $\mu_{\text{f},j}(t):=\ee{y_j(t)\big|\cF_t^\vx}$, where $t\in[0,T]$ and $j=1,\ldots,l$, we assume that for all $i=1,\ldots,k$ and $j=1,\ldots,l$ we have:
    \begin{equation*}
        \pp\left(\int_0^T\left|\Lambda^\vx_{ij}(t,\vzt)\mu_{\text{f},j}(t)\right|^2\rmd t<+\infty\right)=1.
    \end{equation*}

    \begin{itemize}
        \item[$\blacktriangleright$] While assumptions \textbf{\magenta{(1)}}--\textbf{\magenta{(4)}} are there to guarantee existence and uniqueness of a continuous strong solution to the CGNS of SDEs, the additional assumptions \textbf{\magenta{(5)}}--\textbf{\magenta{(8)}} extend these properties to solutions of the optimal nonlinear posterior state estimation equations, and combined with the previous conditions they are additionally sufficient in achieving the stochastic stability (e.g., in the mean-square sense) \cite{saito1996stability, higham2000stability, higham2000mean, higham2003exponential, huang2012exponential, tocino2012mean}, as well as the strong and weak consistency of temporally discretized numerical integration schemes applied to \eqref{eq:condgauss1}--\eqref{eq:condgauss2}, like those of Euler-Maruyama (EM) and Milstein \cite{kloeden1992numerical, sarkka2019applied}, for when the limit of the maximum width of the time subintervals goes to zero. These conditions are known to be equivalent to strong and weak convergence, respectively, through an appropriate stochastic version of the Lax equivalence theorem \cite{kloeden1992numerical, lang2010lax}. 
    \end{itemize}
\end{enumerate}
We note here that the conditions \textbf{\magenta{(1)}}--\textbf{\magenta{(8)}} assumed here, are only sufficient, with the possibility of some of them being slightly weakened, as already discussed (see the works of Kolodziej et al. for the CGNS framework specifically \cite{kolodziej1980state, kolodziej1986state}).

Contingent on this set of sufficient assumptions, it is then possible to show that the posterior distribution of the unobserved variables, when conditioning on the observational data, is Gaussian. This is exactly the reason why the system in \eqref{eq:condgauss1}--\eqref{eq:condgauss2} is called a conditionally Gaussian dynamical system of equations. For the remainder of this work, as to make our notation more explicit, we write $\big(\boldsymbol{\cdot}\big|\vx(s),s\leq t\big)$ to indicate the fact that we are conditioning on the observable $\sigma$-algebra at time $t$, $\big(\boldsymbol{\cdot}\big|\cF_t^\vx\big)$ (and likewise for when conditioning on other (sub-)$\sigma$-algebras, unless otherwise noted). Strictly speaking this is an abuse of notation since it is not correct to condition on an explicit path but instead the $\sigma$-algebra generated by the state variable, specifically the observable filtration in this case \cite{oksendal2003stochastic}. Nonetheless, this abuse of notation is adopted for notational simplicity and didactic reasons.

\begin{thm}[\textbf{Conditional Gaussianity}] \label{thm:condgaussianity}
    Let $\vxt$ and $\vyt$ satisfy \eqref{eq:condgauss1}--\eqref{eq:condgauss2} and assume that the regularity conditions \textbf{\magenta{(1)}}--\textbf{\magenta{(8)}} hold. Additionally, assume that the initial conditional distribution $\pp\big(\vy(0)\leq \boldsymbol{\alpha}_0\big|\vx(0)\big)$\footnote{The event $\{\vy(s)\leq \boldsymbol{\alpha}=(\alpha_{1},\ldots,\alpha_{l})^\tran\}$ is to be understood coordinate-wise, i.e., $\{y_1(s)\leq \alpha_{1},\ldots,y_l(s)\leq \alpha_{l}\}$, with the linear partial ordering assumed on $\mathbb{C}$, i.e., $y_j(s)\leq \alpha_{j}\Leftrightarrow \big(\mathrm{Re}(y_j(s))<\mathrm{Re}(\alpha_{j})\big) \vee \big(\mathrm{Re}(y_j(s))=\mathrm{Re}(\alpha_{j}) \wedge \mathrm{Im}(y_j(s))\leq\mathrm{Im}(\alpha_{j})\big)$. (Note that this ordering is still not compatible with multiplication, since $\mathbb{C}$ is algebraically closed thus prohibiting it from being an ordered field.) This ordering is adopted for the remainder of this work.} be $\pp$-a.s.\ Gaussian, $\mathcal{N}_l(\vm{\nf}(0),\mr{\nf}(0))$, and mutually independent from the Wiener processes $\vw_1$ and $\vw_2$, where
    \begin{equation*}
        \vm{\nf}(0):=\ee{\vy(0)\big|\vx(0)} \quad\text{and}\quad \mr{\nf}(0):=\ee{(\vy(0)-\vm{\nf}(0))(\vy(0)-\vm{\nf}(0))^\dagger\big|\vx(0)}.
    \end{equation*}
    Furthermore, assume $\pp\left(\mathrm{tr}(\mr{\nf}(0))<+\infty\right)=1$, where $\mathrm{tr}(\boldsymbol{\cdot})$ denotes the trace operator, meaning the initial estimation mean-square error between $\vy(0)$ and $\vm{\nf}(0)$ is almost surely finite. Then, for any set of $\{t_j\}_{1\leq j\leq n}$ such that $0\leq t_1 < t_2 < \cdots < t_n\leq t$, with $t\in[0,T]$, and $\va_1,\ldots,\va_n\in\mathbb{C}^l$, the conditional distribution
    \begin{equation*}
        \pp\big(\vy(t_1)\leq \va_1,\ldots,\vy(t_n)\leq \va_n\big|\vx(s),s\leq t\big),
    \end{equation*}
    is $\pp$-a.s.\ Gaussian.
\end{thm}
\begin{proof}[Proof of Theorem \ref{thm:condgaussianity}]
    This is Theorem 12.6 in Liptser \& Shiryaev \cite{liptser2001statisticsII}, which is the multi-dimensional analog of Theorem 11.1. For the analogous result in the case of discrete time, see Theorem 13.3, with the respective sufficient assumptions given in Subchapter 13.2.1. Thorough details are also provided in Kolodziej \cite{kolodziej1980state} for the continuous-time case. The proof, regardless of continuous- or discrete-time, uses the conditional characteristic function method and a conditional version of the law-uniqueness theorem \cite{kolodziej1980state, yuan2016some}, and is independent of the discretization-based martingale-free proof framework adopted in this work.
\end{proof}

\subsection{Analytically Solvable Filter and Smoother Posterior Distributions} \label{sec:2.5}

Before proceeding, we first make some notational clarifications. The big-O notation has its usual meaning of describing the limiting behavior of a scalar, vector, or matrix expression. Specifically, if $\mathbf{f}$ is either a scalar-, vector-, or matrix-valued function with $\|\boldsymbol{\cdot}\|$ being an appropriate norm over its range, then for $g$ being a strictly positive scalar function over the positive reals, we have that
\begin{equation*}
    \mathbf{f}(\dt)=O(g(\dt)) \ (\text{as } \dt\to0^+) \overset{\text{def}}{\Longleftrightarrow} \limsup_{\dt\to0^+}\frac{\|\mathbf{f}(\dt)\|}{g(\dt)}<+\infty,
\end{equation*}
where $\dt$, in this work, exclusively describes the mesh or norm (length of the longest subinterval) of a partition over the time interval of interest. Additionally, if $\vz\in\mathbb{C}^{k}$ and $\mathbf{w}\in\mathbb{C}^{l}$ are two random vectors, then $\mathrm{Cov}(\vz,\mathbf{w})=\overline{\vz^\prime(\mathbf{w}^\prime)^\dagger}\in\mathbb{C}^{k\times l}$ is their cross-covariance matrix, while $\mathrm{Var}(\vz,\mathbf{w})\in\mathbb{C}^{(k+l)\times (k+l)}$ is their covariance matrix, which when decomposed into a $2\times 2$ block matrix, its $(1,2)$ block is exactly $\mathrm{Cov}(\vz,\mathbf{w})$, while its $(2,1)$ block is $\mathrm{Cov}(\vz,\mathbf{w})^\dagger=\mathrm{Cov}(\mathbf{w},\vz)\in\mathbb{C}^{l\times k}$. Essentially, with this notation, we have that $\mathrm{Var}(\vz,\mathbf{w})\equiv\mathrm{Cov}((\vz^\tran,\mathbf{w}^\tran)^\tran,(\vz^\tran,\mathbf{w}^\tran)^\tran)$. Finally, for $A(\lambda)$ and $B(\lambda)$ being two families of objects parameterized by $\lambda\in \Lambda$, then $A\lesssim B$ is equivalent to there $\exists C>0$ such that $\forall\lambda\in\Lambda$ we have $A(\lambda)\leq CB(\lambda)$. This aids in simplifying our notation and relieving from the need of keeping track of constant multiplicative factors in inequalities.

In what follows, we establish the time-discretization setup which is utilized throughout this work. Specifically, we define the EM discretization of the CGNS in \eqref{eq:condgauss1}--\eqref{eq:condgauss2} and the associated continuous-time extended EM discrete approximation which recovers the strong solution to the CGNS under the expected mean-square metric up to first-order, i.e., $O(\dt)$ (see foregoing discussion). We first start by "compactifying" the CGNS in \eqref{eq:condgauss1}--\eqref{eq:condgauss2}, which we represent via the formal integral equation
\begin{equation} \label{eq:compactGGNS}
    \mathbf{v}(t)=\mathbf{v}(0)+\int_0^t\vf(s,\mathbf{v}(s))\rmd s+\int_0^t\ms(s,\mathbf{v}(s))\rmd \mathbf{W}(s),
\end{equation}
for
\begin{gather*}
    \mathbf{v}(t):=\begin{pmatrix}
        \vxt\\\vyt
    \end{pmatrix}, \quad \vw(t):=\begin{pmatrix}
        \vw_1(t)\\ \vw_2(t)
    \end{pmatrix}, \quad  \vf(t,\mathbf{v}):=\begin{pmatrix}
            \mathbf{0}_{k\times k} & \ml^\vx(t,\vx) \\
            \mathbf{0}_{l\times k} & \ml^\vy(t,\vx)
        \end{pmatrix}\mathbf{v}+\begin{pmatrix}
            \vf^\vx(t,\vx)\\\vf^\vy(t,\vx)
        \end{pmatrix}, \\ 
        \ms(t,\mathbf{v}):=\begin{pmatrix}
            \ms_1^\vx(t,\vx) & \ms_2^\vx(t,\vx)\\
            \ms_2^\vy(t,\vx) & \ms_2^\vy(t,\vx)
        \end{pmatrix}.
\end{gather*}
As aforementioned, the discrete-time approximation of $\mathbf{v}$ adopted in this work, which is coincidentally the simplest one used for Monte Carlo approximation of \eqref{eq:compactGGNS}, and also recovers its strong solution in the limit, is given by the explicit EM scheme that is defined as the discretized process
\begin{equation} \label{eq:emdiscrete}
    \hat{\mathbf{v}}_{\dt}^{j+1}=\hat{\mathbf{v}}_{\dt}^{j}+\vf(t_j,\hat{\mathbf{v}}_{\dt}^{j})(t_{j+1}-t_j)+\ms(t_j,\hat{\mathbf{v}}_{\dt}^{j})\sqrt{t_{j+1}-t_j}\ve^j, \ j=0,1,\ldots,J-1, \quad \hat{\mathbf{v}}_{\dt}^{0}\overset{\rmd}{=}\mathbf{v}(0),
\end{equation}
where $\{t_j\}_{0\leq j\leq J}$ is a partition of $[0,T]$ with $t_0=0$, $t_n=T$, and norm or mesh given by $\dt=\underset{0\leq j\leq J-1}{\max}\{\dt_j\}$ for $\dt_j=t_{j+1}-t_j$, and $\ve^j$ are $(k+l)$-dimensional complex standard normal vectors that are mutually independent from each other, as well as from the initial distribution of the state variables, $\mathbf{v}(0)$. For simplicity, we from now on assume that this temporal partition is uniform, i.e., it is an equidistant time discretization with a uniform time subinterval length of $\dt_j\equiv \dt=T/J$, $\forall j$. This implies that $t_j=j\dt$ for $j=0,1,\ldots,J$. With this, we now define the continuous-time extension of the discrete approximation $\{\hat{\mathbf{v}}_{\dt}^j\}_{0\leq j\leq J}$ by the following piecewise definition:
\begin{equation} \label{eq:continuoustimeext1}
    \mathbf{v}_{\dt}(t):= \hat{\mathbf{v}}_{\dt}^{j}+\vf(t_j,\hat{\mathbf{v}}_{\dt}^{j})(t-t_j)+\ms(t_j,\hat{\mathbf{v}}_{\dt}^{j})(\vw(t)-\vw(t_j)) \text{ for } t\in[t_j,t_{j+1}),
\end{equation}
with $j=0,1,\ldots,J-1$, or equivalently as
\begin{equation} \label{eq:continuoustimeext2}
    \mathbf{v}_{\dt}(t):=\hat{\mathbf{v}}_{\dt}^{0}+\int_0^t\vf(\eta_{\dt}(s),\bar{\mathbf{v}}_{\dt}(s))\rmd s+\int_0^t\ms(\eta_{\dt}(s),\bar{\mathbf{v}}_{\dt}(s))\rmd \mathbf{W}(s),
\end{equation}
where $\bar{\mathbf{v}}_{\dt}(s)$ is defined as the piecewise constant process $\bar{\mathbf{v}}_{\dt}(t):=\hat{\mathbf{v}}_{\dt}^{j}$, for $t\in[t_j,t_{j+1})$, $\forall j=0,1,\ldots,J-1$, and likewise $\eta_{\dt}(t):=t_j=j\dt=jT/J$ for $t\in[t_j,t_{j+1})$, $\forall j=0,1,\ldots,J-1$ (in this sense we can think of this partition as being tagged by the left end-points $t_j$). Note how $\mathbf{v}_{\dt}(t)$ and $\bar{\mathbf{v}}_{\dt}(t)$ coincide with the discrete solution at the gridpoints, i.e., $\mathbf{v}_{\dt}(t_j)=\bar{\mathbf{v}}_{\dt}(t_j)=\hat{\mathbf{v}}_{\dt}^{j}$, $\forall j=0,1,\ldots,J$. As discussed, it is then immediate that this process recovers the strong solution to \eqref{eq:compactGGNS} (i.e., \eqref{eq:condgauss1}--\eqref{eq:condgauss2}) through the estimate
\begin{equation} \label{eq:emconvergence}
    \ee{\underset{t\in[0,T]}{\sup}\{\left\|\mathbf{v}(t)-\mathbf{v}_{\dt}(t)\right\|_2^2\}}\lesssim O(\dt),
\end{equation}
for $\dt$ small. Based then on the previous discussion concerning the consistency, stability, and convergence of the explicit EM scheme defined by \eqref{eq:emdiscrete}--\eqref{eq:continuoustimeext2}, and through \eqref{eq:emconvergence}, we henceforth associate the continuous-time extension of the discrete EM process with the strong solution of the CGNS, as a way to also further clarify the subsequent notation. As such, in the sequel, we have $\vx\approx\vx_{\dt}$ and $\vy\approx\vy_{\dt}$, as defined in \eqref{eq:continuoustimeext1} or \eqref{eq:continuoustimeext2}, with this association being true or exact in the limit as $\dt\to0^+$ (via \eqref{eq:emconvergence}). This association is explicitly made through $\vx^j:=\vx(t_j)\approx\vx_{\dt}(t_j)=\hat{\vx}_{\dt}^j$ and $\vy^j:=\vy(t_j)\approx\vy_{\dt}(t_j)=\hat{\vy}_{\dt}^j$, where generally in this work the superscript $\boldsymbol{\cdot}^{\,j}$ denotes the discrete approximation of the continuous form of the respective variable when evaluated on $t_j$ under this regime, which through \eqref{eq:emconvergence} and the continuity of the model parameters in \eqref{eq:condgauss1}--\eqref{eq:condgauss2} provided by assumption \textbf{\magenta{(3)}} and assumption \textbf{\magenta{(9)}}, which is outlined in what follows, likewise recovers the continuous-time functional in the limit. For example, $\ml^{\vx,j}\equiv \ml^\vx(t_j,\vx(t_j))$, and to make things concrete in terms of the convergence, by using $\vx$ as an example, we have from \eqref{eq:emconvergence} that 
\begin{equation*}
   \lim_{\dt\to0^+} \vx_{\dt}(t)=\vx(t), \ \forall t\in[0,T]\Leftrightarrow \lim_{\substack{\dt\to0^+\\ J'\dt=T'}} \vx_{\dt}^{J'}=\vx(T'), \ \forall T'\in[0,T],
\end{equation*}
where for the remainder of this work the limits being taken as $\dt\to0^+$ are to be understood in this sense. As such, under this framework, we discretize \eqref{eq:condgauss1}--\eqref{eq:condgauss2} as (note the form of \eqref{eq:emdiscrete}):
\begin{align}
    \vx^{j+1} &= \vx^{j}+\left(\ml^{\vx,j}\vy^j+\vf^{\vx,j}\right)\dt+\ms_1^{\vx,j}\sqrt{\dt} \ve_1^j+\ms_2^{\vx,j}\sqrt{\dt} \ve_2^j, \label{eq:discretecondgauss1}\\
    \vy^{j+1} &= \vy^{j}+\left(\ml^{\vy,j}\vy^j+\vf^{\vy,j}\right)\dt+\ms_1^{\vy,j}\sqrt{\dt} \ve_1^j+\ms_2^{\vy,j}\sqrt{\dt} \ve_2^j, \label{eq:discretecondgauss2}
\end{align}
where $\ve_1^j$ and $\ve_2^j$ are mutually independent complex standard Gaussian random noises, and $\dt$ is assumed to be sufficiently small (or equivalently $J$ is sufficiently large). 

This notational setup and formulation of the continuous-time extension of the EM discretized solution to the CGNS are required to understand the proofs to the subsequent results presented this work, and which can be found in the appropriate appendices.

Accompanying the conditional Gaussianity provided by Theorem \ref{thm:condgaussianity}, we further assume the following additional conditions (which concur with assumptions \textbf{\magenta{(1)}}--\textbf{\magenta{(8)}} from before):
\begin{itemize}

\item[$\blacktriangleright$] Since the focus of this work is to develop a martingale-free framework for the optimal nonlinear posterior state estimation and sampling equations through the EM discretization scheme described by \eqref{eq:compactGGNS}--\eqref{eq:discretecondgauss2}, we point to the following works (and references therein), which analyze and prove the uniform or final time instant mean-square or mean-absolute error, and pathwise convergence of the explicit EM method, both in the strong and weak sense, with the strong order of convergence being $p=1/2$, and with the weak one being $p=1$ \cite{kloeden1992higher, higham2002strong, hutzenthaler2012convergence, gyongy1998note, yuan2008note, wang2016numerical}. These results are generally established under conditions equivalent to \textbf{\magenta{(1)}}--\textbf{\magenta{(8)}} assumed here (or weaker), along with the additional assumptions:
        \begin{gather*}
           \text{\textbf{\magenta{(9)}}} \quad \left\|\mathbf{f}(t,\vz,\mathbf{u})-\mathbf{f}(s,\vz,\mathbf{u})\right\|_2+\left\|\ms(t,\vz)-\ms(s,\vz)\right\|_2\leq L_6\big(1+(\left\|\vz\right\|^2_2+\left\|\mathbf{u}\right\|^2_2)^{1/2}\big)|t-s|^{1/2},\\
            \text{\textbf{\magenta{(10)}}} \quad \ee{\left\|
            \vx(0)-\vx_{\dt}(0)\right\|^2_2+\left\|
            \vy(0)-\vy_{\dt}(0)\right\|^2_2}\leq L_7\dt,
        \end{gather*}
        where $\vz\in C^0([0,T];\mathbb{C}^k)$, $\mathbf{u}\in C^0([0,T];\mathbb{C}^l)$, and $\mathbf{f}:=\begin{pmatrix}
            \mathbf{0}_{k\times k} & \ml^\vx \\
            \mathbf{0}_{l\times k} & \ml^\vy
        \end{pmatrix}\begin{pmatrix}
            \vx\\\vy
        \end{pmatrix}+\begin{pmatrix}
            \vf^\vx\\\vf^\vy
        \end{pmatrix}$, $\ms:=\begin{pmatrix}
            \ms_1^\vx & \ms_2^\vx\\
            \ms_2^\vy & \ms_2^\vy
        \end{pmatrix}$, for $L_6,L_7>0$, with $\begin{pmatrix}
            \vx_{\dt}(t)\\\vy_{\dt}(t)
        \end{pmatrix}$ being the continuous-time extension of the EM discretization corresponding to \eqref{eq:condgauss1}--\eqref{eq:condgauss2} (see \eqref{eq:continuoustimeext1}--\eqref{eq:continuoustimeext2}). Especially, and most importantly, under the optimality in the mean-square sense that the posterior statistics (filter or smoother) enjoy, the expected uniform (in $[0,T]$) squared error of the EM method (which is stronger to to the uniform or terminal time expected absolute error) is bounded by $O(\dt)$ \cite{kloeden1992numerical, pamen2017strong}, where $\dt$ is the norm or mesh of the temporal partition of $[0,T]$, and under assumptions \textbf{\magenta{(1)}}--\textbf{\magenta{(10)}}, the strong solution of the CGNS provided by \textbf{\magenta{(1)}}--\textbf{\magenta{(4)}} is exactly recovered in the limit as $\dt\to0^+$ \cite{kaneko1988note, mikulevicius1991rate}. This has the important implication that martingale theory can be avoided in proving most results in the continuous-time CGNS framework, since the desired results can be instead established for the discrete-time case using such time discretization schemes, for which then the pertinent results for the case of continuous time can be obtained from these by a formal passage to the limit as the maximal discretization time step vanishes (or the number of time subintervals grows unboundedly) \cite{chen2018efficient}. This provides a unified treatment of the discrete- and continuous-time settings of CGNSs. This is further made possible exactly because the optimal nonlinear posterior state estimation equations, and associated sampling SDEs, are defined by first-order differential operators (infinitesimal generators), which by strong and weak consistency guarantee a first-order coherence between the discrete- and continuous-time counterparts of these dynamical equations. 

    \item[\textbf{\magenta{(11)}}] $\displaystyle \int_0^T \ee{|\ml_{ij}^\vx(t,\vx)|^4+|\vf_i^\vy(t,\vx)|^4+|\Sigma^{\vy,1}_{ij}(t,\vx)|^4+|\Sigma^{\vy,2}_{ij}(t,\vx)|^4}\rmd t<+\infty$. The assumption on the finiteness of the quantity $\int_0^T \ee{|(\vf_i^\vy(t,\vx))|^4}\rmd t$ is also essential in the stochastic optimal control problem based on \eqref{eq:condgauss2}, as it restricts the additive controls to the class of expected $L^4$ functionals, i.e., which have their expectation in the Bochner space $L^4([0,T];\mathbb{C}^p)$ (see Chapter 3 of Kolodziej \cite{kolodziej1980state} and Kolodziej \& Mohler \cite{kolodziej1986state})

    \item[\textbf{\magenta{(12)}}] $\displaystyle \ee{\left\|\vy(0)\right\|_2^4}<+\infty$. This also aids in providing the control $\ee{\underset{t\in[0,T]}{\sup}\{\left\|\vy(t)\right\|_2^4\}}<+\infty$ \cite{kolodziej1980state, liu2019stochastic}.

    \begin{itemize}
        \item [$\blacktriangleright$] Like before, the sufficient assumptions \textbf{\magenta{(11)}} and \textbf{\magenta{(12)}} can be slightly weakened \cite{kolodziej1980state, kolodziej1986state}.
    \end{itemize}
\end{itemize}
With these, we can then yield the optimal nonlinear filter state estimation equations as showcased in the following theorem. In what follows, the subscript "$\,\text{\normalfont{f}}\,$" is used to denote the filter conditional Gaussian statistics, which appropriately stands for filter. The filter conditional Gaussian statistics are also known as the filter posterior mean and filter posterior covariance under the Bayesian inference dynamics framework.

\begin{thm}[\textbf{Optimal Nonlinear Filter State Estimation Equations}]
    \label{thm:filtering}
    Let the assumptions of Theorem \ref{thm:condgaussianity} and the additional regularity conditions \textbf{\magenta{(9)}}--\textbf{\magenta{(12)}}, to hold. Then for any $t\in[0,T]$ the $\cF^\vx_t$-measurable Gaussian statistics of the Gaussian conditional distribution
    \begin{equation*}
        \pp\big(\vyt\big|\vx(s),s\leq t\big) \overset{\rmd}{\sim}\mathcal{N}_l(\vmt{\nf},\mrt{\nf}),
    \end{equation*}
    where
    \begin{equation*}
        \vmt{\nf}:=\ee{\vyt\big|\vx(s),s\leq t} \quad\text{and}\quad \mrt{\nf}:=\ee{(\vyt-\vmt{\nf})(\vyt-\vmt{\nf})^\dagger\big|\vx(s),s\leq t},
    \end{equation*}
    are the unique continuous solutions of the system of optimal nonlinear filter equations,
    \begin{align}
    \rmd \vmt{\nf}&=(\ml^\vy\vm{\nf}+\vf^\vy)\rmd t+(\ms^\vy\circ \ms^\vx+\mr{\nf}(\ml^\vx)^\dagger)(\ms^\vx\circ\ms^\vx)^{-1}(\rmd \vx -(\ml^\vx\vm{\nf}+\vf^\vx)\rmd t), \label{eq:filter1}\\
    \rmd \mrt{\nf}&=\big(\ml^\vy\mr{\nf}+\mr{\nf}(\ml^\vy)^\dagger+\ms^\vy\circ\ms^\vy-(\ms^\vy\circ \ms^\vx+\mr{\nf}(\ml^\vx)^\dagger)(\ms^\vx\circ\ms^\vx)^{-1}(\ms^\vx\circ \ms^\vy+\ml^\vx\mr{\nf})\big)\rmd t\label{eq:filter2},
    \end{align}
    with initial values $\vm{\nf}(0)=\ee{\vy(0)\big|\vx(0)}$ and $\mr{\nf}(0)=\ee{(\vy(0)-\vm{\nf}(0))(\vy(0)-\vm{\nf}(0))^\dagger\big|\vx(0)}$, where the noise interactions through the Gramians (with respect to rows) are defined as
    \begin{gather*}
        (\ms^\vx\circ \ms^\vx)(t,\vx):=\ms_1^\vx(t,\vx)\ms_1^\vx(t,\vx)^\dagger+\ms_2^\vx(t,\vx)\ms_2^\vx(t,\vx)^\dagger, \\
        (\ms^\vy\circ \ms^\vy)(t,\vx):=\ms_1^\vy(t,\vx)\ms_1^\vy(t,\vx)^\dagger+\ms_2^\vy(t,\vx)\ms_2^\vy(t,\vx)^\dagger, \\
        (\ms^\vx\circ \ms^\vy)(t,\vx):=\ms_1^\vx(t,\vx)\ms_1^\vy(t,\vx)^\dagger+\ms_2^\vx(t,\vx)\ms_2^\vy(t,\vx)^\dagger,\quad (\ms^\vy\circ \ms^\vx)(t,\vx):=(\ms^\vx\circ \ms^\vy)(t,\vx)^\dagger.
    \end{gather*}
    Furthermore, if the initial covariance matrix $\mr{\nf}(0)$ is positive-definite ($\pp$-a.s.), then all the matrices $\mrt{\nf}$, for $t\in[0,T]$, remain positive-definite ($\pp$-a.s.).
\end{thm}

See Appendix \ref{sec:Appendix_B} for the proof. We highlight that Theorem \ref{thm:filtering} can also be established under sufficient assumptions that only guarantee the existence of a weak, i.e., martingale, solution to \eqref{eq:condgauss1}--\eqref{eq:condgauss2}, meaning the regularity conditions in \textbf{\magenta{(1)}}--\textbf{\magenta{(12)}} can be substantially relaxed \cite{kolodziej1980state}.

The filter mean equation under the Bayesian inference framework for DA has an intuitive structure. The first two terms on the right-hand side of \eqref{eq:filter1}, specifically $\ml^\vy\vm{\nf}+\vf^\vy$, represent the forecast mean derived from the latent dynamics in \eqref{eq:condgauss2}. The remaining terms account for the correction of the posterior mean state based on partial observations. The matrix-valued factor preceding the innovation or measurement pre-fit residual $\rmd \vxt -(\ml^\vx\vm{\nf}+\vf^\vx)\rmd t$ resembles the Kalman gain in classical Kalman filter theory, expressed as $\mathbf{K} = (\ms^\vy\circ\ms^\vx+\mr{\nf}(\ml^{\vx})^{\dagger}) (\ms^\vx\circ\ms^\vx)^{-1}$, which quantifies the influence of observations on updating the predicted state. Even if the latent dynamics do not explicitly depend on $\vx$, the observational process in \eqref{eq:condgauss1} couples the observed and latent components, allowing observations to adjust the state estimation and correct model forecasts. Additionally, $\vm{\nf}(t)$ satisfies a random ordinary differential equation linear in $\vm{\nf}(t)$ \cite{strand1970random, soong1973random, neckel2013random, han2017random}, while $\mr{\nf}(t)$ is the solution to a random Riccati equation \cite{wang1999stability, casaban2018solving, bishop2019stability}, with coefficients depending on the observable random vector $\vx$.

The CGNS framework also provides closed analytic formulae for optimal smoother state estimation. The smoother posterior distribution at time $t\in[0,T]$ leverages observational information from the entire interval, resulting in a more accurate, less uncertain, and unbiased estimated state compared to the filter solution. To solve the optimal nonlinear smoother equations, a forward pass (filtering) from $t=0$ to $t=T$ is followed by a backward pass (smoothing) from $t=T$ to $t=0$. The forthcoming theorem presents the optimal nonlinear smoother equations, using the subscript $\,\text{\normalfont{s}}\,$ (note the lack of italicization) to denote smoother conditional Gaussian statistics, which should not be confused with the time variable
$s$ in $\vx(s)$ or $\vx^s$. These smoother conditional Gaussian statistics are referred to as the smoother posterior mean and covariance within the Bayesian inference framework. But, for the assertion of this result to be feasible for the case of our martingale-free framework, we enforce a final extraneous assumption on the model parameters of the CGNS, specifically the noise feedbacks of the hidden process:
\begin{itemize}
    \item[\textbf{\magenta{(13)}}] Adopting the temporal discretization setup from the beginning of the section, we assume that
    \begin{equation*}
        a_{\dt}(M):=\underset{\theta>0}{\inf}\Bigg\{\mathbb{E}\Bigg[e^{\theta\big(M\big(1+\underset{0\leq j \leq J}{\max}\big\{\big\|\ms_1^{\vy,j}\big\|_2^2+\big\|\ms_2^{\vy,j}\big\|_2^2\big\}\big)-1/\dt\big)}\Bigg]\Bigg\}=o(\dt),
    \end{equation*}
    as $\dt\to0^+$, for every $M=M(\omega)$ with $\pp(M<+\infty)=1$.
    \begin{itemize}
        \item [$\blacktriangleright$] As illuminated by the proof framework in this work, this is admittedly a rather ad hoc condition. This restriction essentially enforces a temporal control onto the uncertainty levels that should be expected in the latent phase space from a free unconditional run of the CGNS, without being too restrictive (e.g., assuming constant feedbacks). As an example, this condition is trivially satisfied if  $\underset{t\in[0,T]}{\sup}\big\{\big\|\ms^\vy_m(t,\vx)\big\|_2^2\big\}<+\infty$, $\pp$-a.s., for $m=1,2$, which is in many applications a natural condition being satisfied; for example this holds when the noise feedbacks are continuous in time for almost every path of $\vx$ with $T<+\infty$. The use of this assumption will be in developing a consistent optimal discrete-time smoother, and as is showcased in Appendix \ref{sec:Appendix_C} this assumption can be slightly weakened or altered while still being conformant to our proof framework under dominated convergence results. In such a scenario, this assumption can be regarded as a ``Courant-Friedrichs-Lewy-type" convergence restriction enforced onto the time step $\dt$ from an asymptotic expansion-based standpoint. Specifically, by requiring $\displaystyle
        \frac{\partial}{\partial \theta}\mathbb{E}\Bigg[e^{\theta\big(M\big(1+\underset{0\leq j \leq J}{\max}\big\{\big\|\ms_1^{\vy,j}\big\|_2^2+\big\|\ms_2^{\vy,j}\big\|_2^2\big\}\big)-1/\dt\big)}\Bigg]<0$ for $\theta>0$, we get $\displaystyle \dt<\frac{1}{M\big(1+\underset{0\leq j \leq J}{\max}\big\{\big\|\ms_1^{\vy,j}\big\|_2^2+\big\|\ms_2^{\vy,j}\big\|_2^2\big\}\big)}$ $\pp$-a.s, where due to the explicit nature of our time discretization scheme (by the EM method), then the right-hand side of this inequality can be regarded as a ratio of the "length-scale" to the "velocity" of the uncertainty levels over the latent state space (see expressions of $M$ in the proof of Lemma \ref{lem:tailboundsspectral}). 
    \end{itemize}
\end{itemize}
To establish the assertion of this theorem we require various Lemmas, which we state and prove, along with the proof of this theorem, in Appendix \ref{sec:Appendix_C}.

\begin{thm} [\textbf{Optimal Nonlinear Smoother State Estimation Backward Equations}]
    \label{thm:smoothing}
    Let the assumptions of Theorem \ref{thm:condgaussianity} and the additional regularity conditions \textbf{\magenta{(9)}}--\textbf{\magenta{(13)}}, to hold. In addition, let
    \begin{equation} \label{eq:positivedefinite}
        \pp\left( \underset{t\in[0,T]}{\mathrm{inf}}\{ \mathrm{det}(\mr{\nf}(t))\}>0\right)=1,
    \end{equation}
    since it is not guaranteed that the positive-definiteness of the filter posterior covariance tensor is preserved in the backward pass. Then the system of optimal nonlinear smoother equations is given by the following backward differential equations,
    \begin{align}
        \smooth{\rmd \vm{\ns}}(t)&=-(\ml^\vy\vm{\ns}+\vf^\vy-\mb\mr{\nf}^{-1}(\vm{\nf}-\vm{\ns}))\rmd t+(\ms^\vy\circ \ms^\vx)(\ms^\vx\circ \ms^\vx)^{-1}( \smooth{\rmd\vx}+(\ml^\vx\vm{\ns}+\vf^\vx)\rmd t), \label{eq:revbackinter1} \\
         \smooth{\rmd \mr{\ns}}(t)&= -\big((\ma+\mb\mr{\nf}^{-1})\mr{\ns}+\mr{\ns}(\ma+\mb\mr{\nf}^{-1})^\dagger-\mb\big)\rmd t,\label{eq:revbackinter2}
    \end{align}
    for $T\geq t\geq 0$, with $t$ running backward, where the auxiliary matrices $\ma$ and $\mb$ are defined in \eqref{eq:auxiliarymata}--\eqref{eq:auxiliarymatb}. The backward-arrow notation on the left-hand side of \eqref{eq:revbackinter1}--\eqref{eq:revbackinter2} and in $\smooth{\rmd \vx}$ on right-hand side of \eqref{eq:revbackinter1} is to be understood as
    \begin{equation*}
        \smooth{\rmd \vm{\ns}}(t):=\lim_{\Delta t\to 0} \left(\vm{\ns}(t)-\vm{\ns}(t+\Delta t)\right), \quad \smooth{\rmd \mr{\ns}}(t):=\lim_{\Delta t\to 0} \left(\mr{\ns}(t)-\mr{\ns}(t+\Delta t)\right), \quad \smooth{\rmd \vx}(t):=\lim_{\Delta t\to 0} \left(\vx(t)-\vx(t+\Delta t)\right),
    \end{equation*}
    for $\vm{\ns}(t)$ and $\mr{\ns}(t)$ being the mean and covariance of the smoother posterior distribution at time $t$ which is Gaussian, $\pp\big(\vyt\big|\vx(s), s\in[0,T]\big)\overset{\rmd}{\sim}\mathcal{N}_l\big(\vm{\ns}(t),\mr{\ns}(t)\big)$. In other words, the notation $\smooth{\frac{\rmd\boldsymbol{\cdot}}{\rmd t}}$ corresponds to the negative of the usual derivative, which means the system in \eqref{eq:revbackinter1}--\eqref{eq:revbackinter2} is meant to be solved backward over $[0,T]$. The "starting" values for the smoother posterior statistics, $\left(\vm{\ns}(T),\mr{\ns}(T)\right)$, are the same as those of the corresponding filter estimates at the endpoint $t=T$, $\left(\vm{\nf}(T),\mr{\nf}(T)\right)$.
\end{thm}

We would like to note here that the smoother evolution equations in \eqref{eq:revbackinter1}--\eqref{eq:revbackinter2} are consistent with representations through Markov smoothing semigroups associated with the transport equation for the smoother conditional distribution corresponding to similar nonlinear partially observed backward diffusion flow models \cite{anderson2021backward}.

\section{Optimal Conditional Sampling Procedures} \label{sec:3}

The posterior distribution from filtering or smoothing provides optimal pointwise state estimation in the mean-square sense, where "pointwise" refers to estimation at a fixed time instant. However, in many applications, it is crucial to obtain sampled trajectories of the hidden processes conditioned on observations. These sampled trajectories reflect temporal statistics that cannot be accessed through pointwise posterior estimates alone. Unfortunately, relying solely on posterior distributions is insufficient to generate unbiased, dynamically consistent trajectories. This is because each point in a sampled trajectory must account for the nonlinear temporal dependence and correlations with neighboring points. In practice, the posterior mean time series is often used as a surrogate for the hidden variable, but this approach overlooks the temporal correlation contained in the posterior covariance, namely the correlated uncertainty. It leads to the failure of capturing many key dynamical features, such as the autocorrelation function (ACF). Simply drawing independent samples from the posterior distributions at different time points results in a noisy time series that misses underlying dynamics due to the lack of temporal correlation integration.

The closed-form analytic formulae presented in this section provide an efficient and dynamically consistent method for sampling unobserved trajectories of $\vy$ conditioned on the observed $\vx$. This applies to both forward sampling (up to the current time) and backward sampling (over the entire observation period). These sampling methods allow to improve the understanding of nonlinear dynamical features in CTNDSs. All subsequent results assume the validity of assumptions \textbf{\magenta{(1)}}--\textbf{\magenta{(13)}} (or their weaker counterparts), along with any additional assumptions outlined in Section \ref{sec:2.4}.

\subsection{Filter-Based Forward Sampling} \label{sec:3.1}

In the main theorem of this subsection, we state the formula for optimal nonlinear forward sampling of the unobserved process in a CGNS, conditioned on observable data up to the current time. Its proof can be found in Appendix \ref{sec:Appendix_D}. In what follows, the subscript "$\,\text{f}\,$" stands for "filter" since the forward sampling procedure is based on the filter solution; see also Theorem \ref{thm:forwardsamplefilter} for a consistency result.

\begin{thm}[\textbf{Optimal Nonlinear Forward Sampling Formula}] \label{thm:forwardsample}
    When conditioned on the natural filtration of $\cF$ w.r.t.\ $\{\vx(s)\}_{s\leq t}$ and $\vy(t)$, $\cF_t^{\vx,\vy(t)}:= \cF_t^\vx\vee\cF_t^{\vy(t)}=\sigma\big(\cF_t^\vx\cup\sigma(\vy(t))\big)$ (also known as the join of the (sub-)$\sigma$-algebras), for some $t\in[0,T]$, then the optimal (in the mean-square sense) strategy for sampling the trajectories associated with the unobservable variables $\vy$ at time $t$, which we denote by $\hat{\vy}_{\text{\nf}}$, is given by the following SDE:
    \begin{equation} \label{eq:forwardsample}
        \rmd\hat{\vy}_{\text{\nf}}(t) = \rmd\vm{\nf}+(\ma-\mr{\nf}\mc)(\hat{\vy}_{\text{\nf}}-\vm{\nf})\rmd t+(\mb+\mr{\nf}\mc\mr{\nf})^{1/2}\rmd \vw_{\hat{\vy}_{\text{\nf}}},
    \end{equation}
    where $\rmd \vw_{\hat{\vy}_{\text{\nf}}}$ is an $l$-dimensional complex-valued white noise source that is mutually independent from $\vw_1$ and $\vw_2$, $\ma$ and $\mb$ are given as in \eqref{eq:auxiliarymata}--\eqref{eq:auxiliarymatb}, and $\mc$ is defined as
    \begin{equation} \label{eq:auxiliarymatc}
        \mc(t,\vx) := \ml^\vx(t,\vx)^{\dagger}(\ms^\vx\circ \ms^\vx)^{-1}(t,\vx)\ml^\vx(t,\vx).
    \end{equation}
    The $\boldsymbol{\cdot}^{1/2}$ denotes the square root of a matrix. It also holds that $\mb(t,\vx)+\mr{\nf}(t)\mc(t,\vx)\mr{\nf}(t)$ is necessarily
    nonnegative-definite $\pp$-a.s., or that
    \begin{equation} \label{eq:positivedefiniteforwardsample}
            \pp\left( \underset{t\in[0,T]}{\mathrm{inf}}\{ \mathrm{det}(\mb(t,\vx)+\mr{\nf}(t)\mc(t,\vx)\mr{\nf}(t))\}>0\right)=1,
        \end{equation}
    and so its square root is unique as can be shown by its eigendecomposition.
\end{thm}

Similar to the filtering procedure in Theorem \ref{thm:filtering}, the forward sampling procedure in Theorem \ref{thm:forwardsample} is a sequential method, which makes use of the information only in the past to compute the pathwise value at the current time instant, thus allowing it to be integrated into online schemes where the observations arrive in a serial manner. Furthermore, the formula given in \eqref{eq:forwardsample} states that the sampled trajectory of $\vy$, $\hat{\vy}_{\text{\nf}}$, is based on the posterior mean estimate of the filter $\vm{\nf}$, but it also includes temporal correlated uncertainties that depend both on the observations and the underlying nonlinear model, both in the deterministic and stochastic dynamics. Actually, through an equivalent representation of the filter posterior mean and covariance matrix (using elementary algebraic manipulations) \cite{liptser2001statisticsII},
\begin{align}
\begin{split}
    \rmd\vm{\nf}(t)&=\big((\ma-\mr{\nf}\mc)\vm{\nf}+\vf^\vy\big)\rmd t+\big(\ms^\vy\circ \ms^\vx+\mr{\nf}\left(\ml^\vx\right)^{\dagger}\big)(\ms^\vx\circ \ms^\vx)^{-1}\big(\rmd \vx - \vf^\vx\rmd t\big), \\
    \rmd\mr{\nf}(t)&=\big((\ma-\mr{\nf}\mc)\mr{\nf}+\mr{\nf}(\ma-\mr{\nf}\mc)^\dagger+\mb+\mr{\nf}\mc\mr{\nf}\big)\rmd t,
\end{split}\label{eq:alternativefiltermean}
\end{align}
it becomes quickly apparent that for the exponential mean-square stability of the filter state estimation equations to be satisfied, it is necessary that the real parts of all the eigenvalues of $\mathbf{A}-\mr{\nf}\boldsymbol{\Gamma}$ are negative \cite{soong1973random, neckel2013random, khasminskii2012stochastic, arnold2014random, crauel2015nonautonomous, han2017random, bishop2019stability}, and as such the samples $\hat{\vy}_{\text{\nf}}$ have the tendency to collapse towards the filter mean state $\vm{\nf}$; this convergence assertion is made concrete in Theorem \ref{thm:forwardsamplefilter} and its proof. Notice how this also illustrates the fact that the damping in the filter posterior mean and the forward-based sampled trajectories is the same. For more on the stochastic stability of the optimal nonlinear filter, we recommend the surveys of Chigansky \cite{chigansky2006stability} and van Handel \cite{vanhandel2007filtering}, and their references therein. As for the stability of the nonlinear filter covariance matrix, see Wang \& Lei for the discrete-time case \cite{wang1999stability} and Bishop \& Del Moral for the continuous-time case \cite{bishop2019stability}.

Another important observation regarding \eqref{eq:forwardsample}, is that its noise feedback is not simply $(\ms^\vy\circ\ms^\vy)^{1/2}$ as in the original latent dynamical equation \eqref{eq:condgauss2}, but instead
\begin{equation*}
    (\mb+\mr{\nf}\mc\mr{\nf})^{1/2}=\left((\ms^\vy\circ\ms^\vy)-\big((\ms^\vy\circ\ms^\vx)(\ms^\vx\circ\ms^\vx)^{-1}(\ms^\vx\circ\ms^\vy)-\mr{\nf}\mc\mr{\nf}\big)\right)^{1/2}.
\end{equation*}
Now, while it is tempting here to presume that this equation has much larger damping in its uncertainty compared to the original model, by enforcing that $\mb+\mr{\nf}\mc\mr{\nf}\preceq(\ms^\vy\circ\ms^\vy)$, where "$\preceq$" is understood in the Loewner partial ordering sense over the closed pointed convex cone of nonnegative-definite matrices, it could very well be the case that this is not actually true. Indeed, observe that for $\mb+\mr{\nf}\mc\mr{\nf}\preceq(\ms^\vy\circ\ms^\vy)$ to hold, it is equivalent to have
\begin{equation*}
    {(\ms^\vy\circ\ms^\vx)(\ms^\vx\circ\ms^\vx)^{-1}(\ms^\vx\circ\ms^\vy)-\mr{\nf}\mc\mr{\nf}\succeq \mathbf{0}_{l\times l}}.
\end{equation*}
Now, while $(\ms^\vy\circ\ms^\vx)(\ms^\vx\circ\ms^\vx)^{-1}(\ms^\vx\circ\ms^\vy)$ is positive-definite by assumption \textbf{\magenta{(5)}}, and so is $\mr{\nf}\mc\mr{\nf}$ for the same reason together with the positive-definiteness of the filter covariance matrix, after possibly the additional assumption of $\ms^\vy\circ\ms^\vx$ and $\ml^\vx$ being $\pp$-a.s.\ full rank, a simple application of the Lidskii-Mirsky-Wielandt majorization theorem yields \cite{li1999lidskii, stewart1990matrix}
\begin{equation} \label{eq:eigconditionwindow}
    \mu_i((\ms^\vy\circ\ms^\vx)(\ms^\vx\circ\ms^\vx)^{-1}(\ms^\vx\circ\ms^\vy)-\mr{\nf}\mc\mr{\nf})\leq \rho((\ms^\vy\circ\ms^\vx)(\ms^\vx\circ\ms^\vx)^{-1}(\ms^\vx\circ\ms^\vy))-\mu_i(\mr{\nf}\mc\mr{\nf}),
\end{equation}
for any $i=1,\ldots,l$ (where here we use the notation from Lemma \ref{lem:tailboundsspectral}). Even under our set of regularity conditions, this leaves enough leeway such that $\mu_{i_o}((\ms^\vy\circ\ms^\vx)(\ms^\vx\circ\ms^\vx)^{-1}(\ms^\vx\circ\ms^\vy)-\mr{\nf}\mc\mr{\nf})<0$ for some $i_o\in\{1,\ldots,l\}$ and appropriately chosen model parameters in \eqref{eq:condgauss1}--\eqref{eq:condgauss2}, which means $\mb+\mr{\nf}\mc\mr{\nf}\not\preceq(\ms^\vy\circ\ms^\vy)$. Actually, such an example is chosen in our numerical case study in Section \ref{sec:4} as to illustrate this fact. Intuitively, it means some observations can actually deteriorate the filter posterior state estimation, thus potentially leading to an increase in the uncertainty in the sampled trajectories.

On the other hand, if sufficient conditions are enforced such that $\mb+\mr{\nf}\mc\mr{\nf}\preceq(\ms^\vy\circ\ms^\vy)$ holds $\pp$-a.s., then the uncertainty represented by this noise coefficient takes into account the information from the observations effectively, which is embodied in $\mr{\nf}$, $(\ms^\vx\circ\ms^\vx)^{-1}$, and $\ml^\vx$, thus leading to a much smaller overall uncertainty instead. Therefore, the uncertainty, represented by the equilibrium variance, is lower in the forward samples $\hat{\vy}_{\text{\nf}}$ conditioned on the observations, \eqref{eq:forwardsample}, than the $\vy$ resulting from a free run of the coupled model, \eqref{eq:condgauss2}. Of course, this reduction in the uncertainty stems from the information that is provided by the observational time series.

Another useful conclusion that can be drawn from \eqref{eq:forwardsample} stems from looking at the associated forward equation for the residual. Specifically, we define the residual of the sampled trajectory $\hat{\vy}_{\text{\nf}}$ with respect to the filter mean state $\vm{\nf}$, i.e., $\Tilde{\hat{\vy}}_{\text{f}}=\hat{\vy}_{\text{\nf}}-\vm{\nf}$; notice how this is different from the mean-fluctuation decomposition in \eqref{eq:reynoldsdecomp}, which is why we do not use the overline notation but instead a tilde (see also Theorem \ref{thm:forwardsamplefilter}). Then, by \eqref{eq:forwardsample}, $\Tilde{\hat{\vy}}_{\text{f}}$ satisfies the SDE:
\begin{equation} \label{eq:residualforward}
    \rmd\Tilde{\hat{\vy}}_{\text{f}}(t) = \frac{1}{2}\big(\rmd\hat{\vy}_{\text{\nf}}+(\ma-\mr{\nf}\mc)\Tilde{\hat{\vy}}_{\text{f}}\rmd t+(\mb+\mr{\nf}\mc\mr{\nf})^{1/2}\rmd \vw_{\hat{\vy}_{\text{\nf}}}\big).
\end{equation}
This equation generates the residual time series, reflecting the uncertainty in the sampled trajectory, which exhibits non-trivial temporal dependence. The temporal behavior of $\Tilde{\hat{\vy}}_{\text{f}}$ in \eqref{eq:residualforward} differs significantly from that of $\vy$ in \eqref{eq:condgauss2}, as seen even in the one-dimensional case. Assume $k=l=1$, so both $\vx=x$ and $\vy=y$ are one-dimensional, with constant coefficients, where $\ml^\vx=\Lambda^x>0$ and $\ml^\vy=\Lambda^y<0$. Here, $\mr{\nf}=R_f$ converges to a positive constant $R_{\infty}>0$ over time, as shown by \eqref{eq:filter2} \cite{gardiner2009stochastic, sarkka2019applied}. The decorrelation time of the marginal, conditional, and joint distributions (since the constant-coefficient CGNS has reversible conditional Gaussianity) equals the reciprocal of its damping coefficient \cite{harlim2008mathematical}, per Doob's theorem \cite{doob1942topics, doob1942brownian, borisov1983criterion}, which states that the autocorrelation of a stationary Gaussian Markov process is proportional to $e^{-as}$, where $a$ is the damping coefficient. For $\ma=A$ and $\mc=\Gamma$, the decorrelation times for the original system \eqref{eq:condgauss2} and the residual of the sampled trajectory \eqref{eq:residualforward} are $-1/\Lambda^y>0$ and $-1/(A-R_{\infty}\Gamma)>0$, respectively, with the latter holding asymptotically. Then, by assuming the noise feedbacks of the system are positive, we have that $-(A-R_{\infty}\Gamma)>-\Lambda^y>0$, which shows that the decorrelation time of the residual is shorter due to the added information from observations. This means that, in a linear time-invariant CGNS, the residual decorrelates faster than the unobserved process. However, this does not necessarily imply lower uncertainty in the forward samples.

We conclude this subsection with the following theorem, which links the pointwise Gaussian statistics of the forward-sampled trajectories to those of the filter posterior distribution. This establishes that the pointwise Gaussian statistics of the trajectories generated by the forward sampling procedure are consistent, at each time instant, with those of the filtering posterior distribution. See Appendix \ref{sec:Appendix_E} for the proof.
\begin{thm}[\textbf{Consistency Between Forward Sampling and Filter State Estimation}] \label{thm:forwardsamplefilter}
    Assume that the real parts of all the eigenvalues of $\mathbf{A}-\mr{\nf}\boldsymbol{\Gamma}$ are negative ($\pp$-a.s.). Then, the conditional Gaussian nonlinear filter estimate, \eqref{eq:filter1}--\eqref{eq:filter2}, can be recovered by applying a mean-fluctuation decomposition to the forward sampling equation \eqref{eq:forwardsample}. In other words:
    \begin{enumerate}[label=\textbf{(\alph*)}]
        \item The ensemble average of the forward sampling equation \eqref{eq:forwardsample}, conditioned on the natural filtration of $\cF$ w.r.t.\ $\{\vx(s)\}_{s\leq t}$ and $\vy(t)$, $\cF_t^{\vx,\vy(t)}$, for some $t\in[0,T]$, is the time evolution of the posterior mean in \eqref{eq:filter1}.

        \item The ensemble average of the quadratic form of the residual equation of \eqref{eq:forwardsample}, conditioned on the natural filtration of $\cF$ w.r.t.\ $\{\vx(s)\}_{s\leq t}$ and $\vy(t)$, $\cF_t^{\vx,\vy(t)}$, for some $t\in[0,T]$, is the time evolution of the posterior covariance in \eqref{eq:filter2}.
    \end{enumerate}
\end{thm}

\subsection{Smoother-Based Backward Sampling} \label{sec:3.2}

The forward sampling formula in Theorem \ref{thm:forwardsample} is derived from the nonlinear filter estimate, as outlined in Theorem \ref{thm:forwardsamplefilter}. Since smoothing provides a more accurate and less uncertain state estimation, it is natural to consider a corresponding smoother-based sampling formula for the latent variables. This formula runs backward in time and is conditioned on the full sequence of measurements from the observable variables. The following result explicitly presents this backward sampling procedure for the latent variables, where its proof can be found in Appendix \ref{sec:Appendix_F}. As before, the subscript "$\,\text{s}\,$'' denotes "smoother'', indicating that this sampler is based on the backward smoother solution. See also Theorem \ref{thm:backwardsamplesmoother} for the consistency result.

\begin{thm}[\textbf{Optimal Nonlinear Backward Sampling Formula}] \label{thm:backwardsample}
    When conditioned on the terminating observable $\sigma$-algebra, $\cF_T^{\vx}$, i.e., a complete realization of the observed variables $\vx$, then the optimal (in the mean-square sense) strategy for sampling the trajectories associated with the unobservable variables $\vy$ at time $T\geq t\geq 0$, which we denote by $\hat{\vy}_{\text{\ns}}$, satisfies the following explicit formula:
    \begin{equation} \label{eq:backwardsample}
        \smooth{\rmd \hat{\vy}_{\text{\ns}}}(t) = (-\ml^\vy \vm{\nf}-\vf^\vy +(\ma+\mb\mr{\nf}^{-1})(\vm{\nf}-\hat{\vy}_{\text{\ns}}))\rmd t+ (\ms^\vy\circ \ms^\vx)(\ms^\vx\circ \ms^\vx)^{-1}\big(\smooth{\rmd \vx}+(\ml^\vx\vm{\nf}+\vf^\vx)\rmd t \big) +\mb^{1/2}\rmd \vw_{\hat{\vy}_{\text{\ns}}},
    \end{equation}
   for $T\geq t\geq 0$ and where $\rmd \vw_{\hat{\vy}_{\text{\ns}}}$ are independent \textit{forward} white-noise increments in \textit{reverse} time, $\tau=T-t$, and $\ma,\mb$ are given in \eqref{eq:auxiliarymata}--\eqref{eq:auxiliarymatb}. As shown in Theorem \ref{thm:filtering}, the square root of $\mb$ exists and is unique ($\pp$-a.s.). The backward-arrow notation on the left-hand is to be understood as
    \begin{equation*}
        \smooth{\rmd \hat{\vy}_{\text{\ns}}}(t):=\lim_{\Delta t\to 0} \left(\hat{\vy}_{\text{\ns}}(t)-\hat{\vy}_{\text{\ns}}(t+\Delta t)\right),
    \end{equation*}
    and likewise for $\smooth{\rmd \vx}$ on the right-hand side, where on the right-hand side of \eqref{eq:backwardsample}, $\hat{\vy}_{\text{\ns}}$ is evaluated at $t+\dt$ while all other coefficients and the filter estimate are being calculated at $t$. The formula in \eqref{eq:backwardsample} starts from $t=T$ and it is ran backwards towards $t=0$. The initial value of $\hat{\vy}_{\text{\ns}}$ in \eqref{eq:backwardsample} is drawn from the conditional Gaussian distribution corresponding to the filter estimate at $t=T$, $\mathcal{N}_l(\vm{\nf}(T),\mr{\nf}(T))$, due to the fact that the filter and the smoother estimates coincide at the end point.
\end{thm}

Comparing \eqref{eq:backwardsample} with the true underlying dynamics of $\vy$ in \eqref{eq:condgauss2}, beyond the usage of the filter mean for the deterministic dynamics, the backward sampling equation involves the extra forcing term
\begin{equation*}
    (\ma+\mb\mr{\nf}^{-1})(\vm{\nf}-\hat{\vy}_{\text{\ns}})\rmd t+(\ms^\vy\circ \ms^\vx)(\ms^\vx\circ \ms^\vx)^{-1}\big(\smooth{\rmd \vx}+(\ml^\vx\vm{\nf}+\vf^\vx)\rmd t \big).
\end{equation*}
This correction term is similar to the one in the optimal nonlinear smoother mean equation, \eqref{eq:revbackinter1}:
\begin{equation*}
    \mb\mr{\nf}^{-1}(\vm{\nf}-\vm{\ns})\rmd t+(\ms^\vy\circ \ms^\vx)(\ms^\vx\circ \ms^\vx)^{-1}\big( \smooth{\rmd\vx}+(\ml^\vx\vm{\ns}+\vf^\vx)\rmd t \big),
\end{equation*}
which takes into account the information provided by the observations and builds a connection between the smoother solution and backward sampled trajectories. Specifically, for the case of the backward samples, this correction term plays the important role of forcing and driving the sampled trajectory to meander around the filter mean state $\vm{\nf}$, under some suitable assumptions on the model parameters appearing in $\ma$ and $\mb$. However, due to the memory of the latent process, the system response of the forcing has a delayed effect, and actually fluctuates around the smoother mean $\vm{\ns}$,  which can be seen by the following corollary to Theorem \ref{thm:backwardsample}, which provides an alternative way to express the backward sampling formula is as,
\begin{equation} \label{eq:alternativebacksample}
    \smooth{\rmd \hat{\vy}}_{\text{s}}(t)=\smooth{\rmd \vm{\ns}}-(\ma+\mb\mr{\nf}^{-1})(\hat{\vy}_{\text{\ns}}-\vm{\ns})\rmd t+\mb^{1/2}\rmd \vw_{\hat{\vy}_{\text{\ns}}}.
\end{equation}

To see why this holds, we simply take the difference between \eqref{eq:backwardsample} and \eqref{eq:revbackinter1} and perform the necessary algebra. The result in \eqref{eq:alternativebacksample} presents the backward sampling framework in a form analogous to the forward sampling strategy in \eqref{eq:forwardsample}, showing that the backward-sampled trajectory oscillates around the smoother mean state. Furthermore, note that the deterministic component of the forcing has feedback of the form $\mb\mr{\nf}^{-1}+\ma$, which by \eqref{eq:revbackinter2},
\begin{equation*}
    \smooth{\rmd \vm{\ns}}(t)=-(\ma+\mb\mr{\nf}^{-1})\vm{\ns}-(\vf^\vy-\mb\mr{\nf}^{-1}\vm{\nf})\rmd t+(\ms^\vy\circ \ms^\vx)(\ms^\vx\circ \ms^\vx)^{-1}( \smooth{\rmd\vx}+\vf^\vx\rmd t),
\end{equation*}
from \eqref{eq:revbackinter1}, and the theory of random ODEs, governs the stability of the smoother covariance through its spectrum’s projection onto the real axis, which must remain positive (accounting for the minus signs in \eqref{eq:revbackinter1}--\eqref{eq:revbackinter2}) \cite{crauel2015nonautonomous, kloeden2011nonautonomous, lu2021mathematical}. To simplify the role of observations, consider no noise cross-interaction and constant noise strength in $\vy$, i.e., $\ms^\vy\circ \ms^\vy$ is constant. In this case, by \eqref{eq:filter2}, $\mr{\nf}$ positively correlates with the observational noise $\ms^\vx\circ \ms^\vx$, with a lower uncertainty in $\vx$ leading to a smaller $\mr{\nf}$ (in the trace sense) and a larger correction weight (since the quadratic term in its random Riccati is being weighted by $-\mc$ \cite{bishop2019stability}). It means observations play an important role in determining the backward sampled trajectories. Concerning the dynamical properties of the CGNS, the backward sampling formula also possesses the enticing feature that it retains the structure of the latent dynamics, which has the important implication that the temporal ACF, as well as other higher-order moments and temporal correlations, and even the marginal of both components of the state space and joint PDFs, can be accurately recovered using these sampled trajectories. An example towards this direction is showcased in the numerical case study included in Section \ref{sec:4}.

Analogous now to Theorem \ref{thm:forwardsamplefilter}, the following result establishes a consistency connection between the pointwise statistics of the backward sampled trajectories and the Gaussian statistics of the smoother posterior distribution. The proof to this result can be found in Appendix \ref{sec:Appendix_G}

\begin{thm}
    [\textbf{Consistency Between Backward Sampling and Smoother State Estimation}] \label{thm:backwardsamplesmoother}
    Assume that the real parts of all the eigenvalues of $\ma+\mb\mr{\nf}^{-1}$ are positive ($\pp$-a.s.). Additionally, assume that at $t=T$ the backward samples coincide with the conditional Gaussian nonlinear smoother mean. Then, the conditional Gaussian nonlinear smoother estimate, \eqref{eq:revbackinter1}--\eqref{eq:revbackinter2}, can be recovered by applying a mean-fluctuation decomposition to the backward sampling equation \eqref{eq:backwardsample}. In other words:
    \begin{enumerate}[label=\textbf{(\alph*)}]
        \item The ensemble average of the backward sampling equation \eqref{eq:backwardsample}, conditioned on the natural filtration of $\cF$ w.r.t.\ $\{\vx(s)\}_{s\leq T}$, $\cF_T^{\vx}$, i.e., a complete realization of the observed variables $\vx$, is the time evolution of the smoother mean in \eqref{eq:revbackinter1}.
        \item The ensemble average of the quadratic form of the residual equation of \eqref{eq:backwardsample}, conditioned on the natural filtration of $\cF$ w.r.t.\ $\{\vx(s)\}_{s\leq T}$, $\cF_T^{\vx}$, i.e., a complete realization of the observed variables $\vx$, is the time evolution of the smoother covariance in \eqref{eq:revbackinter2}.
    \end{enumerate}
\end{thm}

\subsection{A Hierarchy of Uncertainty in the Sampling Formulae for the Unobserved Variables} \label{sec:3.3}

Based on the results provided by Theorems \ref{thm:forwardsample}--\ref{thm:backwardsamplesmoother}, as well as the discussions accompanying these conclusions in Section \ref{sec:3.1} and \ref{sec:3.2}, respectively, it is then possible, under some suitable assumptions, to infer the emergence of a hierarchy in the noise feedback components of the sampling uncertainty, which can be expressed through a double matrix inequality between the noise feedback matrix in the original equation driving the latent dynamics, \eqref{eq:condgauss2}, and the noise amplitude matrices appearing in the optimal nonlinear forward and backward sampling formulas, \eqref{eq:forwardsample} and \eqref{eq:backwardsample}, respectively. We state this observation in the following remark.
\begin{rmk}
    [\textbf{Hierarchy in the Uncertainty of the Latent Dynamics}] \label{rmk:uncertaintyhierarchy}
Let "$\succeq$" to define the Loewner partial order over the closed and pointed convex cone of nonnegative-definite matrices in the space of self-adjoint matrices, with its interior being the open blunt convex cone of positive-definite matrices (in the spectral operator norm). This means $\mathbf{R}\succeq \mathbf{T}$ if $\mathbf{R}-\mathbf{T}$ is nonnegative definite, and $\mathbf{R}\succ\mathbf{T}$ if $\mathbf{R}-\mathbf{T}$ is positive definite. Then, it always holds that $\pp$-a.s.\ $\ms^\vy\circ\ms^\vy\succeq \mb$, with strict inequality if $\ms^\vy\circ\ms^\vx$ is $\pp$-a.s.\ of full rank, and that $\pp$-a.s.\ $\mb+\mr{\nf}\mc\mr{\nf}\succeq \mb$, with strict inequality if $\ml^\vx$ is $\pp$-a.s.\ of full rank (with the latter also being important in establishing the stochastic controllability of the CGNS). These are true for any $t\in[0,T]$.

As for, $\ms^\vy\circ\ms^\vy\succeq\mb+\mr{\nf}\mc\mr{\nf}$, or, equivalently, $(\ms^\vy\circ\ms^\vx)(\ms^\vx\circ\ms^\vx)^{-1}(\ms^\vx\circ\ms^\vy)\succeq \mr{\nf}\mc\mr{\nf}$, it does not necessarily always hold even under the current set of assumptions \textbf{\magenta{(1)}}--\textbf{\magenta{(13)}}, including the additional ones considered in Sections \ref{sec:2} and \ref{sec:3} thus far, as already discussed in Section \ref{sec:3.1} (see \eqref{eq:eigconditionwindow}). Following an analysis similar to the one in Lemma \ref{lem:tailboundsspectral}, specifically through the Laplace transform method for bounds in expectation of extremal eigenvalues \cite{tropp2019matrix}, we have that such a result can hold in expectation (or in the weak-sense), i.e., $\ee{(\ms^\vy\circ\ms^\vx)(\ms^\vx\circ\ms^\vx)^{-1}(\ms^\vx\circ\ms^\vy)}\succeq \ee{\mr{\nf}\mc\mr{\nf}}$, with $\ee{\boldsymbol{\cdot}}$ denoting the full expectation, if sufficient conditions are assumed to provide the control
\begin{equation*}
    \ee{\mathrm{tr}\left(e^{\theta\big((\ms^\vy\circ\ms^\vx)(\ms^\vx\circ\ms^\vx)^{-1}(\ms^\vx\circ\ms^\vy)-\mr{\nf}\mc\mr{\nf}\big)}\right)}\leq 1, \ \forall \theta<0,
\end{equation*}
As such, we have:
\begin{equation}
\begin{gathered}
    \ms^\vy\circ\ms^\vy\succeq \mb, \quad \mb+\mr{\nf}\mc\mr{\nf}\succeq \mb, \ \forall t\in[0,T] \ (\pp\text{-a.s.}),\\
    \forall t\in[0,T]: \ \ee{\mathrm{tr}\left(e^{\theta\big(\ms^\vy\circ\ms^\vy-\mb-\mr{\nf}\mc\mr{\nf}\big)}\right)}\leq 1, \ \forall \theta<0 \Leftrightarrow (\ms^\vy\circ\ms^\vx)(\ms^\vx\circ\ms^\vx)^{-1}(\ms^\vx\circ\ms^\vy)\succeq_{\mathbb{E}} \mr{\nf}\mc\mr{\nf}.
\end{gathered} \label{eq:uncertaintyhier}
\end{equation}
\end{rmk}

Essentially, Remark \ref{rmk:uncertaintyhierarchy} asserts that an increasing incorporation of information in the simulation of the latent dynamics leads to a decrease in their sample-to-sample fluctuations, since $(\ms^\vy\circ\ms^\vy)$, $\mb+\mr{\nf}\mc\mr{\nf}$, and $\mb$ are the noise feedbacks matrices corresponding to \eqref{eq:condgauss2}, \eqref{eq:forwardsample}, and \eqref{eq:backwardsample}, respectively. The result in \eqref{eq:uncertaintyhier} does so via a concise, intuitive explanation. When not conditioning on any of the observations, the fluctuations component of the latent uncertainty is large. However, by incorporating the information provided by the observed variables, this uncertainty can potentially decrease. If we move forward, i.e., by conditioning on only up to the current observational values, the initial latent uncertainty defined by the sum of the Gramians of the noise feedback matrices in the unobservable process, $\ms^\vy\circ \ms^\vy$, is "corrected" by the matrix $(\ms^\vy\circ\ms^\vx)(\ms^\vx\circ\ms^\vx)^{-1}(\ms^\vx\circ\ms^\vy)-\mr{\nf}\mc\mr{\nf}$ to give the forward-based sampling noise amplitude $\mb+\mr{\nf}\mc\mr{\nf}$. As noted by the discussion following the proof of Theorem \ref{thm:forwardsample}, specifically \eqref{eq:eigconditionwindow}, it does not always hold that $\ms^\vy\circ \ms^\vy\succeq\mb+\mr{\nf}\mc\mr{\nf}$, but assuming the validity of \eqref{eq:uncertaintyhier}, then this matrix inequality (in either the strong or weak sense) conveys that the information gained by incorporating the observations of $\vx$ into the state estimation, is inversely proportional to the observational process' uncertainty weighted by the noise cross-interaction, i.e., $(\ms^\vy\circ\ms^\vx)(\ms^\vx\circ\ms^\vx)^{-1}(\ms^\vx\circ\ms^\vy)$, but which is then counterbalanced at the same time by the information that is lost due to the inherent uncertainty of not knowing the unobservable process $\vy$, $-\mr{\nf}\mc\mr{\nf}$, with this term being the weighted factor (weighted by the filter posterior covariance) which defines the quadratic dynamics in the filter covariance's random evolution equation. Going further than that, when moving backwards, i.e., by conditioning on a complete batch of observations and then proceeding with smoothing, always guarantees a decrease in the overall, as well as component-wise, uncertainty of the hidden process, with the noise feedback now decreasing down to $\mb$, with the difference in fluctuations being just the (generally) nonnegative-definite matrix $(\ms^\vy\circ\ms^\vx)(\ms^\vx\circ\ms^\vx)^{-1}(\ms^\vx\circ\ms^\vy)$, since no filter uncertainty is adopted, and so the only uncertainty stems from that of the sampled trajectory of $\vx$, $(\ms^\vx\circ\ms^\vx)^{-1}$, and any potential noise cross-interactions that emerge by the model structure, i.e., whether $(\ms^\vx\circ\ms^\vy)\equiv\mathbf{0}_{k\times l}$ or not; this further showcases, analytically, the unbiasedness of the smoother estimate.

It should be noted that we do not make any explicit assertions about the overall uncertainty of the system and, by extension, that of the posterior estimated states. We are just presenting some results concerning its constituents, mainly in this case the noise component of the dynamical uncertainty. As can be clearly seen from the forward and backward fluctuation sampling equations, \eqref{eq:forwresidualeqn} and \eqref{eq:backresidualeqn}, the uncertainty is partly controlled by the noise feedback matrices appearing in \eqref{eq:uncertaintyhier}, but also by the corresponding damping multiplicative factors which help relax the respective sample state towards the corresponding posterior mean in an exponentially fast fashion, or, equivalently, by forcing the magnitude of the fluctuations down to zero (where these damping terms, as already noted, also control the mean-square stability of the associated posterior statistics through the projection of their spectrum onto the abscissa of the complex plane \cite{soong1973random, neckel2013random, khasminskii2012stochastic, arnold2014random, crauel2015nonautonomous, kloeden2011nonautonomous, han2017random, lu2021mathematical}). So, while the noise amplitude matrices satisfy a hierarchy such as the one indicated in Remark \ref{rmk:uncertaintyhierarchy}, this only partly illuminates the temporal behavior of the uncertainty of the system and how the observations affect it. After all, it is known that the filter and smoother estimated states will always decrease the overall state uncertainty regardless of how informative (or not) the current measurement or observation is. As such, even if the filter noise amplitude might be larger, there is always a trade-off between damping and noise, as indicated by the bias-variance decomposition of the mean-square error \cite{kohavi1997bias}, which should necessarily lead to a less uncertain state by the optimality in this metric. See Section \ref{sec:4.3} and Figure \ref{fig:case_study_fig_3} for further discussion on the matter.

\subsection{Pathwise Error Metrics for the Posterior Mean Time Series and the Sampled Trajectories} \label{sec:3.4}
In many applications, the skill in the probabilistic estimated state is quantified by the pathwise error between the posterior mean time series and the truth. The motivation of adopting such a pathwise approach is that for a Gaussian posterior distribution, the posterior mean estimate at each fixed time instant is a convex combination of the prior mean at that time instant and the pointwise maximum likelihood estimate, a feature that is inherited from the conjugacy of the model and its inclusion in the exponential family of distributions \cite{diaconis1979conjugate}.

The pathwise error between the posterior mean time series and the truth is easy to compute in practice, by just taking the difference between the two with respect to some norm at each time instant, which forms an error time series. However, by computing the pathwise error solely via the posterior mean time series in such a naive approach, we neglect the posterior uncertainty, which is a time-dependent function that contains a significant amount of information about the posterior estimate, something which is especially essential for CTNDSs. Besides, the posterior mean time series is simply a collection of the convex combinations of the pointwise maximum likelihood estimates and prior means from each time instant, which may not even be dynamically consistent with the underlying model. To resolve this issue, since the sampled trajectories from either \eqref{eq:forwardsample} or \eqref{eq:backwardsample} take into account the posterior uncertainty and temporal correlations, we may then assess the pathwise error through these sampled trajectories instead, which is more suitable for many circumstances, especially in understanding the dynamical behavior of the unobserved processes conditioned on the observations.

This subsection focuses on comparing the pathwise error in the posterior mean time series and that in the sampled trajectories. Towards this investigation, the standardized root-mean-square error (SRMSE) and the Pearson anomaly pattern correlation coefficient (Corr) are the two most widely used pathwise measurements in practice \cite{kalnay2003atmospheric, taylor2001summarizing, houtekamer1998data, lermusiaux1999data, hendon2009prospects, kim2012seasonal}. Let $\vy(t)$ denote the true $l$-dimensional time series of the unobservables and $\hat{\vy}(t)$ to represent a sampled trajectory of the unobservables, either from the filter-based forward- or smoother-based backward-conditional sampling formulas, \eqref{eq:forwardsample} and \eqref{eq:backwardsample}, respectively. Furthermore, we let $\overline{\hat{\vy}}(t):=\ee{\hat{\vy}(t)\big|\mathcal{G}}$ to denote the posterior mean time series from filtering or smoothing, depending on how we sampled $\hat{\vy}(t)$ (i.e., $\mathcal{G}=\cF_t^{\vx,\vyt}$ or $\mathcal{G}=\cF_T^{\vx}$, respectively), where the expectation is taken over all the ensembles (i.e., the samples denote the particles) at a fixed time instant. By Theorems \ref{thm:forwardsamplefilter} and \ref{thm:backwardsamplesmoother}, these are indeed the filter and smoother posterior mean time series for $t\in[0,T]$, respectively. Moreover, let $\mathbf{c}(t)\in\mathbb{C}^p$ be a stochastic process such that $\mathbf{c}(t)\in L^1([0,T];L^2\left(\Omega, \cF, \pp; \mathbb{C}^{p}\right))$, and define its discrete-time signal $\mathbf{c}_{\dt}:=\{\mathbf{c}^j_{\dt}=\mathbf{c}(t_j)\}_{j=0,1,\ldots,J}$ induced by the uniform time-discretization scheme outlined in the beginning of Section \ref{sec:2.5}. Then, we define the time average, the temporal covariance matrix, and standard deviation over $[0,T]$ of the discretized process through the discrete measure defined by the partition as follows:
\begin{gather}
    \langle \mathbf{c}_{\dt} \rangle := \frac{1}{J} \sum_{j=0}^J \mathbf{c}(t_j), \label{eq:temporal_average}\\
    \mathbf{V}(\mathbf{c}_{\dt}) := \left( V_{jm}(\mathbf{c}_{\dt})\right)_{(J+1)\times(J+1)}, \quad  V_{jm}(\mathbf{c}_{\dt}):=(\mathbf{c}(t_j)-\langle \mathbf{c}_{\dt} \rangle)^\dagger(\mathbf{c}(t_m)-\langle \mathbf{c}_{\dt} \rangle)/J, \label{eq:temporal_covariance}\\
    \mathrm{std}(\mathbf{c}_{\dt}):=\mathrm{tr}(\mathbf{V}(\mathbf{c}_{\dt}))^{1/2}.\label{eq:temporal_std}
\end{gather}
Similarly, we can define analogous metrics for the continuous-time signal via integrals. Then, the SRMSE and the Corr between the true trajectory and the posterior mean time series are defined as (where we drop the subscript of $\dt$ for notational simplicity):
\begin{gather}
    \text{SRMSE}\big(\vy,\overline{\hat{\vy}}\big):=\frac{1}{\mathrm{std}(\vy)}\sqrt{\frac{1}{J}\sum_{j=0}^J\big\|\vy^j-\overline{\hat{\vy}^j}\big\|_2^2}, \label{eq:rmse} \\
    \text{Corr}\big(\vy,\overline{\hat{\vy}}\big):=\frac{\sum_{j=0}^J \big\langle \vy^j-\langle \vy\rangle,\overline{\hat{\vy}^j}-\langle \overline{\hat{\vy}}\rangle\big\rangle_2}{\mathrm{std}(\vy)\mathrm{std}\big(\overline{\hat{\vy}}\big)J},\label{eq:corr}
\end{gather}
where $\big\langle\boldsymbol{\cdot},\boldsymbol{\cdot}\big\rangle_2$ denotes the usual Euclidean inner product over $\mathbb{C}^l$. These quantities depend on the equipartition of the time interval $[0,T]$ and are sample-based (in time), thus contrasting the corresponding population quantities which use the true temporal ensemble average instead. Nevertheless, consistency, both in the weak and strong sense, can be achieved through standard asymptotic statistical analysis in the limit as $\dt\to 0^+$, assuming prior conditions hold. The same formulae apply for the SRMSE and Corr in the sampled trajectories by replacing $\overline{\hat{\vy}}$ with $\hat{\vy}$. Notably, uncertainty in the posterior estimates does not manifest in the pathwise measurements using the posterior mean time series but does when using sample trajectories.

The following theorem establishes a link between the SRMSE in the posterior mean time series and that in the sampled trajectory, whether from forward or backward sampling \eqref{eq:forwardsample} and \eqref{eq:backwardsample}, with the latter accounting for both the average error and the uncertainty in the sampling process.
\begin{thm}[\textbf{Expected Sampling Squared Error}] \label{thm:rmsesampling}
The expected squared error in $\hat{\vy}$ related to the truth $\vy$ is given by
\begin{align}
\begin{split}
    \ee{\mathrm{SRMSE}(\vy,\hat{\vy})^2\big|\mathcal{G}}&=\ee{\frac{1}{\mathrm{std}(\vy)^2J}\sum_{j=0}^J \big\|\vy^j-\hat{\vy}^j\big\|_2^2\Big|\mathcal{G}}\\
    &=\frac{1}{\mathrm{std}(\vy)^2}\Big(\frac{1}{J}\sum_{j=0}^J\big\|\vy^j-\overline{\hat{\vy}^j}\big\|_2^2+\mathbb{E}\Big[\frac{1}{J}\sum_{j=0}^J\big\|\overline{\hat{\vy}^j}-\hat{\vy}^j\big\|_2^2\Big|\mathcal{G}\Big]\Big)\\
    &= \mathrm{SRMSE}\big(\vy,\overline{\hat{\vy}}\big)^2+\frac{\mathrm{std}\big(\overline{\hat{\vy}}\big)^2}{\mathrm{std}(\vy)^2}\ee{\mathrm{SRMSE}\big(\overline{\hat{\vy}},\hat{\vy}\big)^2\big|\mathcal{G}}.
\end{split}\label{eq:biasvariancetrade}
\end{align}
\end{thm}
\begin{proof}[\textbf{Proof of Theorem \ref{thm:rmsesampling}}]
    \eqref{eq:biasvariancetrade} is easily obtained via simple algebra and by the linearity of the expectation, after using the polarization identity,
\begin{equation*}
    \big\|\vy^j-\hat{\vy}^j\big\|_2^2=\big\|\vy^j-\overline{\hat{\vy}}^j+\overline{\hat{\vy}}^j-\hat{\vy}^j\big\|_2^2=\big\|\vy^j-\overline{\hat{\vy}}^j\big\|_2^2+\big\|\overline{\hat{\vy}}^j-\hat{\vy}^j\big\|_2^2+2\mathrm{Re}\big(\big\langle \vy^j-\overline{\hat{\vy}}^j, \overline{\hat{\vy}}^j-\hat{\vy}^j\big\rangle_2\big).
\end{equation*}
This result essentially relies on the fact that $\ee{\hat{\vy}^j|\mathcal{G}}=\overline{\hat{\vy}^j}$ regardless of the sampling strategy, due to Theorems \ref{thm:forwardsamplefilter} and \ref{thm:backwardsamplesmoother} and the uniqueness of solutions to the posterior state estimation and conditional sampling equations, i.e. for each $j=0,1,\ldots,J-1$ we have
\begin{equation*}
   \ee{\hat{\vy}^{j+1}|\mathcal{G}}=\begin{cases}
       \ee{\hat{\vy}^{j+1}\big|\cF_{t_{j+1}}^{\vx,\vy^j}}=\overline{\hat{\vy}_{\text{\nf}}^{j+1}}, & \hat{\vy}^j=\hat{\vy}_{\text{\nf}}^j, \ \mathcal{G}=\cF_{t_{j+1}}^{\vx,\vy^j}\\
       \ee{\hat{\vy}^{j+1}|\cF_T^{\vx}}=\overline{\hat{\vy}_{\text{\ns}}^{j+1}}, & \hat{\vy}^j=\hat{\vy}_{\text{\ns}}^j, \ \mathcal{G}=\cF_T^{\vx},
   \end{cases}
\end{equation*}
with this equaling correspondingly either the filter or smoother posterior mean at $t_{j+1}$.
\end{proof}

This corresponds to the bias-variance decomposition \cite{kohavi1997bias}, where the first term in \eqref{eq:biasvariancetrade} represents the bias and the second term captures the variance. Notably, the bias component is precisely the square of the traditional SRMSE in the posterior mean time series \eqref{eq:rmse}. The variance term quantifies the additional uncertainties in the sampled trajectory, beyond the variability of the posterior mean time series, which arise from the corresponding posterior covariance tensor—whether from the filter or smoother—and are often overlooked in applications that consider only the posterior mean.

The next objective is to examine the difference in anomaly pattern correlation between the posterior mean time series and the sampled trajectories, as presented in the following theorem. Note that for this result we require that our scheme is accurate enough, in other words the time step is assumed to be small.
\begin{thm}[\textbf{Anomaly Pattern Correlation of Sampled Trajectories}] \label{thm:corrsampling}
    Assuming $\dt$ is sufficiently small, then the Corr coefficient between the sampled trajectory $\hat{\vy}$ and the truth $\vy$ is given by
    \begin{align}
    \begin{split}
        \mathrm{Corr}(\vy,\hat{\vy})&=\frac{\sum_{j=0}^J \big\langle \vy^j-\langle \vy\rangle,\hat{\vy}^j-\langle \hat{\vy}\rangle\big\rangle_2}{\sqrt{\sum_{j=0}^J \left\|\vy^j-\langle \vy\rangle\right\|_2^2}\sqrt{\sum_{j=0}^J \left\|\hat{\vy}^j-\langle \hat{\vy}\rangle\right\|_2^2}}\\
        &=\frac{\sum_{j=0}^J \Big(\left\langle \vy^j-\langle \vy\rangle,\hat{\vy}^{\prime,j}-\langle \hat{\vy}^{\prime}\rangle\right\rangle_2+\big\langle \vy^j-\langle \vy\rangle,\overline{\hat{\vy}^j}-\langle \overline{\hat{\vy}}\rangle\big\rangle_2 \Big)}{\sqrt{\sum_{j=0}^J \big\|\vy^j-\langle \vy\rangle\big\|_2^2}\sqrt{\sum_{j=0}^J \Big(\big\|\overline{\hat{\vy}^j}-\big\langle \overline{\hat{\vy}}\big\rangle\big\|_2^2+\big\|\hat{\vy}^{\prime,j}-\langle\hat{\vy}^{\prime}\rangle\big\|_2^2+2\mathrm{Re}\big(\big\langle \overline{\hat{\vy}^j}-\big\langle \overline{\hat{\vy}}\big\rangle, \hat{\vy}^{\prime,j}-\langle\hat{\vy}^{\prime}\rangle\big\rangle_2\big)\Big)}}\\
        &=\eta \cdot \mathrm{Corr}\big(\vy,\overline{\hat{\vy}}\big),
    \end{split}\label{eq:correrror}
    \end{align}
    where the constant $0\leq \eta\leq 1$ is defined by
    \begin{equation*}
       \eta := \frac{\sqrt{\sum_{j=0}^J \big\|\overline{\hat{\vy}^j}-\langle \overline{\hat{\vy}}\rangle\big\|_2^2}}{\sqrt{\sum_{j=0}^J \Big(\big\|\overline{\hat{\vy}^j}-\big\langle \overline{\hat{\vy}}\big\rangle\big\|_2^2+\big\|\hat{\vy}^{\prime,j}-\langle\hat{\vy}^{\prime}\rangle\big\|_2^2+2\mathrm{Re}\big(\big\langle \overline{\hat{\vy}^j}-\big\langle \overline{\hat{\vy}}\big\rangle, \hat{\vy}^{\prime,j}-\langle\hat{\vy}^{\prime}\rangle\big\rangle_2\big)\Big)}},
    \end{equation*}
    and $\hat{\vy}^{\prime,j}=\hat{\vy}^j-\overline{\hat{\vy}^j}$ is the residual part of the sampled trajectory related to the posterior mean (where the overline denotes the conditional expectation with respect to the associated $\sigma$-algebra depending on the sampling strategy).
\end{thm}
\begin{proof}[\textbf{Proof of Theorem \ref{thm:corrsampling}}]
    Simple algebra and the bilinearity of the Euclidean inner product give the following identity:
\begin{equation*}
    \big\langle \vy^j-\langle \vy \rangle, \hat{\vy}^j-\langle \hat{\vy} \rangle \big\rangle_2 = \big\langle \vy^j - \langle \vy \rangle, \hat{\vy}^{\prime,j} - \langle \hat{\vy}^{\prime} \rangle \big\rangle_2 + \big\langle \vy^j - \langle \vy \rangle, \overline{\hat{\vy}^j} - \langle \overline{\hat{\vy}} \rangle \big\rangle_2.
\end{equation*}
Using an appropriate polarization identity, similar to that used in the proof of Theorem \ref{thm:rmsesampling}, we recover the first equality. Next, observe that the time series of $\hat{\vy}^\prime$, according to \eqref{eq:forwresidualeqn} for the forward sampler and \eqref{eq:backresidualeqn} for the backward sampler, satisfies a conditionally linear Gaussian SDE, which is conditionally independent of the truth $\vy$ (due to the independence of the stochastic forcings, i.e., their respective stochastic forcings are mutually independent Wiener processes), that also satisfies a conditionally linear Gaussian SDE. Hence, by conditional independence in the temporal direction, $\mathrm{Corr}(\vy, \hat{\vy}') \approx 0$, granted $\dt$ is sufficiently small (or $J$ is large enough) such that the sample-based quantity converges to the population one under the law of large numbers, with the convergence rate being the parametric one of $O(\sqrt{J})$; this is because even though the population-counterpart of their correlation is zero, this does not translate to the sample-based metric as well, unless $\dt$ is sufficiently small. This fact is used to derive the second equality of \eqref{eq:correrror}.
\end{proof}

Theorem \ref{thm:corrsampling} shows that the anomaly pattern correlation between the sampled trajectory and the truth, $\mathrm{Corr}(\vy, \hat{\vy})$, equals that between the truth and the posterior mean time series, $\mathrm{Corr}\big(\vy, \overline{\hat{\vy}}\big)$, up to a multiplicative factor $\eta$ which depends on the filter or smoother sampling uncertainty $\hat{\vy}^\prime$. As this uncertainty increases, the denominator in $\eta$ also increases, making $\eta$ smaller. Since $\eta \leq 1$, we have $|\mathrm{Corr}(\vy, \hat{\vy})| \leq |\mathrm{Corr}\big(\vy, \overline{\hat{\vy}}\big)|$, as expected. Only in the limiting case where the posterior uncertainty approaches zero does $\mathrm{Corr}(\vy, \hat{\vy})$ converge to $\mathrm{Corr}\big(\vy, \overline{\hat{\vy}}\big)$.

\subsection{The Temporal Autocorrelation Function} \label{sec:3.5}
The temporal autocorrelation function (ACF) measures how a signal correlates with a delayed version of itself as a function of delay \cite{gardiner2009stochastic}. For a stochastic process $\mathbf{u}$ over $[0,T]$, its scalar temporal ACF is defined as
\begin{equation} \label{eq:temporalacf}
    \mathrm{ACF}_{\mathbf{u}}(s):=\lim_{T'\to T} \frac{1}{T'}\int_0^{T'} \frac{\mathrm{tr}\left(\mathrm{Cov}(\mathbf{u}(t),\mathbf{u}(t+s))\right)}{\mathrm{tr}\left(\mathrm{Var}(\mathbf{u}(t))\right)} \rmd t.
\end{equation}
This scalar ACF focuses on "diagonal'' behavior, assuming negligible inter-coordinate correlations \cite{puccetti2022measuring}. A similar ACF in both the perfect and approximate models implies similar dynamics, at least in second-order statistics, quantified via spectral representations \cite{qi2016predicting}. For CTNDSs with extreme events, higher-order statistics are crucial, so the ACF serves as a rough measure of predictability.

To quantify system memory and dynamics recovery, we compare the ACFs of sampled trajectories to the truth. Assuming weak- or wide-sense stationarity in the CGNS dynamics (e.g., it is sufficient to assume that the CGNS of processes enjoys (geometric) ergodicity \cite{mattingly2002ergodicity, majda2016ergodicity, chen2018rigorous}), the ACFs for the truth and posterior mean time series, as $T\to +\infty$, are
\begin{align}
    \mathrm{ACF}_{\vy}(s)&:=\frac{\big\langle\mathrm{tr}\big(\ee{(\vy(t)-\overline{\vy}_{\infty})(\vy(t+s)-\overline{\vy}_{\infty})^\dagger}\big)\big\rangle}{\mathrm{tr}\left(\mathrm{Var}(\vy)_{\infty}\right)} \label{eq:acftruth},\\
    \mathrm{ACF}_{\overline{\hat{\vy}}}(s)&:=\frac{\big\langle\mathrm{tr}\big(\mathbb{E}\big[\big(\overline{\hat{\vy}}(t)-\overline{\hat{\vy}}_{\infty}\big)\big(\overline{\hat{\vy}}(t+s)-\overline{\hat{\vy}}_{\infty}\big)^\dagger\big]\big)\big\rangle}{\mathrm{tr}\left(\mathrm{Var}\big(\overline{\hat{\vy}}\big)_{\infty}\right)} \label{eq:acfpostmean},
\end{align}
where $\overline{\vy}_{\infty}$ and $\overline{\hat{\vy}}_{\infty}$ are the equilibrium means, and $\mathrm{Var}(\vy)_{\infty}$ and $\mathrm{Var}\big(\overline{\hat{\vy}}\big)_{\infty}$ are the equilibrium covariances; we highlight that the expectations and variances taken in this regime, in \eqref{eq:acftruth} and \eqref{eq:acfpostmean}, are the total ones since we consider the dynamics at the equilibrium. In practice, the sample covariance tensor is used to compute the ACF. These expressions, as well as \eqref{eq:acfsample} in Theorem \ref{thm:acfsampling}, can be derived using Euclidean norms, the interchangeability of the expectation and trace operators, and the cyclic property of the trace. The following theorem compares the ACF of the posterior mean time series with that of the sampled trajectory.

\begin{thm}
    [\textbf{Temporal ACF of Samples}] \label{thm:acfsampling}
    The temporal ACF of the sampled trajectory $\hat{\vy}$ can be calculated as
    \begin{align}
    \begin{split}
        \mathrm{ACF}_{\hat{\vy}}(s)&:=\frac{\big\langle\mathrm{tr}\big(\ee{(\hat{\vy}(t)-\overline{\hat{\vy}}_{\infty})(\hat{\vy}(t+s)-\overline{\hat{\vy}}_{\infty})^\dagger}\big)\big\rangle}{\mathrm{tr}\left(\mathrm{Var}\big(\hat{\vy}\big)_{\infty}\right)}\\
        &=\frac{\big\langle\mathrm{tr}\big(\ee{(\hat{\vy}'(t)+\overline{\hat{\vy}}(t)-\overline{\hat{\vy}}_{\infty})(\hat{\vy}'(t+s)+\overline{\hat{\vy}}(t+s)-\overline{\hat{\vy}}_{\infty})^\dagger}\big)\big\rangle}{\mathrm{tr}\left(\mathrm{Var}\big(\hat{\vy}'+\overline{\hat{\vy}}\big)_{\infty}\right)}\\
        &=\frac{\big\langle\mathrm{tr}\big(\mathbb{E}\big[\big(\overline{\hat{\vy}}(t)-\overline{\hat{\vy}}_{\infty}\big)\big(\overline{\hat{\vy}}(t+s)-\overline{\hat{\vy}}_{\infty}\big)^\dagger\big]\big)\big\rangle}{\mathrm{tr}\left(\mathrm{Var}(\hat{\vy}')_{\infty}\right)+\mathrm{tr}\left(\mathrm{Var}\big(\overline{\hat{\vy}}\big)_{\infty}\right)}+\frac{\big\langle\mathrm{tr}\big(\ee{\hat{\vy}'(t)\hat{\vy}'(t+s)^\dagger}\big)\big\rangle}{\mathrm{tr}\left(\mathrm{Var}(\hat{\vy}')_{\infty}\right)+\mathrm{tr}\left(\mathrm{Var}\big(\overline{\hat{\vy}}\big)_{\infty}\right)}\\
        &=\beta_1 \mathrm{ACF}_{\overline{\hat{\vy}}}(s)+\beta_2 \mathrm{ACF}_{\hat{\vy}'}(s),
    \end{split}\label{eq:acfsample}
    \end{align}
    where the constants $0\leq \beta_1,\beta_2\leq 1$ are given by
    \begin{align*}
       \beta_1 &= \frac{\mathrm{tr}\left(\mathrm{Var}\big(\overline{\hat{\vy}}\big)_{\infty}\right)}{\mathrm{tr}\left(\mathrm{Var}(\hat{\vy}')_{\infty}\right)+\mathrm{tr}\left(\mathrm{Var}\big(\overline{\hat{\vy}}\big)_{\infty}\right)},\\
       \beta_2 &= \frac{\mathrm{tr}\left(\mathrm{Var}(\hat{\vy}')_{\infty}\right)}{\mathrm{tr}\left(\mathrm{Var}(\hat{\vy}')_{\infty}\right)+\mathrm{tr}\left(\mathrm{Var}\big(\overline{\hat{\vy}}\big)_{\infty}\right)},
    \end{align*}
    and $\beta_1+\beta_2=1$, i.e., $\mathrm{ACF}_{\hat{\vy}}(s)$ is a convex linear combination of $\mathrm{ACF}_{\overline{\hat{\vy}}}(s)$ and $\mathrm{ACF}_{\hat{\vy}'}(s)$.
\end{thm}
\begin{proof}[\textbf{Proof of Theorem \ref{thm:acfsampling}}]
    The definition of the temporal ACF for the sampled trajectories again follows by the consistency Theorems \ref{thm:forwardsamplefilter} and \ref{thm:backwardsamplesmoother}. Simple algebra and using the interchangeability of the expectation and trace operators, along with the stability property of the conditional expectation, measurability of the posterior statistics, wide-sense stationarity of the CGNS, and the Fubini-Tonelli theorem, yields the second and third equality, since the term
    \begin{gather*}
        \hspace{-2cm}\big\langle\mathrm{tr}\big(\mathbb{E}\big[\hat{\vy}'(t)\big(\overline{\hat{\vy}}(t+s)-\overline{\hat{\vy}}_{\infty}\big)^\dagger\big]\big)\big\rangle+\big\langle\mathrm{tr}\big(\mathbb{E}\big[\big(\overline{\hat{\vy}}(t)-\overline{\hat{\vy}}_{\infty}\big)\hat{\vy}'(t+s)^\dagger\big]\big)\big\rangle\\
        \hspace{5cm}=\mathbb{E}\big[\big\langle\overline{\hat{\vy}}(t+s)^\dagger\hat{\vy}'(t)+\hat{\vy}'(t+s)^\dagger\overline{\hat{\vy}}(t)\big\rangle\big]
    \end{gather*}
    vanishes under these conditions. Specifically, under the aforementioned results and conditions, we have
    \begin{equation*}
        \mathbb{E}\big[\big\langle\overline{\hat{\vy}}(t+s)^\dagger\hat{\vy}'(t)\big\rangle\big]=\big\langle\mathbb{E}\big[\overline{\hat{\vy}}(t+s)^\dagger\mathbb{E}\big[\hat{\vy}'(t)\big|\mathcal{G}\big]\big]\big\rangle=0,
    \end{equation*}
    and likewise for $\mathbb{E}\big[\hat{\vy}'(t+s)^\dagger\overline{\hat{\vy}}(t)\big\rangle\big]$. Notice how in the third equality we also used the fact that
    \begin{equation*}
        \mathrm{tr}\left(\mathrm{Var}\big(\hat{\vy}'+\overline{\hat{\vy}}\big)_{\infty}\right)=\mathrm{tr}\left(\mathrm{Var}(\hat{\vy}')_{\infty}\right)+\mathrm{tr}\left(\mathrm{Var}\big(\overline{\hat{\vy}}\big)_{\infty}\right)+2\mathrm{tr}\left(\mathrm{Cov}\big(\hat{\vy}',\overline{\hat{\vy}}\big)_{\infty}\right),
    \end{equation*}
    where $\mathrm{tr}\left(\mathrm{Cov}\big(\hat{\vy}',\overline{\hat{\vy}}\big)_{\infty}\right)=0$, again by the stability property of the conditional expectation and measurability of the posterior mean (by following a similar procedure as we did with the expectation). 
\end{proof}

It is important to note that the temporal ACF of $\hat{\vy}'$ is not zero because $\hat{\vy}'$ satisfies a dynamical equation with a finite damping rate (see \eqref{eq:forwresidualeqn} and \eqref{eq:backresidualeqn}), for the forward and backward sampling algorithms, respectively. This is very different from the anomaly pattern correlation coefficient in \eqref{eq:correrror}, for which $\hat{\vy}'$ has no contribution (apart from the multiplicative factor). This shows why using the ACF is preferred since it incorporates uncertainty in the residual, which does not happen in lower-order statistics like Corr. It is also worth noting that the integration of the ACF is likewise a crucial quantity and is named the decorrelation time \cite{harlim2008mathematical, gardiner2009stochastic}, which measures the memory of the system. The decorrelation time of the sampled trajectory can thus be smaller, equal, or larger than that of the posterior mean time series, leading to nontrivial results.

All these arguments apply regardless of which sampling strategy we use, forward or backward. But, the ACF of the backward sampled trajectories should intuitively almost perfectly reproduce the truth while that associated with the filter or smoother posterior mean time series may contain significant biases, as we investigate in the numerical case study in Section \ref{sec:4}. This also justifies the fact that the posterior mean time series is not dynamically consistent with the truth, while the sampled trajectories take into account the uncertainty and its temporal dependence, which is a more suitable way as a pathwise surrogate for the true signal.

\section{Numerical Case Study: A Nonlinear Physics-Constrained Stochastic Reduced-Order Model} \label{sec:4}

In this section, we showcase an application of the CGNS framework through its optimal state estimation and sampling procedures. The main goal of this case study is to demonstrate how the optimal nonlinear smoother state estimation, and associated optimal nonlinear backward samples, are potent in the recovery of the non-Gaussian intermittent features present in the unobservable components of a partially observed model that fits the conditional Gaussian structure, and how they contrast their filter-based counterparts in doing so. This is established through a reduced-order triad-interaction climate model with cubic nonlinearity in the observables, as well as multiplicative and cross-interacting noise. Such a complex model also showcases the disruption to the commonly assumed to be true hierarchy in the sample-to-sample fluctuation component of the latent uncertainty which was discussed in Section \ref{sec:3.3}.

\subsection{The Model}  \label{sec:4.1}
The following nonlinear stochastic triad-system mimics structural features of low-frequency variability in general circulation models with non-Gaussian features \cite{majda2009normal}, where the observable variables comprise the low-frequency modes, while the hidden variables define the climate modes of high-frequency. This model involves extremely nontrivial interactions both in the deterministic as well as stochastic dynamics, with quadratic interactions between the state variables, a cubic nonlinearity in the observable, correlated additive and multiplicative (CAM) noise, and noise cross-interaction. Furthermore, the quadratic dyad-interaction nonlinear terms, which model the nonlinear advection, conserve the quadratic energy of the system, thus imposing physical constraints on the system, which also induce the intermittent instability onto $u_1$ by functioning as antidamping terms. This model is presented in what follows and fits the CGNS setting outlined in \eqref{eq:condgauss1}--\eqref{eq:condgauss2} for $\vx:=u_1$ and $\vy:=(u_2,u_3)^\tran$:
\begin{align}
    \begin{split}
        \rmd u_1&=\left(\gamma_1u_1+I_{12}u_1u_2+I_{13}u_1u_3-cu_1^3+L_{12}u_2+L_{13}u_3+F_1(t)\right)\rmd t\\
    &\hspace{1cm}+\sigma_1\rmd W_{u_1}+\frac{\sigma_2}{\gamma_2}(L_{12}-I_{12}u_1)\rmd W_{u_2}+\frac{\sigma_3}{\gamma_3}(L_{13}-I_{13}u_1)\rmd W_{u_3},
    \end{split} \label{eq:study1}\\
    \rmd u_2&=\left(-\frac{\gamma_2}{\varepsilon}u_2-L_{12}u_1+L_{23}u_3-I_{12}u_1^2+F_2(t)\right)\rmd t+\frac{\sigma_{2}}{\sqrt{\varepsilon}}\rmd W_{u_2}, \label{eq:study2}\\
    \rmd u_3&=\left(-\frac{\gamma_3}{\varepsilon}u_3-L_{13}u_1-L_{23}u_2-I_{13}u_1^2+F_3(t)\right)\rmd t+\frac{\sigma_{3}}{\sqrt{\varepsilon}}\rmd W_{u_3}, \label{eq:study3}
\end{align}
where $\displaystyle c=\sum_{p=2}^3 \frac{I_{1p}^2}{\gamma_p}>0$, $\gamma_1,L_{12},L_{13},L_{23},I_{12},I_{13}\in\rr$, and $\gamma_2,\gamma_3,\sigma_1,\sigma_2,\sigma_3>0$. 
The parameter $\epsilon \in (0,1]$ denotes the timescale separation between the potentially slow ($u_1$) and fast ($u_2$, $u_3$) variables. CAM noise arises from applying stochastic mode reduction to a more complex system with multiscale features \cite{franzke2005low}. The cubic nonlinearity with CAM has been applied to describe several climate phenomena \cite{majda2009normal}. The coupling between $u_2$ and $u_3$ is purely linear and represented by a skew-symmetric term $\pm L_{23}$, indicating an oscillatory structure. However, nonlinear feedback from $u_1$ influences the governing equations for $u_2$ and $u_3$, with $u_2$ and $u_3$ acting as stochastic damping for $u_1$. When $u_2$ and $u_3$ increase, this anti-damping triggers extreme events in $u_1$. In turn, strong signals in $u_1$ reduce the amplitude of $u_2$ and $u_3$ via nonlinear feedback terms involving $u_1^2$, rapidly suppressing their signals. The quadratic nonlinear terms in the equations conserve energy, which acts as a physical constraint \cite{harlim2014ensemble}. This mirrors many turbulent systems, where quadratic nonlinearities transfer energy between modes or scales, while damping and external forcing dissipate and generate energy. This coupled structure also appears in turbulent ocean flows, where $u_1$ represents geostrophically balanced components, and $u_2$ and $u_3$ model inertio-gravity waves \cite{salmon1998lectures}. The large-scale forcing terms $F_p(t)$ for $p=1,2,3$ represent external inputs, such as seasonal effects, decadal oscillations, or rapid small-scale fluctuations \cite{vallis2017atmospheric, mantua2002pacific}.

In the following, a time series of $u_1$ is observed and the goal is to study the recovery of the states of $u_2$ and $u_3$. The following simulation and model parameters are adopted in this numerical test:
\begin{gather*}
 T=60, \quad \dt=10^{-3}, \quad \gamma_1=1, \quad I_{12}=0.5, \quad I_{13}=0.5, \quad L_{12}=0.5, \quad L_{13}=0.5, \quad F_1=3, \quad \sigma_1=0.5,\\
 \epsilon = 1,\quad \gamma_2 = 1.2, \quad L_{23} = 2, \quad F_2 = 0, \quad \sigma_2=1.2,\\
 \gamma_3=0.5, \quad F_3 =0, \quad \sigma_3=0.8.
\end{gather*}

\subsection{Comparison for Filtering, Smoothing, and Sampling Solutions}  \label{sec:4.2}
Figures \ref{fig:case_study_fig_1}--\ref{fig:case_study_fig_3} demonstrate the performance of filtering, smoothing, forward sampling, and backward sampling methods. Panel (a) of Figure \ref{fig:case_study_fig_1} shows the observed time series of $u_1$, which exhibits intermittency and contains many extreme events. Correspondingly, its PDF in Panel (b) is highly skewed with a one-sided fat tail. In Panels (c) and (e), the smoother state estimates are more accurate than those of the filter in two ways. First, the smoother posterior mean is much closer to the true state. Notably, the filter posterior mean misses most peak events when $u_2$ and $u_3$ are positive, which correspond to extreme events in $u_1$. Since $u_2$ and $u_3$ act as stochastic damping for $u_1$, they trigger the extreme events in $u_1$. Thus, the occurrence of intermittent instability in $u_1$ is delayed relative to its triggers, $u_2$ and $u_3$. Recovering $u_2$ and $u_3$ in real time using only past information (as in filtering) is challenging, while future observations are crucial for accurately capturing these triggers. Second, the posterior variance, which is an indicator of the estimated uncertainty, is significantly lower for the smoother, as it incorporates more observational information. Note that, if the sampled trajectories from the forward and backward sampling methods are shown, they will exactly cover the uncertainty regions illustrated in the figure, per the results of Theorem \ref{thm:forwardsamplefilter} and \ref{thm:backwardsamplesmoother}. Finally, consistent with the pathwise behavior in Panels (b) and (d), the PDF from the smoother mean is closer to the true PDF than that of the filter, as shown in Panels (c) and (e). However, the PDF from the smoother mean does not fully match the true PDF, as the smoother mean still underestimates the variability of the signal, which is captured in the posterior uncertainty. Only the trajectories from backward sampling can generate PDFs that align with the true distribution, assuming they are sufficiently long.

\begin{figure}[ht!]
\centering
    \includegraphics[width=\textwidth]{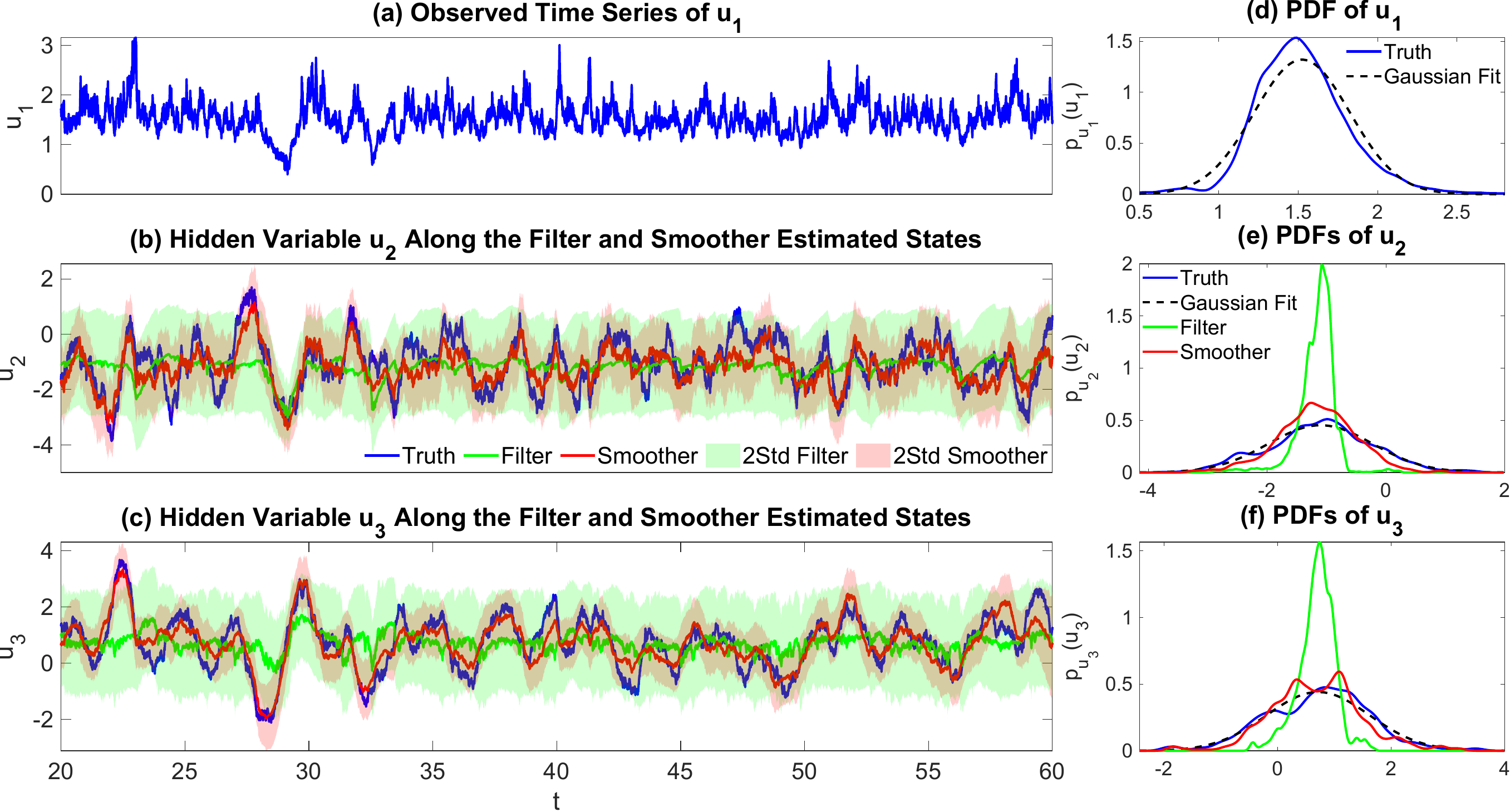}
    \caption{Performance of filtering, smoothing, forward sampling, and backward sampling methods using the reduced-order climate model in \eqref{eq:study1}--\eqref{eq:study3}. Panels (a)–(b): Observed time series of $u_1$ and its corresponding PDF. Panels (c)–(d): State estimation of $u_2$ with filtering and smoothing, along with the corresponding PDFs. Panels (e)–(f): State estimation of $u_3$ with filtering and smoothing, and the corresponding PDFs. The PDFs correspond to the full-time-length signal, i.e., for $t\in[0,60]$.}
    \label{fig:case_study_fig_1}
\end{figure}

Figure \ref{fig:case_study_fig_2} presents a statistical comparison of state estimation using different approaches. Consistent with the findings on PDF recovery, the smoother mean captures the autocorrelation function (ACF; Panels (a)–(d)) and power spectral density (PSD; Panels (e)–(h)) more effectively than the filtering mean. The filtering mean misses many peak events, leading to noticeable errors in both the decay rate during the initial period of the ACF and the low-frequency PSD. Similarly, trajectories from forward sampling, based on the filter estimate, exhibit comparable issues in their ACF and PSD (not shown here). In contrast, the smoother mean accurately captures the statistics with much smaller errors. The trajectories from backward sampling further account for the variability provided by the smoother variance. They are dynamically consistent with the true values and can yield unbiased statistics, provided they are sufficiently long. The numerical evidence presented here verifies the analysis in Section \ref{sec:3.5}. Panels (i)–(j) display the correlation coefficients between the true time series and the estimates from different methods. The smoother mean time series shows a higher correlation with the true series than the filter mean, for both regular and extreme events. Although the sampled trajectories demonstrate improved statistical capture, their correlations with the truth are weaker. This is not surprising, as the enhancement in statistical accuracy arises from incorporating additional random variability from the posterior uncertainty, which naturally decreases the pathwise similarity to the true series. The conclusion is also consistent with the analysis in Section \ref{sec:3.4} (specifically the observation that $|\mathrm{Corr}(\vy, \hat{\vy})| \leq |\mathrm{Corr}\big(\vy, \overline{\hat{\vy}}\big)|$, since $\eta \leq 1$).

\begin{figure}[ht!]
\centering
    \includegraphics[width=\textwidth]{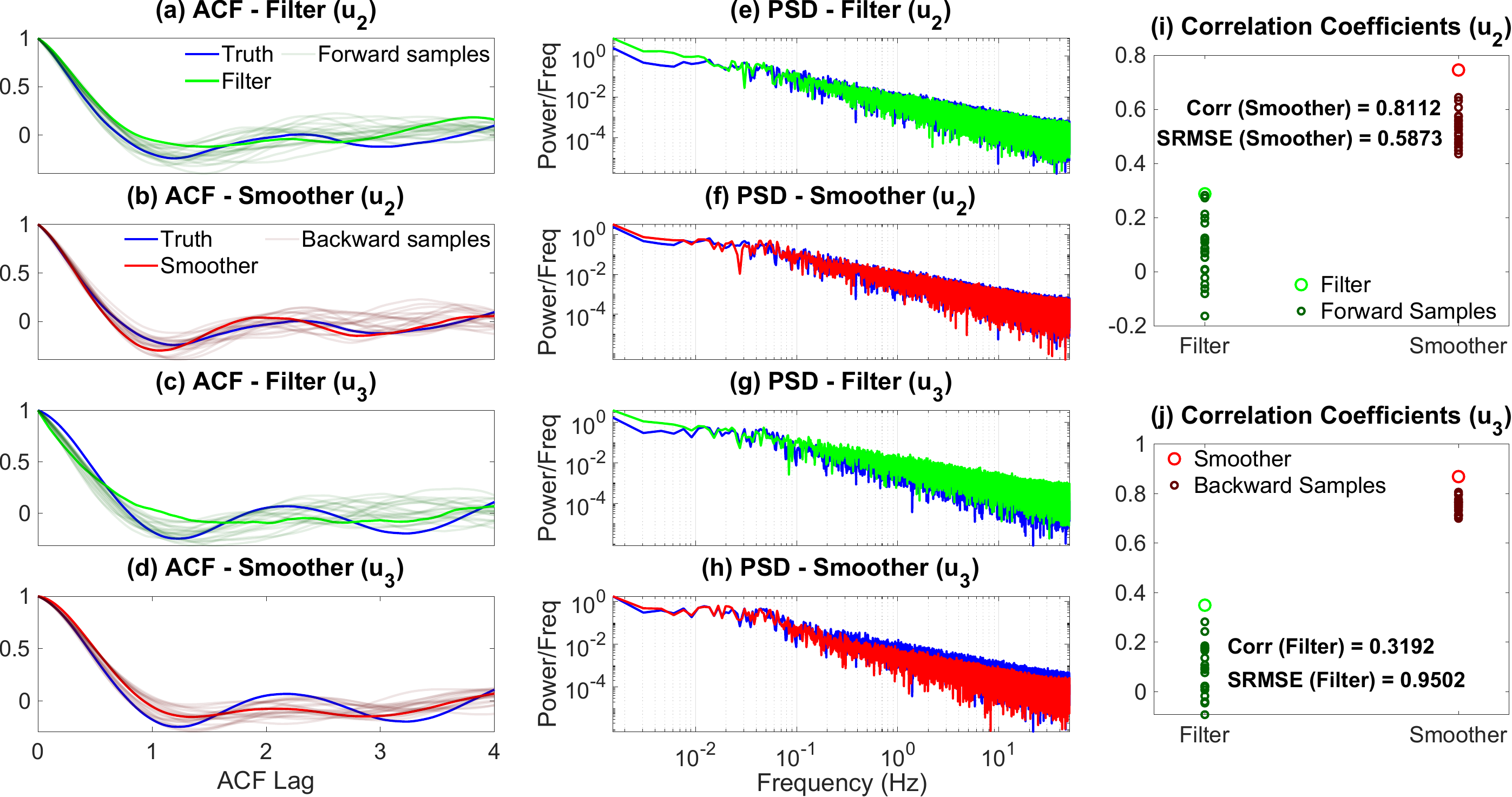}
    \caption{Statistical comparison between the state estimation using different approaches. Panels (a)–(d): Comparison of the autocorrelation function (ACF). Panels (e)–(h): Comparison of the power spectral density (PSD). Panels (i)–(j): Correlation coefficient between the true time series and the recovered ones from different methods. All comparisons are based on a time series of 60 time units. The SRMSE leads to a similar conclusion as the Corr and is therefore omitted here. }
    \label{fig:case_study_fig_2}
\end{figure}

\subsection{Analysis of the Uncertainty Levels in Different State Estimation Methods} \label{sec:4.3}

Finally, in Figure \ref{fig:case_study_fig_3}, we show the time evolution of the expected maximal eigenvalues for the damping (Panels (a)--(b)) and fluctuation (Panels (c)--(d)) components of the sampling uncertainty emerging in the unobservable variables. This figure should be understood based on the discussions following Theorems \ref{thm:forwardsample} and \ref{thm:backwardsample}, and by extending the exposition provided in Section \ref{sec:3.3}. For reference, the unconditional damping and noise feedbacks are extracted from the time evolution equation for the latent dynamics of the CGNS, \eqref{eq:condgauss2}, while the damping and fluctuation components for the forward filter-based and backward smoother-based sampling procedures can be retrieved from their corresponding residual equations, \eqref{eq:forwresidualeqn} and \eqref{eq:backresidualeqn}, by looking at their drift and diffusion coefficients, respectively. Note that the x-axis (time axis) is plotted in the logarithm scale. The solution will arrive near the equilibrium within a short period, regardless of the initial value. All subsequent discussions are based on the equilibrium state.

We start with Panels (c)--(d). Since the backward-sampling noise feedback is always no larger than the unconditional and filter-based ones (in the Loewner partial ordering sense), the maximal eigenvalues for $\mb$ will at all times be less than the respective ones for $\ms^\vy\circ\ms^\vy$ and $\mb+\mr{\nf}\mc\mr{\nf}$, with all being nonnegative due to the nonnegative-definiteness of the noise amplitudes. Yet, as already alluded in Sections \ref{sec:3.1} and \ref{sec:3.3}, it is not always that this hierarchy holds as well between these components for the filter-based sampling and unconditional forward run of the system, as is indicated in Panel (d). Since the black curve, corresponding to the expected minimum eigenvalue of $\ms^\vy\circ\ms^\vy-\mb-\mr{\nf}\mc\mr{\nf}$, goes below zero, this explicitly shows that the difference between these matrices is an indefinite matrix after the fast and exponentially-reached relaxation time. However, it is known that the overall uncertainty in the optimal filter estimated state (e.g., the posterior variance) is all times smaller than the unconditional or marginal uncertainty of $\vy=(u_2,u_3)^\tran$. So while the results shown in Panel (d) might at first seem counterintuitive, it is essential to observe that the correction term in \eqref{eq:filter2}, i.e., the terms other than $\ms^\vy\circ\ms^\vy$, which accounts for the observance of data stemming from $\vx$, does not just depend on the optimal Kalman gain interacting with the uncertainties in the innovation, but is also influenced by the damping components which help, under the Bayesian framework, to mitigate the uncertainty in the latent state variable, regardless of the quality of the current measurement. These corrective terms appear in the time evolution of the posterior covariance matrices and help reduce the overall uncertainty. They correspond to nontrivial damping feedbacks, which we showcase in Panels (a)--(b) for both the forward and backward sampling and the unconditional dynamics. Note that the damping feedback in the forward sampler is exactly the damping that helps stabilize the filter mean per \eqref{eq:alternativefiltermean}, while the smoother-based one is the damping which controls the smoother covariance per \eqref{eq:revbackinter2}, instead. As such, the overall uncertainty, while consistently decreasing with increasing incorporation of information, its constituents might have nontrivial temporal profiles and seemingly showcase counterintuitive results if only a marginal component of the uncertainty is taken into consideration. This is precisely illustrated by the fact that now the filter-based damping is always stronger compared to the unconditional one, which under a bias-variance trade-off-type argument reveals why the uncertainty indeed decreases by forward sampling the unobservable variables. Specifically, note how when the filter-based noise feedback's maximal eigenvalues go above the unconditional ones, the respective damping component of the uncertainty increases to counteract this. As such, while the observation incorporation might increase the sample-to-sample fluctuations, the overall unpredictability in the estimated state will nonetheless diminish due to the strengthened damping.

\begin{figure}[ht!]
\centering
   \includegraphics[width=\textwidth]{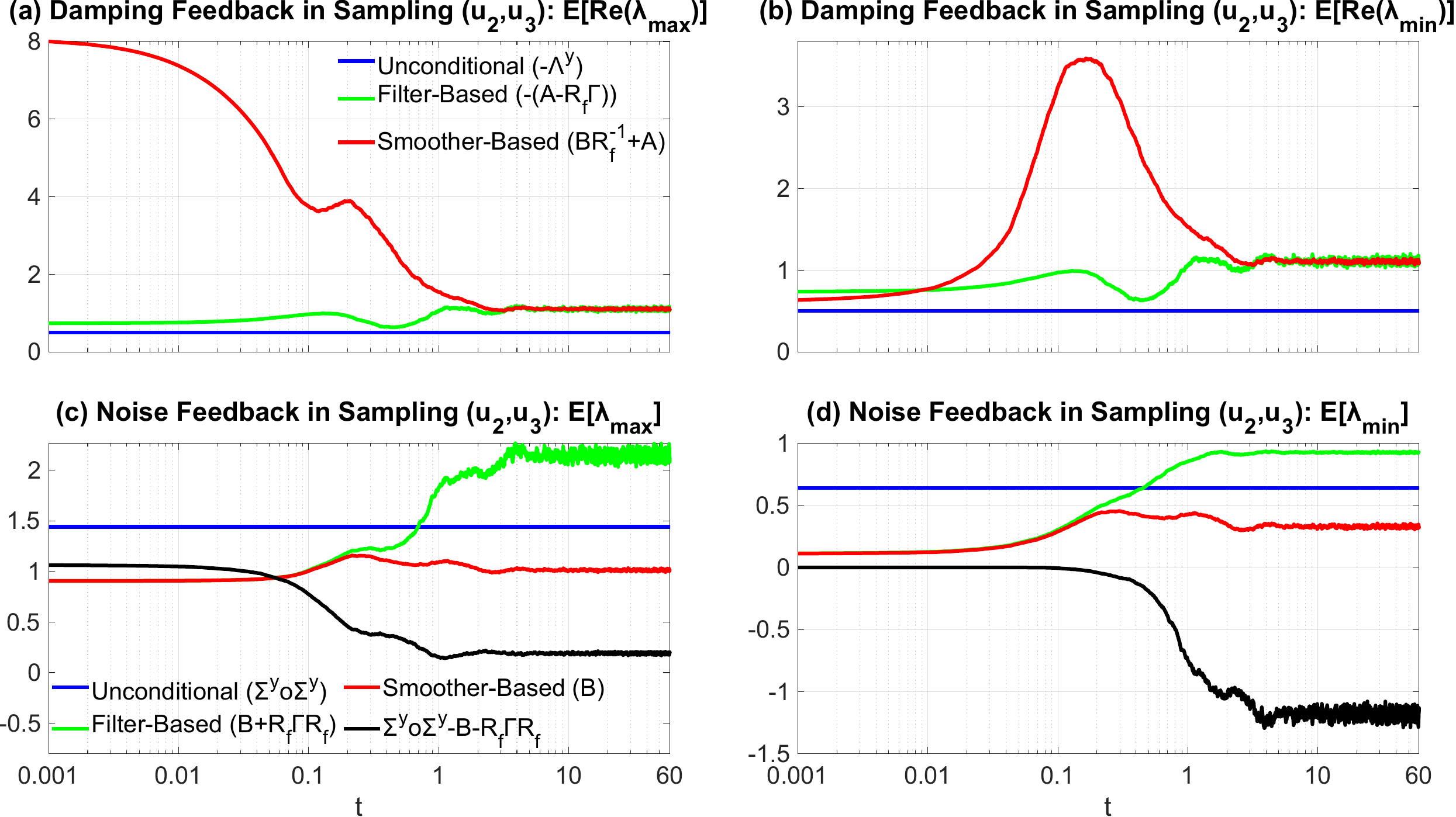}
   \caption{Time evolution of the spectrum for the components of the sample-to-sample uncertainty with respect to the different procedures of simulating the unobservable variables. Panels (a)--(b): Comparison with regards to the expected real part of the maximal and minimal eigenvalues of the damping feedback in the unconditional forward run of \eqref{eq:study2}--\eqref{eq:study3} and the conditional forward and backward optimal nonlinear sampling of $(u_2,u_3)$. The x-axis is logarithmically scaled. Panels (c)--(d): Same as (a)--(b), but instead concerning the fluctuation part of the uncertainty, i.e., the noise feedback matrices. The time evolution of the expected maximal and minimal eigenvalues is also shown for the matrix difference between the unconditional and filter-based sampling noise coefficients.}
   \label{fig:case_study_fig_3}
\end{figure}


\section{Conclusion} \label{sec:5}
This paper develops a martingale-free approach that facilitates the understanding of CGNS. The martingale-free method provides a tractable approach to proving the time evolution of conditional statistics by deriving results through time discretization schemes, with the continuous-time regime obtained via a formal limiting process as the discretization time step approaches zero. This has the consequence of a natural unification between the two settings. Furthermore, this discretized approach further allows for the development of analytic formulae for optimal posterior sampling algorithms of the unobserved state variables with correlated noise, which is crucial in settings where uncertainty from one subset of the phase space feeds into the other. These analytic tools are particularly valuable for studying extreme events and intermittency and apply to high-dimensional systems. Moreover, the approach improves the understanding of different state estimation and sampling methods in characterizing the uncertainty. The effectiveness of the methods is demonstrated through a physics-constrained, triad-interaction reduced-order climate model with cubic nonlinearity and state-dependent cross-interacting noise. The model simulation helps understand the skill of state estimation, especially for recovering the highly non-Gaussian features of intermittency and extreme events.  

\section*{Acknowledgement}
N.C. is grateful to acknowledge the support of the Office of Naval Research (ONR) N00014-24-1-2244 and Army Research Office (ARO) W911NF-23-1-0118. M.A. is supported as a research assistant under these grants.

\clearpage

\section{Appendix}

\subsection{Preliminary Definitions and Lemmas} \label{sec:Appendix_A}
To prove most of the results in this work, we first need to define the so-called Reynolds decomposition of a stochastic quantity.
\begin{defn}[\textbf{Mean-Fluctuation Decomposition}] \label{def:reynoldsdecomp}
    The {mean-fluctuation decomposition} of a complex random variable $\mathbf{z}$ at a fixed time instant, also known as the {Reynolds decomposition}, is given by \cite{vallis2017atmospheric, adrian2000analysis}:
    \begin{equation} \label{eq:reynoldsdecomp}
        \mathbf{z}=\overline{\mathbf{z}}+\mathbf{z}^\prime,
    \end{equation}
    where $\overline{\mathbf{z}}=\ee{\mathbf{z}}$ is the ensemble average or mean of $\mathbf{z}$ and $\mathbf{z}^\prime$, with $\overline{\mathbf{z}^\prime}=\zeros$, is the residual or the deviations around the expectation value (or fluctuations). The covariance tensor of $\mathbf{z}$ is given compactly by $\overline{\mathbf{z}^\prime(\mathbf{z}^\prime)^\dagger}$.
\end{defn}

Furthermore, we state the following three necessary lemmas that are extensively used to establish the assertions stated in this article. The first one concerns itself with how affine transformations affect Gaussian distributions.
\begin{lem}[\textbf{Affine Transformation of Normal Variable}] \label{lem:affinity}
      The distribution of a multivariate Gaussian random variable is invariant under affine transformations. In other words, for any constant vector $\mathbf{c}$ and constant nonzero matrix $\mathbf{G}$, if $\mathbf{z}\overset{\rmd}{\sim}\mathcal{N}_{\mathrm{dim}(\vz)}(\vm{},\mr{})$, then $\mathbf{G}\mathbf{z}+\mathbf{c}\overset{\rmd}{\sim}\mathcal{N}_{\mathrm{dim}(\mathbf{G}\vz)}(\mathbf{G}\vm{}+\mathbf{c},\mathbf{G}\mr{}\mathbf{G}^\dagger)$.
\end{lem}
\begin{proof}[\textbf{Proof of Lemma \ref{lem:affinity}}] This result is proved in Proposition 3.1 of Eaton \cite{eaton2007multivariate}.
\end{proof}

The second one explicitly showcases how the Bayesian update of a conditional and marginal normal behaves. In this regard, and for the remainder of this work, we denote the PDF of a multivariate $\mathrm{dim}(\mathbf{u})$-dimensional Gaussian distribution with mean $\vm{}$ and covariance matrix $\mr{}$ by $\mathcal{N}_{\mathrm{dim}(\mathbf{u})}\left(\mathbf{u};\vm{},\mr{}\right)$.
\begin{lem}[\textbf{Bayesian Update of a Gaussian}] \label{lem:bayesian}
    For two Gaussian distributions, with $k=\mathrm{dim}(\mathbf{u}_1)$ and $l=\mathrm{dim}(\mathbf{u}_2)$, we have
    \begin{equation*}
        \int_{\mathbb{C}^{k}} \mathcal{N}_l\left(\mathbf{u}_2;\mathbf{F}\mathbf{u}_1+\mathbf{b},\mr{$2$}\right)\mathcal{N}_k\left(\mathbf{u}_1;\vm{$1$},\mr{$1$}\right)\rmd\mathbf{u}_1=\mathcal{N}_l\big(\mathbf{u}_2;\mathbf{F}\vm{$1$}+\mathbf{b},\mathbf{F}\mr{$1$}\mathbf{F}^{\dagger}+\mr{$2$}\big),
    \end{equation*}
    for appropriately sized vectors and matrices, i.e., $\mathbf{b}$ and $\mathbf{F}$ are assumed to be conformable.
\end{lem}
\begin{proof}[\textbf{Proof of Lemma \ref{lem:bayesian}}] This result is proved in Section 2.3.3 of Bishop \cite{christopher2016pattern} (formulated through precision matrices, with the precision matrix defined as the inverse of the covariance matrix).
\end{proof}

Finally, the third lemma, also known as the theorem on normal correlations, is a fundamental result in conditional autoregressive modeling. Here we use the notation $p(\boldsymbol{\cdot})$ to denote the law or density (PDF), conditional or not, of the probability distribution $\pp(\boldsymbol{\cdot})$ corresponding to the random variable or vector of interest.
\begin{lem}[\textbf{Conditional Distribution of Joint Normal or Theorem on Normal Correlation}] \label{lem:conditional}
    Let the following random vector be Gaussian:
    \begin{equation*}
        \vz=(\vz_1^\tran,\vz_2^\tran)^\mathtt{T},
    \end{equation*}
    with mean $\vm{}$ and covariance $\mr{}$, where
    \begin{equation*}
        \vm{}=\begin{pmatrix}
            \vm{$1$}\\\vm{$2$}
        \end{pmatrix}, \quad \mr{}=\begin{pmatrix}
            \mr{$11$} & \mr{$12$}\\
            \mr{$21$} & \mr{$22$}
        \end{pmatrix}.
    \end{equation*}
    Then the following conditional distribution is Gaussian:
    \begin{equation*}
        \pp\big(\vz_1\big|\vz_2=\boldsymbol{\zeta}_2\big)\overset{\rmd}{\sim}\mathcal{N}_{\mathrm{dim}(\vz_1)}\left(\widetilde{\vm{}},\widetilde{\mr{}}\right),
    \end{equation*}
    with a mean dependent on the mean and state of the conditioner, while the covariance matrix is independent of it, or,
    \begin{align*}
        \widetilde{\vm{}}&=\vm{$1$}+\mr{$12$}\mathbf{R}^{-1}_{22}(\boldsymbol{\zeta}_2-\vm{$2$}), \\
        \widetilde{\mr{}}&= \mr{$11$}-\mr{$12$}\mathbf{R}^{-1}_{22}\mr{$21$}.
    \end{align*}
\end{lem}
\begin{proof}[\textbf{Proof of Lemma \ref{lem:conditional}}] This is proved in Section 4.1 of Eaton \cite{eaton2007multivariate} and Sections 2.3.1--2.3.2 of Bishop \cite{christopher2016pattern}. It also corresponds to Theorems 13.1 and 13.2 in Liptser \& Shiryaev \cite{liptser2001statisticsII}.
\end{proof}

\subsection{Proof of Theorem \ref{thm:filtering}} \label{sec:Appendix_B}

\begin{proof}[\textbf{Proof of Theorem \ref{thm:filtering}}]
This is Theorem 12.7 in Liptser \& Shiryaev \cite{liptser2001statisticsII}, which is the multi-dimensional analog of Theorems 12.1 and 12.3. For the analogous result in the case of discrete time, see Theorem 13.4, which outlines the corresponding optimal recursive nonlinear filter difference equations, with the respective sufficient assumptions given in Subchapter 13.2.1. Thorough details are also provided in \cite{kolodziej1980state} for the continuous-time case. Here, we instead provide a proof under our unified discrete- and continuous-time framework without the use of martingales.

We consider the discretized equations \eqref{eq:discretecondgauss1}--\eqref{eq:discretecondgauss2}. In a similar manner, we then provide the recurrence relations that the fluctuation terms satisfy when the mean-fluctuation decomposition \eqref{eq:reynoldsdecomp} is employed in \eqref{eq:condgauss1}--\eqref{eq:condgauss2} (with the expectation being the conditional one under the observation of the discrete data $\vx^0,\ldots,\vx^j$ for $j\in\{0,\ldots, J\}$):
    \begin{align}
        \vx^{\prime,j+1}&:=\vx^{j+1}-\overline{\vx^{j+1}}=\ml^{\vx,j}\vy^{\prime, j}\dt+\ms_1^{\vx,j}\sqrt{\dt} \ve_1^j+\ms_2^{\vx,j}\sqrt{\dt} \ve_2^j,\label{eq:discretefluc1} \\
        \vy^{\prime,j+1}&:=\vy^{j+1}-\overline{\vy^{j+1}}=\vy^{\prime, j}+\ml^{\vy,j}\vy^{\prime, j}\dt+\ms_1^{\vy,j}\sqrt{\dt} \ve_1^j+\ms_2^{\vy,j}\sqrt{\dt} \ve_2^j, \label{eq:discretefluc2}
    \end{align}
    where since during averaging we conditioning on $\vx^0,\ldots,\vx^j$, terms involving $\vx^j$ should not appear in the equations for the fluctuations due to the form of \eqref{eq:discretecondgauss1}--\eqref{eq:discretecondgauss2}. We now turn our attention to the joint distribution of $(\vx^{j+1},\vy^{j+1})$, when conditioning on $\vx^{s},\ s\leq j$. During the proof of Theorem 13.6 of Liptser \& Shiryaev \cite{liptser2001statisticsII} it is proved via induction and the uniqueness of conditional characteristic functions \cite{yuan2016some} that the conditional joint distribution
    \begin{equation} \label{eq:condgaussianity}
        \pp\big(\vy^{j+1}\leq \va,\vx^{j+1}\leq \boldsymbol{\beta}\big|\vx^{s},\ s\leq j\big),
    \end{equation}
    for sequences $\{(\vx^j,\vy^j)\}_{j=0,1,\ldots,J}$ governed by discrete-time systems in the form of \eqref{eq:discretecondgauss1}--\eqref{eq:discretecondgauss2}, is $\pp$-a.s.\ Gaussian, $\forall j=0,\ldots,J-1$ (after using the exchangeability (in law) property of jointly distributed Gaussian random vectors by virtue of Lemma \ref{lem:affinity} and permutation matrices). We then aim to describe the mean and covariance matrix of this distribution. First, in light of \eqref{eq:discretecondgauss1}--\eqref{eq:discretecondgauss2}, \eqref{eq:condgaussianity}, and Lemma \ref{lem:affinity} (as to isolate the marginals), a bit of linear algebra provides the following marginalized conditionally Gaussian distributions:
    \begin{align}
        \pp\big(\vx^{j+1}\big|\vx^{s},s\leq j\big)&\overset{\rmd}{\sim}\mathcal{N}_k\big(\vx^j+(\ml^{\vx,j}\vm{\nf}^j+\vf^{\vx,j})\dt,(\ms^\vx\circ\ms^\vx)^j\dt\big), \label{eq:cg1}\\
        \pp\big(\vy^{j+1}\big|\vx^{s},s\leq j\big)&\overset{\rmd}{\sim}\mathcal{N}_l\big(\vm{\nf}^j+(\ml^{\vy,j}\vm{\nf}^j+\vf^{\vy,j})\dt,\mr{\nf}^j+(\ml^{\vy,j}\mr{\nf}^j+\mr{\nf}^j(\ml^{\vy,j})^{\dagger}+(\ms^\vy\circ\ms^\vy)^j)\dt\big),
        \label{eq:cg2}
    \end{align}
     where $\pp\big(\vy^j\big|\vx^s,s\leq j\big)\overset{\rmd}{\sim}\mathcal{N}_l(\vm{\nf}^j,\mr{\nf}^j)$ by Theorem \ref{thm:condgaussianity}. Note again how due to the nature of the EM discretization scheme being used, terms of order higher than $O(\dt)$ have been dropped in the Gaussian statistics appearing in \eqref{eq:cg1}--\eqref{eq:cg2}. As for the cross-covariance tensor, $\mathrm{Cov}\left(\vx^{j+1},\vy^{j+1}\big|\vx^{s},s\leq j\right):=\overline{\vx^{\prime,j+1}(\vy^{\prime,j+1})^\dagger}$, where the overline denotes taking the expectation when conditioned on the observations $\vx^0,\ldots,\vx^j$, we multiply \eqref{eq:discretefluc1} by the Hermitian transpose of \eqref{eq:discretefluc2} from the right, which gives
    \begin{equation}
        \mathrm{Cov}\left(\vx^{j+1},\vy^{j+1}|\vx^s,s\leq j\right):=\overline{\vx^{\prime,j+1}(\vy^{\prime,j+1})^\dagger}=(\ml^{\vx,j}\mr{\nf}^j+(\ms^\vx\circ\ms^\vy)^j)\dt,\label{eq:cg3}
    \end{equation}
    where again terms of order higher than $O(\dt)$ have been dropped, while the terms of order $O(\sqrt{\dt})$ disappear since by the law of total expectation and properties of the Gaussian random noises we have
    \begin{equation*}
        \ee{\ve_m^j(\vy^{\prime,j+1})^\dagger\big|\cF_{t_j}^\vx}=\ee{\ee{\ve_m^j(\vy^{\prime,j+1})^\dagger\big|\cF_{t_j}^{\vx,\vy^{j+1}}}\Big|\cF_{t_j}^\vx}=\ee{\ee{\ve_m^j\big|\cF_{t_j}^{\vx,\vy^{j+1}}}(\vy^{\prime,j+1})^\dagger\Big|\cF_{t_j}^\vx}
        =0, \ \forall m=1,2,
    \end{equation*}
    where $\cF_{t_j}^{\vx,\vy^{j+1}}$ denotes the join $\sigma$-algebra $\cF_{t_j}^\vx\vee\cF^{\vy^{j+1}}:=\sigma\Big(\cF_{t_j}^\vx\cup\sigma(\vy^{j+1})\Big)\supseteq\cF_{t_j}^\vx$. As such, by \eqref{eq:condgaussianity} and \eqref{eq:cg1}--\eqref{eq:cg3}, the following joint distribution can be formed
    \begin{equation} \label{eq:jointcondgaussianity}
        \pp\big(\vy^{j+1},\vx^{j+1}\big|\vx^s,s\leq j\big)\overset{\rmd}{\sim}\mathcal{N}_{l+k}\left( \ee{\vy^{j+1},\vx^{j+1}\big|\vx^s,s\leq j},\mathrm{Var}\big(\vy^{j+1},\vx^{j+1}\big|\vx^s,s\leq j\big)\right),
    \end{equation}
    where
    \begin{equation}
    \begin{gathered}
        \ee{\vy^{j+1},\vx^{j+1}\big|\vx^s,s\leq j}= \begin{pmatrix}
            \vm{\nf}^j+(\ml^{\vy,j}\vm{\nf}^j+\vf^{\vy,j})\dt\\
            \vx^j+(\ml^{\vx,j}\vm{\nf}^j+\vf^{\vx,j})\dt
        \end{pmatrix},\\
        \mathrm{Var}\left(\vy^{j+1},\vx^{j+1}\big|\vx^s,s\leq j\right)=\begin{pmatrix}
            \mr{\nf}^j(\mathbf{I}_{l\times l}+\mathbf{H}^j\dt) & \mr{\nf}^j(\mathbf{G}^{\vx,j})^\dagger\dt\\
            \mathbf{G}^{\vx,j}\mr{\nf}^j\dt & (\ms^\vx\circ\ms^\vx)^j\dt
        \end{pmatrix},
    \end{gathered} \label{eq:jointcondgaussianitydetails}
    \end{equation}
    for
    \begin{equation*}
        \mathbf{G}^{\vx,j}:=\ml^{\vx,j}+(\ms^\vx\circ\ms^\vy)^j\big(\mr{\nf}^j\big)^{-1},\quad \mathbf{H}^j:=\big(\mr{\nf}^j\big)^{-1}(\ml^{\vy,j}\mr{\nf}^j+\mr{\nf}^j(\ml^{\vy,j})^{\dagger}+(\ms^\vy\circ\ms^\vy)^j),
    \end{equation*}
    with $\mathbf{I}_{l\times l}$ denoting the identity matrix of size $l\times l$. As such, in light of Lemma \ref{lem:conditional} and \eqref{eq:jointcondgaussianity}--\eqref{eq:jointcondgaussianitydetails}, by further conditioning on $\vx^{j+1}$ (meaning we have additionally observed $\vx^{j+1}$), then we have that,
    \begin{equation*}
        \pp\big(\vy^{j+1}\big|\vx^s, s\leq j+1\big)\overset{\rmd}{\sim} \mathcal{N}_l(\vm{\nf}^{j+1},\mr{\nf}^{j+1}),
    \end{equation*}
    where we have the following recursive relations for the conditional filter mean and covariance, which essentially are the discrete-time optimal nonlinear filter state estimation equations:
    \begin{align}
        \vm{\nf}^{j+1}&=\vm{\nf}^j+(\ml^{\vy,j}\vm{\nf}^j+\vf^{\vy,j})\dt+\mathbf{K}^j(\vx^{j+1}-\vx^j-(\ml^{\vx,j}\vm{\nf}^j+\vf^{\vx,j})\dt), \label{eq:discretefilter1} \\
        \mr{\nf}^{j+1}&=\mr{\nf}^j+\big(\ml^{\vy,j}\mr{\nf}^j+\mr{\nf}^j(\ml^{\vy,j})^{\dagger}+(\ms^\vy\circ\ms^\vy)^j-\mathbf{K}^j((\ms^\vx\circ\ms^\vy)^j+\ml^{\vx,j}\mr{\nf}^j)\big)\dt, \label{eq:discretefilter2}
    \end{align}
    where
    \begin{equation} \label{eq:kalmangainmatrix}
        \mathbf{K}^j:=((\ms^\vy\circ\ms^\vx)^j+\mr{\nf}^j(\ml^{\vx,j})^{\dagger})\big((\ms^\vx\circ\ms^\vx)^{j}\big)^{-1},
    \end{equation}
    which for the remainder of this work is called the Kalman gain matrix (see subsequent discussion), evaluated at $t_j$. Taking then the limit as $\dt\to0^+$, under the EM convergence framework we have formalized, results in the recovery of the continuous-time filter differential equations for the Gaussian statistics of the posterior distribution given in \eqref{eq:filter1}--\eqref{eq:filter2}, when $t\in[0,T]$.

    As for the uniqueness and continuity of solutions to \eqref{eq:filter1}--\eqref{eq:filter2} (as well as their continuous dependence on the initial conditions and additional model parameters), the details can be found in Kolodziej \cite{kolodziej1980state}, but essentially these follow by the fact that $\vm{\nf}(t)$ satisfies a random ordinary differential equation linear in $\vm{\nf}(t)$ \cite{strand1970random, soong1973random, neckel2013random, han2017random}, while $\mr{\nf}(t)$ satisfies a random Riccati equation \cite{wang1999stability, casaban2018solving, bishop2019stability}, with matrix-valued coefficients and forcings that satisfy the conditions of the random variant of the Picard-Lindel\"of theorem (under our assumed regularity conditions) \cite{strand1970random, han2017random, soong1973random}. An alternative approach to the existence, uniqueness, and continuity of the filter Gaussian statistics is provided in the corresponding theorems from Liptser \& Shiryayev \cite{liptser2001statisticsI, liptser2001statisticsII}.

    For the final statement of the theorem, having now transitioned to the limit in continuous-time, the result can be established via elementary results from the theory of differential equations. We first define the random time
    \begin{equation*}
        \tau:={\inf}\{t\in[0,T]:\mathrm{det}(\mr{\nf}(t))=0\},
    \end{equation*}
    with the convention that $\tau=T$ if $\underset{t\in[0,T]}{\inf}\{\mathrm{det}(\mr{\nf}(t))\}>0$. This is the smallest random time on which the filter covariance matrix becomes singular, and by the virtue of the filter statistics being adapted to the filtration $\{\cF_t^\vx\}_{t\in[0,T]}$, continuous in time, and by continuity of the determinant over $\mathbb{C}^{l\times l}$, then $\tau$ is a predictable (previsible) stopping time \cite{rogers2000diffusions}. As a result, for any $t<\min\{\tau,T\}$, the matrices $\mr{\nf}^{-1}(t)$ are well-defined, and as such over this period we have ($\pp$-a.s.)
    \begin{equation*}
        \mathbf{0}_{l\times l}=\frac{\rmd}{\rmd t}(\mathbf{I}_{l\times l})=\frac{\rmd}{\rmd t}(\mr{\nf}(t)\mr{\nf}^{-1}(t))=\frac{\rmd\mr{\nf}}{\rmd t}(t)\mr{\nf}^{-1}(t)+\mr{\nf}(t)\frac{\rmd\mr{\nf}^{-1}}{\rmd t}(t),
    \end{equation*}
    or,
    \begin{equation*}
        \frac{\rmd\mr{\nf}^{-1}}{\rmd t}(t)=-\mr{\nf}^{-1}(t)\frac{\rmd\mr{\nf}}{\rmd t}(t)\mr{\nf}^{-1}(t).
    \end{equation*}
    Using then \eqref{eq:filter2}, we have that for $t<\min\{\tau,T\}$, the inverse of the filter covariance matrix satisfies the following random Riccati differential equation:
    \begin{equation} \label{eq:inversefilterdiff}
        \rmd\mr{\nf}^{-1}(t)=-\big(\ma^\dagger\mr{\nf}^{-1}+\mr{\nf}^{-1}\ma-(\ml^\vx)^\dagger(\ms^\vx\circ\ms^\vx)^{-1}\ml^\vx+\mr{\nf}^{-1}\mb\mr{\nf}^{-1}\big)\rmd t,
    \end{equation}
    where the auxiliary matrices in \eqref{eq:inversefilterdiff} are defined as:
    \begin{align}
        \ma(t,\vx)&:=\ml^\vy(t,\vx)-(\ms^\vy\circ \ms^\vx)(t,\vx)(\ms^\vx\circ \ms^\vx)^{-1}(t,\vx)\ml^\vx(t,\vx)\in\mathbb{C}^{l \times l}, \label{eq:auxiliarymata}\\
        \mb(t,\vx) &:=  (\ms^\vy\circ \ms^\vy)(t,\vx)-(\ms^\vy\circ \ms^\vx)(t,\vx)(\ms^\vx\circ \ms^\vx)^{-1}(t,\vx)(\ms^\vx\circ \ms^\vy)(t,\vx)\in\mathbb{C}^{l \times l}. \label{eq:auxiliarymatb}
    \end{align}
    Now, by $\mathrm{det}(\mr{\nf}^{-1}(t))=\mathrm{det}(\mr{\nf}(t))^{-1}$, continuity of the determinant, and $\tau$ being a predictable stopping time, it is necessary that the elements of the matrix $\mr{\nf}^{-1}(t)$ must increase as $t\uparrow \tau$. Nonetheless, it is actually possible to show that they are $\pp$-a.s.\ bounded.

    We first prove the auxiliary result of,
    \begin{equation} \label{eq:nonnegativedefmatb}
        \mb\succeq \mathbf{0}_{l\times l} \quad (\pp\text{-a.s.}),
    \end{equation}
    i.e., $\mb$ is nonnegative-definite. Let $\ve_1\overset{\rmd}{\sim}\mathcal{N}_d(\mathbf{0}_{d\times 1},\mathbf{I}_d)$ and $\ve_2\overset{\rmd}{\sim}\mathcal{N}_r(\mathbf{0}_{r\times 1},\mathbf{I}_r)$, with $\ve_1\perp\!\!\!\!\perp\ve_2$, and define $\boldsymbol{\theta}:=\ms^\vy_1\ve_1+\ms^\vy_2\ve_2$ and $\boldsymbol{\xi}:=\ms^\vx_1\ve_1+\ms^\vx_2\ve_2$. Then, by Theorem \ref{thm:condgaussianity}, we have
    \begin{equation*}
        \mathrm{Cov}\big(\boldsymbol{\theta},\boldsymbol{\theta}\big|\boldsymbol{\xi}\big):=\mb=(\ms^\vy\circ \ms^\vy)(t,\vx)-(\ms^\vy\circ \ms^\vx)(t,\vx)(\ms^\vx\circ \ms^\vx)^{-1}(t,\vx)(\ms^\vx\circ \ms^\vy)(t,\vx),
    \end{equation*}
    and as a covariance matrix it is necessarily nonnegative-definite.

    By taking now the trace of \eqref{eq:inversefilterdiff} and using its continuity and linearity to yield $\mathrm{tr}(\rmd\mr{\nf}^{-1}(t))=\rmd(\mathrm{tr}(\mr{\nf}^{-1}(t)))$, it follows then by the invariance of the trace under cyclic permutations of products that
    \begin{equation*}
        \rmd(\mathrm{tr}(\mr{\nf}^{-1}(t)))=-2\mathrm{tr}\big(\mr{\nf}^{-1}(t)\ma\big)+\mathrm{tr}\big((\ml^\vx)^\dagger(\ms^\vx\circ\ms^\vx)^{-1}\ml^\vx\big)-\mathrm{tr}\big(\mr{\nf}^{-1}\mb\mr{\nf}^{-1}\big),
    \end{equation*}
    and since $\mr{\nf}^{-1}\mb\mr{\nf}^{-1}\succeq \mathbf{0}_{l\times l}$ by \eqref{eq:nonnegativedefmatb}, for $t<\min\{\tau,T\}$, we end up with
    \begin{equation*}
        \rmd(\mathrm{tr}(\mr{\nf}^{-1}(t)))\leq -2\mathrm{tr}\big(\mr{\nf}^{-1}(t)\ma\big)+M\left\|(\ms^\vx\circ\ms^\vx)^{-1/2}\ml^\vx\right\|_2^2,
    \end{equation*}
    since by the Cauchy-Schwarz inequality on the Frobenius inner product and the equivalence of norms in finite-dimensional spaces we have that $\exists M>0$ such that
    \begin{equation*}
        \mathrm{tr}\big((\ml^\vx)^\dagger(\ms^\vx\circ\ms^\vx)^{-1}\ml^\vx\big)=\mathrm{tr}\big(\big((\ms^\vx\circ\ms^\vx)^{-1/2}\ml^\vx\big)^\dagger(\ms^\vx\circ\ms^\vx)^{-1/2}\ml^\vx\big)\leq M\left\|(\ms^\vx\circ\ms^\vx)^{-1/2}\ml^\vx\right\|_2^2.
    \end{equation*}
    After integration over $[0,t]$ for $t<\min\{\tau,T\}$ we then have
    \begin{equation} \label{eq:intermediateresult1}
        \mathrm{tr}(\mr{\nf}^{-1}(t))\leq \mathrm{tr}(\mr{\nf}^{-1}(0))+M\int_0^t\left\|(\ms^\vx\circ\ms^\vx)^{-1/2}(s,\vx)\ml^\vx(s,\vx)\right\|_2^2\rmd s-2\int_0^t\mathrm{tr}\big(\mr{\nf}^{-1}(s)\ma(s,\vx)\big)\rmd s.
    \end{equation}
    Now, since over this interval we have $\mr{\nf}^{-1}(t)\succ\mathbf{0}_{l\times l}$, then
    \begin{equation} \label{eq:intermediateresult2}
        \left\|\mr{\nf}^{-1}(t)\right\|^2_2\leq\rho(\mr{\nf}^{-1}(t))^2\leq \sum_{\mathclap{\mu(t)\in \lambda(\mr{\nf}^{-1}(t))}}\mu(t)^2= \mathrm{tr}(\mr{\nf}^{-1}(t))^2,
    \end{equation}
    where $\rho(\boldsymbol{\cdot})$ denotes the spectral radius and $\lambda(\boldsymbol{\cdot})$ denotes the spectrum of a linear operator. Furthermore, by the sub-multiplicativity of the induced Euclidean norm, we have
    \begin{align}
    \begin{split}
        \mathrm{tr}(\mr{\nf}^{-1}(0))+M\int_0^t\left\|(\ms^\vx\circ\ms^\vx)^{-1/2}(s,\vx)\ml^\vx(s,\vx)\right\|_2^2\rmd s&\\
        &\hspace{-4cm}\leq \mathrm{tr}(\mr{\nf}^{-1}(0))+M\int_0^t\left\|(\ms^\vx\circ\ms^\vx)^{-1/2}(s,\vx)\right\|_2^2\left\|\ml^\vx(s,\vx)\right\|_2^2\rmd s:=C_1,
    \end{split} \label{eq:intermediateresult3}
    \end{align}
    where $C_1<+\infty$ ($\pp$-a.s.), by assumptions \textbf{\magenta{(2)}}, \textbf{\magenta{(5)}}, and $\pp\left(\mathrm{tr}(\mr{\nf}(0))<+\infty\right)=1$. Now, again by the Cauchy-Schwarz inequality on the Frobenius inner product and the equivalence of norms in finite-dimensional spaces we have
    \begin{align}
    \begin{split}
        -\int_0^t\mathrm{tr}\big(\mr{\nf}^{-1}(s)\ma(s,\vx)\big)\rmd s\leq \int_0^t|\mathrm{tr}\big(\mr{\nf}^{-1}(s)\ma(s,\vx)\big)|\rmd s\leq M\int_0^t\left\|\mr{\nf}^{-1}(s)\right\|_2\left\|\ma(s,\vx)\right\|_2\rmd s,
    \end{split} \label{eq:intermediateresult4}
    \end{align}
    and so by combining \eqref{eq:intermediateresult2}--\eqref{eq:intermediateresult4} with \eqref{eq:intermediateresult1}, and applying Young's inequality we get
    \begin{equation*}
        \left\|\mr{\nf}^{-1}(t)\right\|^2_2\leq 2C^2_1+8M^2\left(\int_0^t\left\|\mr{\nf}^{-1}(s)\right\|_2\left\|\ma(s,\vx)\right\|_2\rmd s\right)^2.
    \end{equation*}
    Now, as a positive function, $\displaystyle\exists\int_0^t\left\|\mr{\nf}^{-1}(s)\right\|^2_2\rmd s\in(0,+\infty]$, and as such, in the generalized sense, Cauchy-Schwarz yields
    \begin{equation*}
        \left(\int_0^t\left\|\mr{\nf}^{-1}(s)\right\|_2\left\|\ma(s,\vx)\right\|_2\rmd s\right)^2\leq \int_0^t\left\|\mr{\nf}^{-1}(s)\right\|^2_2\rmd s\int_0^t\left\|\ma(s,\vx)\right\|_2^2\rmd s,
    \end{equation*}
    where $\int_0^t\left\|\ma(s,\vx)\right\|_2^2\rmd s\leq \int_0^T\left\|\ma(t,\vx)\right\|_2^2\rmd t<+\infty$ since by Young's inequality and the sub-multiplicativity of the induced Euclidean norm, we have
    \begin{align*}
        C_2:=\int_0^T\left\|\ma(t,\vx)\right\|_2^2\rmd t&\leq 2\int_0^T\left\|\ml^\vy\right\|_2^2\rmd t+2\int_0^T\left\|(\ms^\vx\circ\ms^\vx)^{-1}\right\|_2^2(\left\|\ms^\vy\circ\ms^\vx\right\|_2^4+\left\|\ml^\vy\right\|_2^4)\rmd t\\
        &\lesssim \int_0^T\left\|\ml^\vy\right\|_2^2\rmd t+\int_0^T\big(\left\|\ms_1^\vy\right\|_2^4+\left\|\ms_2^\vy\right\|_2^4+\left\|\ml^\vy\right\|_2^4\big)\rmd t<+\infty,
    \end{align*}
    as a result of assumptions \textbf{\magenta{(2)}}, \textbf{\magenta{(5)}}, and \textbf{\magenta{(11)}}. Putting all these results together, we retrieve
    \begin{equation*}
        \left\|\mr{\nf}^{-1}(t)\right\|^2_2\leq 2C^2_1+8M^2C_2\int_0^t\left\|\mr{\nf}^{-1}(s)\right\|^2_2\rmd s,
    \end{equation*}
    where by Gr\"onwall's inequality we have
    \begin{equation*}
        \left\|\mr{\nf}^{-1}(t)\right\|^2_2\leq 2C_1^2+16TC_1^2M^2C_2e^{8TM^2C_2}<+\infty, \ \forall t<\min\{\tau,T\}.
    \end{equation*}
    Using once more the equivalence of norms, in this case for $\left\|\boldsymbol{\cdot}\right\|_2$ and the $L^{\infty,\infty}$ element-wise norm, $\left\|\mathbf{V}\right\|_{\infty,\infty}:=\underset{i,j}{\max}\{|V_{ij}|\}$, then
    \begin{equation*}
        \left\|\mr{\nf}^{-1}(t)\right\|^2_{\infty,\infty}\lesssim \left\|\mr{\nf}^{-1}(t)\right\|^2_2\lesssim 1+e^{8TM^2C_2}<+\infty,
    \end{equation*}
    which proves that the elements of $\mr{\nf}^{-1}(t)$ remain $\pp$-a.s.\ bounded as $t\uparrow\tau$. But, because of this boundedness of elements, as to not contradict the definition of $\tau$ (the elements of the matrix $\mr{\nf}^{-1}(t)$ must increase as $t\uparrow \tau$), it is then necessary that $\tau=T$, and as such $\pp(\tau<T)=0$. We note here that we have essentially proved the boundedness of the elements and spectral or Euclidean norm of $\mr{\nf}^{-1}(t)$, which we can also show for the filter covariance matrix itself, $\mr{\nf}(t)$, with a similar procedure. Similarly, we can also show that $\pp\big(\int_0^T\left\|\ms^\vx\circ\ms^\vx(t,\vx)\right\|_2^2\rmd t<+\infty\big)$; details are included in the work of Kolodziej \cite{kolodziej1980state}.
\end{proof}

\subsection{Proof of Theorem \ref{thm:smoothing}} \label{sec:Appendix_C}

The following lemmas are needed to derive the optimal nonlinear smoother state estimation backward equations. The first lemma provides some concentration inequalities, specifically tail bounds on the spectrum norm, and by extension the spectral radius, of some auxiliary matrices of interest. These are important for the expansion up to leading-order terms of expressions of the form $(\mathbf{I}_p\pm \mathbf{V}\dt)^{-1}$ for $\mathbf{V}\in\mathbb{C}^{p\times p}$ through the Neumann series associated with this inverse.

\begin{lem}[\textbf{Tail Bounds for some Spectral Norms}] \label{lem:tailboundsspectral}
    Let the regularity conditions \textbf{\magenta{(1)}}--\textbf{\magenta{(13)}} to hold, and adopt the temporal discretization scheme that was formulated at the beginning of Section \ref{sec:2.5}. We further define the following auxiliary matrices:
    \begin{gather*}
        \mathbf{H}(t,\vx):=\mr{\nf}^{-1}(t)(\ml^{\vy}(t,\vx)\mr{\nf}(t)+\mr{\nf}(t)(\ml^{\vy}(t,\vx))^{\dagger}+(\ms^\vy\circ\ms^\vy)(t,\vx)),\\
        \mathbf{L}(t,\vx):=(\ms^\vx\circ\ms^\vx)^{-1}(t,\vx)\mathbf{G}^\vx(t,\vx)\mr{\nf}(t)(\mathbf{G}^\vx(t,\vx))^\dagger,
    \end{gather*}
    where $\mathbf{G}^{\vx}:=\ml^{\vx,j}(t,\vx)+(\ms^\vx\circ\ms^\vy)(t,\vx)\mr{\nf}^{-1}(t)$. Then the following are true:
    \begin{gather}
    \lim_{\substack{\dt\to0^+\\J\dt=T}}\pp\left(\underset{0\leq j\leq J}{\max}\left\{\left\|\mathbf{H}^j\dt\right\|_2\right\}\geq 1\right)=0 \quad \text{and} \quad\lim_{\substack{\dt\to0^+\\J\dt=T}}\pp\left(\underset{0\leq j\leq J}{\max}\left\{\left\|\mathbf{L}^j\dt\right\|_2\right\}\geq 1\right)=0. \label{eq:spectralnormbound}
    \end{gather}
\end{lem}
\begin{proof}[\textbf{Proof of Lemma \ref{lem:tailboundsspectral}}]
    We start with the tail result concerning $\mathbf{H}$ in \eqref{eq:spectralnormbound}. Define
\begin{equation*}
    B_j(\dt):=\left\{\omega\in\Omega:\left\|\mathbf{H}^j(\omega)\dt\right\|_2\geq 1\right\}=\left\{\omega\in\Omega:\left\|\mathbf{H}(t_j,\vx(t_j,\omega))\dt\right\|_2\geq 1\right\},
\end{equation*}
for each $j=0,1,\ldots,J$. Then, for $A(\dt):=\left\{\omega\in\Omega:\underset{0\leq j\leq J}{\max}\left\{\left\|\mathbf{H}^j(\omega)\dt\right\|_2\right\}\geq 1\right\}$, it immediately follows that
\begin{equation*}
    A(\dt)=\bigcup_{j=0}^J B_j(\dt), \ \forall J\in\mathbb{N}.
\end{equation*}
By Boole's inequality, we then have $\pp(A(\dt))\leq \sum_{j=0}^J \pp(B_j(\dt))$. While this upper bound might seem rather crude, intuitively, for smaller values of $\dt$, this domination gets tighter. We now look at a specific $\pp(B_j(\dt))$. We first define the so called self-adjoint dilation matrix from operation theory \cite{tropp2012user}, also known as the Jordan–Wielandt matrix in perturbation theory and other areas \cite{eisenstat1995relative, li1999lidskii}. This can be defined as an operator from $\mathbb{C}^{d_1\times d_2}$ to $\mathbb{C}^{(d_1+d_2)\times(d_1+d_2)}$ where for $\mathbf{Y}\in\mathbb{C}^{d_1\times d_2}$
\begin{equation*}
    \boldsymbol{\mathcal{H}}(\mathbf{Y}):=\begin{pmatrix}
        \mathbf{0}_{d_1\times d_1} & \mathbf{Y} \\
        \mathbf{Y}^\dagger & \mathbf{0}_{d_2\times d_2}
    \end{pmatrix}.
\end{equation*}
Its usage is usually in assessing properties of the singular values of the input matrix through its spectrum, since
\begin{equation*}
    \boldsymbol{\mathcal{H}}(\mathbf{Y})^2=\begin{pmatrix}
        \mathbf{Y}\mathbf{Y}^\dagger & \mathbf{0}_{d_1\times d_2} \\
        \mathbf{0}_{d_2\times d_1} & \mathbf{Y}^\dagger\mathbf{Y}
    \end{pmatrix}.
\end{equation*}
The most important properties of this operator is that it is real-linear, self-adjoint, and it preserves the spectral norm of the input \cite{tropp2012user}, in other words
\begin{equation} \label{eq:selfadjointprop}
    \rho(\boldsymbol{\mathcal{H}}(\mathbf{Y}))=\left\|\boldsymbol{\mathcal{H}}(\mathbf{Y})\right\|_2=\left\|\mathbf{Y}\right\|_2
\end{equation}
Through this, then the following series of (in)equalities is true for each $j=0,1,\ldots,J$:
\begin{equation} \label{eq:boundsprobability}
    \pp(B_j(\dt))=\pp\big(\big\|\mathbf{H}^j\dt\big\|_2\geq1\big)=\pp\left(\big\|\boldsymbol{\mathcal{H}}(\mathbf{H}^j)\big\|_2\geq \frac{1}{\dt}\right)\leq \underset{\theta>0}{\mathrm{inf}}\left\{e^{-\theta/\dt}\ee{\mathrm{tr}\left(e^{\theta\boldsymbol{\mathcal{H}}(\mathbf{H}^j)} \right)}\right\},
\end{equation}
where the inequality used at the end follows from the Laplace transform method, which provides tail bounds for the extreme eigenvalues of self-adjoint matrices \cite{ahlswede2002strong, oliveira2009spectrum, tropp2019matrix}, which $\boldsymbol{\mathcal{H}}(\mathbf{H}^j)$ is by definition (even if $\mathbf{H}^j$ is not). We now isolate the quantity $\ee{\mathrm{tr}\left(e^{\theta\boldsymbol{\mathcal{H}}(\mathbf{H}^j)} \right)}$ for some $\theta>0$. Now, generally $\lambda(e^{\theta\mathbf{V}})=e^{\lambda(\theta \mathbf{V})} \pmod{2\pi i}=e^{\theta\lambda(\mathbf{V})} \pmod{2\pi i}$ for arbitrary $\mathbf{V}$, but $\boldsymbol{\mathcal{H}}(\mathbf{H}^j)$ is self-adjoint, and specifically nonnegative-definite, and as such has real and nonnegative eigenvalues which means this relation is exact, or $\lambda(e^{\theta\boldsymbol{\mathcal{H}}(\mathbf{H}^j)})=e^{\theta\lambda(\boldsymbol{\mathcal{H}}(\mathbf{H}^j))}$. Now, by block-matrix algebra, we have $\mathrm{det}(\boldsymbol{\mathcal{H}}(\mathbf{H}^j)-\lambda\mathbf{I}_{2l})=\mathrm{det}(\lambda^2\mathbf{I}_{l}-\mathbf{H}^j(\mathbf{H}^j)^\dagger)$, which shows that the eigenvalues of $\boldsymbol{\mathcal{H}}(\mathbf{H}^j)$ are exactly the square roots of the eigenvalues of $\mathbf{H}^j(\mathbf{H}^j)^\dagger$ (or, equivalently, $(\mathbf{H}^j)^\dagger\mathbf{H}^j$), which are exactly the singular values of $\mathbf{H}^j$. As such, by combining the preceding observations and using the linearity of the trace and expectation, we have
\begin{equation} \label{eq:expectedtrace}
    \ee{\mathrm{tr}\big(e^{\theta\boldsymbol{\mathcal{H}}(\mathbf{H}^j)}\big)}=\sum_{\mathclap{\mu\in\lambda((\mathbf{H}^j)^\dagger\mathbf{H}^j)}}\ee{e^{\theta\sqrt{\mu}}}.
\end{equation}
We now bound the eigenvalues of $(\mathbf{H}^j)^\dagger\mathbf{H}^j$. It is known from a multiplicative analog of the fundamental Lidskii-Mirsky-Wielandt eigenvalue majorization theorem that for $\mu_1(\mathbf{V})\geq \mu_2(\mathbf{V})\geq \cdots \geq \mu_p(\mathbf{V})$ denoting the ordered eigenvalues of some Hermitian matrix $\mathbf{V}\in\mathbb{C}^{p\times p}$ then we have \cite{li1999lidskii}
\begin{equation*}
    \mu_i(\mathbf{S}^\dagger\mathbf{V}\mathbf{S})\leq \rho(\mathbf{S}^\dagger\mathbf{S})\mu_i(\mathbf{V}), \ \forall i=1,\ldots,p,
\end{equation*}
with $\mathbf{S}\in\mathbb{C}^{p\times p}$ arbitrary. Using then this result on $(\mathbf{H}^j)^\dagger\mathbf{H}^j=\big(\mr{\nf}^j\big)^{-1}(\ml^{\vy,j}\mr{\nf}^j+\mr{\nf}^j(\ml^{\vy,j})^{\dagger}+(\ms^\vy\circ\ms^\vy)^j)^2\big(\mr{\nf}^j\big)^{-1}$, we have
\begin{equation} \label{eq:spectrumhjhj}
     \mu_i((\mathbf{H}^j)^\dagger\mathbf{H}^j)\leq \rho\big(\big(\mr{\nf}^j\big)^{-1}\big)^2 \mu_i(\ml^{\vy,j}\mr{\nf}^j+\mr{\nf}^j(\ml^{\vy,j})^{\dagger}+(\ms^\vy\circ\ms^\vy)^j)^2, \ \forall i=1,\ldots,l,
\end{equation}
since $\big(\mr{\nf}^j\big)^{-1}$ is Hermitian and for $p\in\mathbb{C}_m[z]$ a polynomial we have $p(\lambda(\mathbf{V}))=\lambda(p(\mathbf{V}))$ in $\mathbb{C}$. We have already shown during the proof of Theorem \ref{thm:filtering} that $\pp$-a.s.\ $\rho\big(\big(\mr{\nf}^j\big)^{-1}\big)<+\infty$ uniformly in time, with a similar result for the filter covariance matrix itself being true too \cite{kolodziej1980state, wang1999stability, bishop2019stability}. So, here we aim to bound $\lambda(\ml^{\vy,j}\mr{\nf}^j+\mr{\nf}^j(\ml^{\vy,j})^{\dagger}+(\ms^\vy\circ\ms^\vy)^j)$ for each $j=0,1,\ldots,J$. For that, the following series of (in)equalities hold for any $\Tilde{\mu}\in\lambda(\ml^{\vy,j}\mr{\nf}^j+\mr{\nf}^j(\ml^{\vy,j})^{\dagger}+(\ms^\vy\circ\ms^\vy)^j)$:
\begin{align}
\begin{split}
    \Tilde{\mu}&\leq \rho(\ml^{\vy,j}\mr{\nf}^j+\mr{\nf}^j(\ml^{\vy,j})^{\dagger}+(\ms^\vy\circ\ms^\vy)^j)= \big\|\ml^{\vy,j}\mr{\nf}^j+\mr{\nf}^j(\ml^{\vy,j})^{\dagger}+(\ms^\vy\circ\ms^\vy)^j\big\|_2\\
    &\leq 2\big\|\mr{\nf}^j\big\|_2\big\|\ml^{\vy,j}\big\|_2+\big\|(\ms^\vy\circ\ms^\vy)^j\big\|_2.
\end{split} \label{eq:eigenvaluebound}
\end{align}
But, as already mentioned, we have by our set of assumed regularity conditions that $\big\|\mr{\nf}^j\big\|_2<+\infty$ uniformly in time ($\pp$-a.s.) \cite{kolodziej1980state, wang1999stability, bishop2019stability}, $\big\|\ml^{\vy,j}\big\|_2<+\infty$ ($\big\|\ml^{\vy,j}\big\|_2\leq l^{3/2}\big\|\ml^{\vy,j}\big\|_{\infty,\infty}\leq l^{3/2}L_1$ by \textbf{\magenta{(2)}}), while for $\left\|(\ms^\vy\circ\ms^\vy)^j\right\|_2$, we have by the sub-multiplicativity of the spectral norm that
\begin{equation*}
    \big\|(\ms^\vy\circ\ms^\vy)^j\big\|_2\leq \big\|\ms_1^{\vy,j}\big\|^2_2+\big\|\ms_2^{\vy,j}\big\|^2_2.
\end{equation*}
As such, by combining these observations with \eqref{eq:eigenvaluebound} into \eqref{eq:spectrumhjhj}, yields for all $i=1,\ldots,l$ $\pp$-a.s. the following:
\begin{align*}
    \sqrt{\mu_i((\mathbf{H}^j)^\dagger\mathbf{H}^j)}&\leq \rho\big(\big(\mr{\nf}^j\big)^{-1}\big) \mu_i(\ml^{\vy,j}\mr{\nf}^j+\mr{\nf}^j(\ml^{\vy,j})^{\dagger}+(\ms^\vy\circ\ms^\vy)^j)\\
    &\leq \big\|\big(\mr{\nf}^j\big)^{-1}\big\|_2\big(2\big\|\mr{\nf}^j\big\|_2\big\|\ml^{\vy,j}\big\|_2+\big\|\ms_1^{\vy,j}\big\|^2_2+\big\|\ms_2^{\vy,j}\big\|^2_2\big)\\
    &\leq M\big(1+\big\|\ms_1^{\vy,j}\big\|^2_2+\big\|\ms_2^{\vy,j}\big\|^2_2\big),
\end{align*}
where the quantity $M$ is defined as
\begin{equation*}
    M:=\underset{t\in[0,T]}{\sup}\big\{\big\|\mr{f}^{-1}(t)\big\|_2\big\}\max\Big\{2l^{3/2}L_1\underset{t\in[0,T]}{\sup}\big\{\big\|\mr{f}(t)\big\|_2\big\}, 1\Big\},
\end{equation*}
which is finite $\pp$-a.s. Putting the preceding results back into \eqref{eq:expectedtrace}, gives for $\dt>0$:
\begin{align*}
    e^{-\theta/\dt} \ee{\mathrm{tr}\big(e^{\theta\boldsymbol{\mathcal{H}}(\mathbf{H}^j)}\big)}&=e^{-\theta/\dt}\sum_{\mathclap{\mu\in\lambda((\mathbf{H}^j)^\dagger\mathbf{H}^j)}}\ee{e^{\theta\sqrt{\mu}}}\\
    &\leq l\ee{e^{\theta M\big(1+\big\|\ms_1^{\vy,j}\big\|^2_2+\big\|\ms_2^{\vy,j}\big\|^2_2\big)-\theta/\dt}}\\
    &\leq l\ee{e^{\theta M\big(1+\underset{0\leq j\leq J}{\max}\big\{\big\|\ms_1^{\vy,j}\big\|^2_2+\big\|\ms_2^{\vy,j}\big\|^2_2\big\}\big)-\theta/\dt}},
\end{align*}
and as such by \eqref{eq:boundsprobability} we have for any $j=0,1,\ldots,J$ that
\begin{equation*}
    \pp(B_j(\dt))\leq l\underset{\theta>0}{\inf}\Bigg\{\mathbb{E}\Bigg[e^{\theta M\big(1+\underset{0\leq j\leq J}{\max}\big\{\big\|\ms_1^{\vy,j}\big\|^2_2+\big\|\ms_2^{\vy,j}\big\|^2_2\big\}\big)-\theta/\dt}\Bigg]\Bigg\}.
\end{equation*}
Congregating all back into a probability result for $A(\dt)$ and using $J=T/\dt$, retrieves
\begin{equation*}
    \pp(A(\dt))\leq \sum_{j=0}^J\pp(B_j(\dt))\leq l\frac{T+\dt}{\dt}a_{\dt}(M),
\end{equation*}
where $a_{\dt}(M)$ is defined in assumption \textbf{\magenta{(13)}} found in the main text. As such, by condition \textbf{\magenta{(13)}}, we have that
\begin{equation*}
    0\leq \lim_{\substack{\dt\to0^+\\ J\dt=T}} \pp(A(\dt)) \leq \lim_{\substack{\dt\to0^+\\ J\dt=T}} l\frac{T+\dt}{\dt}o(\dt)=0,
\end{equation*}
which yields the desired result of $\lim_{\substack{\dt\to0^+\\J\dt=T}}\pp\left(\underset{0\leq j\leq J}{\max}\left\{\left\|\mathbf{H}^j\dt\right\|_2\right\}\geq 1\right)=0$.

The tail result concerning $\mathbf{L}$ in \eqref{eq:spectralnormbound} is obtained by following a similar procedure. Again, recall that in the proof of Theorem \ref{thm:filtering} it is shown that $\underset{t\in[0,T]}{\sup}\big\{\big\|\mr{f}(t)\big\|_2^2\big\}, \underset{t\in[0,T]}{\sup}\big\{\big\|\mr{f}^{-1}(t)\big\|_2^2\big\}<+\infty$ ($\pp$-a.s.) \cite{kolodziej1980state, wang1999stability, bishop2019stability}. Then, if we are to proceed in the same manner, we have an analogous inequality to \eqref{eq:spectrumhjhj} for any $i=1,\ldots,k$, by using similar algebraic manipulations and tools as the ones showcased for the tail bound concerning $\mathbf{H}$ and in the proof of Theorem \ref{thm:filtering}:
\begin{align*}
    \sqrt{\mu_i((\mathbf{L}^j)^\dagger\mathbf{L}^j)}&\leq \sqrt{\rho\big(\mathbf{G}^{\vx,j}\mr{\nf}^j(\mathbf{G}^{\vx,j})^\dagger\big((\ms^\vx\circ\ms^\vx)^j\big)^{-2}\mathbf{G}^{\vx,j}\mr{\nf}^j(\mathbf{G}^{\vx,j})^\dagger\big)}\\
    &\leq \sqrt{\mathrm{tr}\big(\mathbf{G}^{\vx,j}\mr{\nf}^j(\mathbf{G}^{\vx,j})^\dagger\big((\ms^\vx\circ\ms^\vx)^j\big)^{-2}\mathbf{G}^{\vx,j}\mr{\nf}^j(\mathbf{G}^{\vx,j})^\dagger\big)}\\
    &\leq \sqrt{\mathrm{tr}\big(\big((\ms^\vx\circ\ms^\vx)^j\big)^{-1/2}\mathbf{G}^{\vx,j}\mr{\nf}^j(\mathbf{G}^{\vx,j})^\dagger\big((\ms^\vx\circ\ms^\vx)^j\big)^{-1/2}\big)\big\|\mathbf{G}^{\vx,j}\mr{\nf}^j(\mathbf{G}^{\vx,j})^\dagger\big((\ms^\vx\circ\ms^\vx)^j\big)^{-1}\big\|_2}\\
    &\leq \mathrm{tr}\big(\big((\ms^\vx\circ\ms^\vx)^j\big)^{-1/2}\mathbf{G}^{\vx,j}\mr{\nf}^j(\mathbf{G}^{\vx,j})^\dagger\big((\ms^\vx\circ\ms^\vx)^j\big)^{-1/2}\big)\\
    &\lesssim \big\|(\ms^\vx\circ\ms^\vx)^j\big)^{-1/2}\mathbf{G}^{\vx,j}\big\|^2_2\big\|\mr{\nf}^j(\mathbf{G}^{\vx,j})^\dagger\big((\ms^\vx\circ\ms^\vx)^j\big)^{-1/2}\big\|^2_2,
\end{align*}
where we used the inequality $|\mathrm{tr}(\mathbf{V}^*\mathbf{U})| \leq  \big\|\mathbf{V}\big\|_2 \mathrm{tr}\big((\mathbf{U}^\dagger\mathbf{U})^{1/2}\big)$ for $\mathbf{V},\mathbf{U}\in\mathbb{C}^{p\times p}$ \cite{bhatia2013matrix} and
\begin{equation*}
    (\mathbf{L}^j)^\dagger\mathbf{L}^j=\mathbf{G}^{\vx,j}\mr{\nf}^j(\mathbf{G}^{\vx,j})^\dagger\big((\ms^\vx\circ\ms^\vx)^j\big)^{-2}\mathbf{G}^{\vx,j}\mr{\nf}^j(\mathbf{G}^{\vx,j})^\dagger.
\end{equation*}
However, algebraic manipulations reveal that
\begin{gather*}
    \big\|(\ms^\vx\circ\ms^\vx)^j\big)^{-1/2}\mathbf{G}^{\vx,j}\big\|^2_2\lesssim\big\|\big(\ms^\vx\circ\ms^\vx)^j\big)^{-1/2}\big\|^2_2\big\|\ml^{\vx,j}\big\|^2_2+\big\|\big((\ms^\vx\circ\ms^\vx)^j\big)^{-1/2}(\ms^\vx\circ\ms^\vy)^j\big\|^2_2\big\|\big(\mr{\nf}^j\big)^{-1}\big\|^2_2,\\
    \big\|\mr{\nf}^j(\mathbf{G}^{\vx,j})^\dagger\big((\ms^\vx\circ\ms^\vx)^j\big)^{-1/2}\big\|^2_2\lesssim\big\|\big(\ms^\vx\circ\ms^\vx)^j\big)^{-1/2}\big\|^2_2\big\|\ml^{\vx,j}\big\|^2_2\big\|\mr{\nf}^j\big\|^2_2+\big\|(\ms^\vy\circ\ms^\vx)^j\big((\ms^\vx\circ\ms^\vx)^j\big)^{-1/2}\big\|^2_2,
\end{gather*}
where by \textbf{\magenta{(5)}}, $\big\|\big((\ms^\vx\circ\ms^\vx)^j\big)^{-1}\big\|_2^2$ is uniformly bounded in time ($\pp$-a.s.), and so is $\big\|\ml^{\vx,j}\big\|_2^2$ by \textbf{\magenta{(2)}}, as well as $\big\|\mr{\nf}^j\big\|_2^2$ and $\big\|\big(\mr{\nf}^j\big)^{-1}\big\|_2^2$ as aforementioned, while \cite{kolodziej1980state}
\begin{equation*}   
\big\|\big((\ms^\vx\circ\ms^\vx)^j\big)^{-1/2}(\ms^\vx\circ\ms^\vy)^j\big\|^2_2=\big\|(\ms^\vy\circ\ms^\vx)^j\big((\ms^\vx\circ\ms^\vx)^j\big)^{-1/2}\big\|^2_2\lesssim \big\|\ms_1^{\vy,j}\big\|_2^2+\big\|\ms_2^{\vy,j}\big\|_2^2.
\end{equation*}
We then proceed in the same manner as in the preceding analysis as to deal with the term $\big\|\ms_1^{\vy,j}\big\|^2_2+\big\|\ms_2^{\vy,j}\big\|^2_2$ under assumption \textbf{\magenta{(13)}} for an appropriate random scalar $\Tilde{M}>0$, specifically to obtain a bound of the form
\begin{equation*}
    \sqrt{\mu_i((\mathbf{L}^j)^\dagger\mathbf{L}^j)}\leq \Tilde{M}\big(1+\big\|\ms_1^{\vy,j}\big\|_2^2+\big\|\ms_2^{\vy,j}\big\|_2^2\big),
\end{equation*}
for each $i=1,\ldots,k$ and $j=0,1,\ldots,J$, where now $\Tilde{M}$ depends on the uniform upper bounds, in time, of $\big\|\ml^\vx\big\|_2$, $\big\|\mr{\nf}\big\|_2$, $\big\|\mr{\nf}^{-1}\big\|_2$, but also of $\big\|(\ms^\vx\circ\ms^\vx)^{-1}\big\|_2$.
\end{proof}

The second lemma concerns the distributions of the backward conditionals of the unobserved process, as well as the Markov property that the partially observed system enjoys.
\begin{lem}[\textbf{Distribution of Backwards Conditionals}] \label{lem:backwardscond}
Let the assumptions in Theorem \ref{thm:smoothing} hold. Then the conditional distributions $\pp\big(\vy^j\big|\vy^{j+1},\vx^s,s\leq j\big)$ and $\pp\big(\vy^j\big|\vy^{j+1},\vx^s,s\leq j+1\big)$ are Gaussian, $\mathcal{N}_l(\mathbf{m}_j^j,\mathbf{P}_j^j)$ and $\mathcal{N}_l(\mathbf{m}_{j+1}^j,\mathbf{P}_{j+1}^j)$, respectively, where the subscripts denote up to which observation we are conditioning on, while the superscripts denote the time instant we are currently on in terms of the unobservable. Their conditional means $\mathbf{m}_{r}^j$ and conditional covariance matrices $\mathbf{P}_r^j$, for $r=j,j+1$, satisfy the following respective discrete-time equations:
\begin{align}
    \mathbf{m}_j^j&=\vm{\nf}^j+\mathbf{C}_j^j(\vy^{j+1}-(\mathbf{I}_{l\times l}+\ml^{\vy,j}\dt)\vm{\nf}^j-\vf^{\vy,j}\dt),\label{eq:backcond1} \\
    \mathbf{P}_j^j&=\mr{\nf}^j-\mathbf{C}_j^j(\mr{\nf}^j+(\ml^{\vy,j}\mr{\nf}^j+\mr{\nf}^j(\ml^{\vy,j})^{\dagger}+(\ms^\vy\circ\ms^\vy)^j)\dt)(\mathbf{C}_j^j)^\dagger,\label{eq:backcond2}
\end{align}
and
\begin{align}
    \mathbf{m}_{j+1}^j&=\vm{\nf}^j+\mathbf{C}_{j+1}^j\begin{pmatrix}
        \vy^{j+1}-(\mathbf{I}_{l\times l}+\ml^{\vy,j}\dt)\vm{\nf}^j-\vf^{\vy,j}\dt\\
        \vx^{j+1}-\vx^{j}-(\ml^{\vx,j}\vm{\nf}^j+\vf^{\vx,j})\dt
    \end{pmatrix},\label{eq:backcond3} \\
    \mathbf{P}_{j+1}^j&=\mr{\nf}^j-\mathbf{C}_{j+1}^j\mathbf{C}^j_{22}(\mathbf{C}_{j+1}^j)^\dagger,\label{eq:backcond4}
\end{align}
where the auxiliary matrices $\mathbf{C}_j^j\in\mathbb{C}^{l\times l}$, $\mathbf{C}_{j+1}^j\in\mathbb{C}^{l\times(k+l)}$, and $\mathbf{C}^j_{22}\in\mathbb{C}^{(k+l)\times(k+l)}$, are defined by
\begin{gather}
   \mathbf{C}_j^j:=\mr{\nf}^j(\mathbf{I}_{l\times l}+\ml^{\vy,j}\dt)^\dagger(\mr{\nf}^j+(\ml^{\vy,j}\mr{\nf}^j+\mr{\nf}^j(\ml^{\vy,j})^{\dagger}+(\ms^\vy\circ\ms^\vy)^j)\dt)^{-1}, \label{eq:auxiliarymat1} \\
   \mathbf{C}_{j+1}^j:=\mathbf{C}^j_{12}(\mathbf{C}^j_{22})^{-1}, \text{ where }\mathbf{C}^j_{12}:=\begin{pmatrix}
       \mr{\nf}^j(\mathbf{I}_{l\times l}+(\ml^{\vy,j})^\dagger\dt) &\hspace{0.4em} \mr{\nf}^j(\ml^{\vx,j})^\dagger\dt
   \end{pmatrix} \label{eq:auxiliarymat2}\\
   \text{and }\mathbf{C}^j_{22}:=\begin{pmatrix}
            \mr{\nf}^j+(\ml^{\vy,j}\mr{\nf}^j+\mr{\nf}^j(\ml^{\vy,j})^{\dagger}+(\ms^\vy\circ\ms^\vy)^j)\dt & (\mr{\nf}^j(\ml^{\vx,j})^{\dagger}+(\ms^\vy\circ\ms^\vx)^j)\dt\\
            (\ml^{\vx,j}\mr{\nf}^j+(\ms^\vx\circ\ms^\vy)^j)\dt & (\ms^\vx\circ\ms^\vx)^j\dt
    \end{pmatrix}. \label{eq:auxiliarymat4}
\end{gather}
Furthermore, if there is no noise cross-interaction term, i.e., $\ms^\vy\circ\ms^\vx\equiv \zeros_{l\times k}$, then the following distributions are asymptotically equal as $\dt\to0+$ (or, equivalently, as $J\to\infty$):
\begin{equation} \label{eq:firstordermarkov}
    \lim_{\substack{\dt\to0^+\\ J\dt=T}}\pp\big(\vy^j\big|\vy^{j+1},\vx^s, s\leq J\big)=\lim_{\substack{\dt\to0^+\\ J\dt=T}}\pp\big(\vy^j\big|\vy^{j+1},\vx^s, s\leq j\big),
\end{equation}
which means $\vy^j\big|\vy^{j+1}$ possesses the backwards in time Markov property of first-order w.r.t.\ natural filtration of $\vx$, in continuous time, while it does so up to $O(\dt)$, or to leading-order, in discrete time. In general, when we allow for noise cross-interaction, we instead have the weaker Markov property of
\begin{equation} \label{eq:secondordermarkov}
    \Big(\pp\big(\vy^j\big|\vy^{j+1},\vx^s, s\leq j\big)\neq\Big)\  \pp\big(\vy^j\big|\vy^{j+1},\vx^s, s\leq j+1\big)=\pp\big(\vy^j\big|\vy^{j+1},\vx^s, s\leq J\big),
\end{equation}
which means the smoother update depends on the present as well as immediate ``future" filter states, or that $\vy^j\big|\vy^{j+1}$ is a backwards in time second-order Markov chain w.r.t.\ natural filtration of $\vx$, in continuous time (where in discrete time this equality is actually exact for all $j$ and $\dt$).
\end{lem}
\begin{proof}[\textbf{Proof of Lemma \ref{lem:backwardscond}}]
We start with  the conditional distribution $\pp\big(\vy^j\big|\vy^{j+1},\vx^s,s\leq j\big)$. To do that, we shall first analyze the distribution $\pp\big(\vy^j,\vy^{j+1}\big|\vx^s,s\leq j\big)$. This distribution's Gaussianity can be established by the decomposition
\begin{equation*}
    \pp\big(\vy^j,\vy^{j+1}\big|\vx^s,s\leq j\big)=\pp\big(\vy^{j+1}\big|\vy^j,\vx^s,s\leq j\big)\pp\big(\vy^j\big|\vx^s,s\leq j\big),
\end{equation*}
and the enforcement of Lemma \ref{lem:bayesian}, since each conditional on the right-hand side is Gaussian. Indeed, the first one is Gaussian as an immediate consequence of Lemma \ref{lem:affinity} and \eqref{eq:discretecondgauss2}, while the second one is simply the Gaussian filter distribution. Specifically, in this case we let $\mathbf{u}_2\overset{\rmd}{=}\vy^{j+1}\big|\vy^{j},\vx^s,s\leq j$ and $\mathbf{u}_1\overset{\rmd}{=}\vy^j$ in Lemma \ref{lem:bayesian}, and observe how the mean of $\vy^{j+1}$ is indeed an affine transformation of $\vy^j$, when conditioned on $\vy^j$ and $\vx^s$ for $s\leq j$, by virtue of \eqref{eq:discretecondgauss2}. Otherwise, we can explicitly prove the conditional Gaussianity using the conditional characteristic function method \cite{yuan2016some}. As such we just need to find its Gaussian statistics to fully describe this distribution. We already know from \eqref{eq:cg1} that
\begin{equation*}
    \pp\big(\vy^{j+1}\big|\vx^{s},s\leq j\big)\overset{\rmd}{\sim}\mathcal{N}_l\big(\vm{\nf}^j+(\ml^{\vy,j}\vm{\nf}^j+\vf^{\vy,j})\dt,\mr{\nf}^j+(\ml^{\vy,j}\mr{\nf}^j+\mr{\nf}^j(\ml^{\vy,j})^{\dagger}+(\ms^\vy\circ\ms^\vy)^j)\dt\big),
\end{equation*}
while $\pp\big(\vy^{j}\big|\vx^{s},s\leq j\big)$ is simply given by the filtering formula,
\begin{equation*}
    \pp\big(\vy^{j}\big|\vx^{s},s\leq j\big)\overset{\rmd}{\sim}\mathcal{N}_l(\vm{\nf}^j,\mr{\nf}^j).
\end{equation*}
As for the cross-covariance term, similarly to how we derived \eqref{eq:cg3}, by multiplying \eqref{eq:discretefluc2} by $(\vy^{\prime, j})^\dagger$ from the right-hand side and taking the average when conditioning on $\vx^s,s\leq j$, gives
\begin{equation*}
    \overline{\vy^{\prime, j+1}(\vy^{\prime, j})^\dagger}=(\mathbf{I}_{l\times l}+\ml^{\vy,j}\dt)\mr{\nf}^j.
\end{equation*}
Therefore, by collecting all results thus far leads to
\begin{equation}
    \pp\big(\vy^{j},\vy^{j+1}\big|\vx^s,s\leq j\big)\overset{\rmd}{\sim}\mathcal{N}_{2l}\left( \begin{pmatrix}
        \vm{\nf}^j\\
        (\mathbf{I}_{l\times l}+\ml^{\vy,j}\dt)\vm{\nf}^j+\vf^{\vy,j}\dt
    \end{pmatrix},\begin{pmatrix}
        \mr{\nf}^j & \mr{\nf}^j(\mathbf{I}_{l\times l}+\ml^{\vy,j}\dt)^\dagger \\
        (\mathbf{I}_{l\times l}+\ml^{\vy,j}\dt)\mr{\nf}^j & \mr{\nf}^j(\mathbf{I}_{l\times l}+\mathbf{H}^j\dt)
    \end{pmatrix}\right), \label{eq:jointgaussianitydetails1}
\end{equation}
for $\mathbf{H}^j$ as in the proof of Theorem \ref{thm:filtering}. Applying then Lemma \ref{lem:conditional} to the conclusion \eqref{eq:jointgaussianitydetails1}, yields the desired conditional distribution
\begin{equation*}
   \pp\big(\vy^j\big|\vy^{j+1},\vx^s, s\leq j\big)\overset{\rmd}{\sim} \mathcal{N}_l(\mathbf{m}_j^j,\mathbf{P}_j^j),
\end{equation*}
where
\begin{align*}
    \mathbf{m}_j^j&=\vm{\nf}^j+\mathbf{C}_j^j(\vy^{j+1}-(\mathbf{I}_{l\times l}+\ml^{\vy,j}\dt)\vm{\nf}^j-\vf^{\vy,j}\dt),\\
    \mathbf{P}_j^j&=\mr{\nf}^j-\mathbf{C}_j^j(\mr{\nf}^j+(\ml^{\vy,j}\mr{\nf}^j+\mr{\nf}^j(\ml^{\vy,j})^{\dagger}+(\ms^\vy\circ\ms^\vy)^j)\dt)(\mathbf{C}_j^j)^\dagger,
\end{align*}
with the auxiliary matrix $\mathbf{C}_j^j$ given by
\begin{equation*}
   \mathbf{C}_j^j:=\mr{\nf}^j(\mathbf{I}_{l\times l}+\ml^{\vy,j}\dt)^\dagger(\mr{\nf}^j+(\ml^{\vy,j}\mr{\nf}^j+\mr{\nf}^j(\ml^{\vy,j})^{\dagger}+(\ms^\vy\circ\ms^\vy)^j)\dt)^{-1}=\mr{\nf}^j(\mathbf{I}_{l\times l}+\ml^{\vy,j}\dt)^\dagger(\mr{\nf}^j(\mathbf{I}_{l\times l}+\mathbf{H}^{j}\dt))^{-1}.
\end{equation*}
These establish \eqref{eq:backcond1}--\eqref{eq:backcond2} and \eqref{eq:auxiliarymat1}.

We proceed with $\pp\big(\vy^j\big|\vy^{j+1},\vx^s,s\leq j+1\big)$. For that, we first analyze ${\pp\big(\vy^j,\vy^{j+1},\vx^{j+1}\big|\vx^s,s\leq j\big)}$. The Gaussianity of this distribution, in a parallel manner as before, can be established by its definition,
\begin{equation*}
    \pp\big(\vy^j,\vy^{j+1},\vx^{j+1}\big|\vx^s,s\leq j\big)=\pp\big(\vy^{j+1},\vx^{j+1}\big|\vy^j,\vx^s,s\leq j\big)\pp\big(\vy^j\big|\vx^s,s\leq j\big),
\end{equation*}
and enforcing Lemma \ref{lem:bayesian}, since each conditional on the right-hand side is Gaussian, with the first being Gaussian as shown in the proof of Theorem 13.3 in Liptser \& Shiryaev \cite{liptser2001statisticsII} while the second one is simply the filter posterior distribution which is Gaussian by Theorem \ref{thm:filtering}. To do so, in this case we let $\mathbf{u}_2=((\vy^{j+1})^\tran,(\vx^{j+1})^\tran)^\tran$ and $\mathbf{u}_1=\vy^j$ in Lemma \ref{lem:bayesian}, and observe how the mean of $\mathbf{u}_2$ is indeed an affine transformation of $\mathbf{u}_1=\vy^j$, when conditioned on $\vy^j$ and $\vx^s$ for $s\leq j$, by virtue of \eqref{eq:discretecondgauss1}--\eqref{eq:discretecondgauss2}. As such, having established the normality of this joint distribution, we just need to calculate its Gaussian moments. Marginally, we already know from the proof of Theorem \ref{thm:filtering} that
\begin{align*}
    \pp\big(\vy^{j}\big|\vx^{s},s\leq j\big)&\overset{\rmd}{\sim}\mathcal{N}_l(\vm{\nf}^j,\mr{\nf}^j),\\
    \pp\big(\vy^{j+1}\big|\vx^{s},s\leq j\big)&\overset{\rmd}{\sim}\mathcal{N}_l\big(\vm{\nf}^j+(\ml^{\vy,j}\vm{\nf}^j+\vf^{\vy,j})\dt,\mr{\nf}^j+(\ml^{\vy,j}\mr{\nf}^j+\mr{\nf}^j(\ml^{\vy,j})^{\dagger}+(\ms^\vy\circ\ms^\vy)^j)\dt\big),\\
    \pp\big(\vx^{j+1}\big|\vx^{s},s\leq j\big)&\overset{\rmd}{\sim}\mathcal{N}_k\big(\vx^j+(\ml^{\vx,j}\vm{\nf}^j+\vf^{\vx,j})\dt,(\ms^\vx\circ\ms^\vx)^j\dt\big).
\end{align*}
As for the cross-covariance terms, those are given by a similar procedure as before as
\begin{equation*}
    \overline{\vy^{\prime, j+1}(\vy^{\prime, j})^\dagger}=(\mathbf{I}_{l\times l}+\ml^{\vy,j}\dt)\mr{\nf}^j \quad\text{and} \quad\overline{\vx^{\prime, j+1}(\vy^{\prime, j})^\dagger}=\ml^{\vx,j}\mr{\nf}^j\dt.
\end{equation*}
We also have the conditionally Gaussian joint distribution $p(\vy^{j+1},\vx^{j+1}|\vx^s,s\leq j)$ given in \eqref{eq:jointcondgaussianity}--\eqref{eq:jointcondgaussianitydetails}. As such, by gathering all of the established results, we end up with
\begin{equation} \label{eq:jointgaussianitydetails2}
        \pp\big(\vy^{j},\vy^{j+1},\vx^{j+1}\big|\vx^s,s\leq j\big)\overset{\rmd}{\sim}\mathcal{N}_{2l+k}\left( \begin{pmatrix}
            \vm{\nf}^j\\
            \vm{\nf}^j+\left(\ml^{\vy,j}\vm{\nf}^j+\vf^{\vy,j}\right)\dt\\
            \vx^j+\left(\ml^{\vx,j}\vm{\nf}^j+\vf^{\vx,j}\right)\dt
        \end{pmatrix},\begin{pmatrix}
            \mr{\nf}^j & \mathbf{C}^j_{12} \\
            \mathbf{C}^j_{21} & \mathbf{C}^j_{22}
        \end{pmatrix}\right),
\end{equation}
where
\begin{gather*}
    \mathbf{C}^j_{12}:=\mathrm{Cov}\big(\vy^j,(\vy^{j+1},\vx^{j+1})\big|\vx^s,s\leq j\big)=\begin{pmatrix}
    \mr{\nf}^j(\mathbf{I}_{l\times l}+(\ml^{\vy,j})^\dagger\dt) & \hspace{0.4em} \mr{\nf}^j(\ml^{\vx,j})^\dagger\dt
    \end{pmatrix},\quad \mathbf{C}^j_{21}:=(\mathbf{C}^j_{12})^\dagger,\\
    \hspace*{-1cm}
    \mathbf{C}^j_{22}:=\mathrm{Var}\big(\vy^{j+1},\vx^{j+1}\big|\vx^s,s\leq j\big)=\begin{pmatrix}
            \mr{\nf}^j\Big(\mathbf{I}_{l\times l}+\mathbf{H}^j\dt\Big) & \mr{\nf}^j(\mathbf{G}^{\vx,j})^\dagger\dt\\
            \mathbf{G}^{\vx,j}\mr{\nf}^j\dt & (\ms^\vx\circ\ms^\vx)^j\dt
        \end{pmatrix},
\end{gather*}
for the matrices $\mathbf{G}^{\vx,j}$ and $\mathbf{H}^j$ as defined in the proof of Theorem \ref{thm:filtering}. Once again, we then utilize Lemma \ref{lem:conditional} to the conclusion in \eqref{eq:jointgaussianitydetails2} as to retrieve the desired conditional distribution:
\begin{equation*}
    \pp\big(\vy^j\big|\vy^{j+1},\vx^s, s\leq j+1\big)\overset{\rmd}{\sim} \mathcal{N}_l(\mathbf{m}_{j+1}^j,\mathbf{P}_{j+1}^j),
\end{equation*}
where
\begin{equation*}
    \mathbf{m}_{j+1}^j=\vm{\nf}^j+\mathbf{C}_{j+1}^j\begin{pmatrix}
        \vy^{j+1}-(\mathbf{I}_{l\times l}+\ml^{\vy,j}\dt)\vm{\nf}^j-\vf^{\vy,j}\dt\\
        \vx^{j+1}-\vx^{j}-(\ml^{\vx,j}\vm{\nf}^j+\vf^{\vx,j})\dt
    \end{pmatrix}, \quad \mathbf{P}_{j+1}^j=\mr{\nf}^j-\mathbf{C}_{j+1}^j(\mathbf{C}^j_{12})^\dagger=\mr{\nf}^j-\mathbf{C}_{j+1}^j\mathbf{C}^j_{22}(\mathbf{C}_{j+1}^j)^\dagger,
\end{equation*}
with the auxiliary matrix $\mathbf{C}_{j+1}^j$ defined as $ \mathbf{C}_{j+1}^j:=\mathbf{C}^j_{12}(\mathbf{C}^j_{22})^{-1}$. These establish the rest of the relations, \eqref{eq:backcond3}--\eqref{eq:backcond4} and \eqref{eq:auxiliarymat2}--\eqref{eq:auxiliarymat4}.

Having established \eqref{eq:backcond1}--\eqref{eq:auxiliarymat4}, we now prove the statements concerning the Markov property of the backward conditionally Gaussian distributions of interest, \eqref{eq:firstordermarkov} and \eqref{eq:secondordermarkov}.  First, observe that for any $j=0,1,\ldots,J-1$, we have
\begin{equation} \label{eq:markovproperty}
     \pp\big(\vy^j\big|\vy^{j+1},\vx^s, s\leq J\big)\equiv \pp\big(\vy^j\big|\vy^{j+1},\vx^s, s\leq j+1\big).
\end{equation}
This equality in distribution in \eqref{eq:markovproperty} has the immediate consequence that $\vy^j\big|\vy^{j+1}$ is a backwards in time (the Markov property is intrinsically time-symmetric) second-order Markov chain w.r.t.\ natural filtration of $\vx$, in discrete time for all $J\in\mathbb{N}$. This is in general true, due to the fact that the information of $\vx^s, s\geq j+1$ is already included in the optimal filter estimated state of $\vy^{j+1}$ under our regime of working with the observational $\sigma$-algebra induced by $\{\vx^s\}_{s=0,1,\ldots,j}$, which is the natural filtration of choice in partially observed stochastic dynamical systems, but also includes the uncertainty in $\vx^{j+1}$. Essentially, when at the time step $t_j$, the conditional terms $\vy^{j+1}$ and $\vx^s, s\leq j+1$, resolve the information in the future and past, respectively, fully. To rigorously see why \eqref{eq:markovproperty} is true, by \eqref{eq:discretecondgauss1}--\eqref{eq:discretecondgauss2}, when conditioning on $\vy^{j+1}$ and $\vx^s$ for $s=0,1,\ldots,J$, by property of the filtration induced by the observational process, it is immediate that $\vy^j$ is fully determined by the information conditioned up to $t_{j+1}$, specifically by $\vy^{j+1}$ and $\vx^s$ for $s=0,1,\ldots,j+1$. Beyond this result though, it is possible to show that in the absence of noise cross-interaction, i.e., $\ms^\vy\circ\ms^\vx\equiv \zeros_{l\times k}$, we actually have that
\begin{equation*}
    \lim_{\substack{\dt\to0^+\\ J\dt=T}}\pp\big(\vy^j\big|\vy^{j+1},\vx^s, s\leq J\big)=\lim_{\substack{\dt\to0^+\\ J\dt=T}}\pp\big(\vy^j\big|\vy^{j+1},\vx^s, s\leq j\big),
\end{equation*}
or equivalently when $J\to\infty$, which is analogous to saying that  $\vy^j\big|\vy^{j+1}$ possesses the (backwards in time) Markov property of first-order w.r.t.\ natural filtration of $\vx$, in continuous-time (and up to $O(\dt)$, or to leading-order, in discrete-time). To contrast this with \eqref{eq:markovproperty}, when $\ms^\vy\circ\ms^\vx\not\equiv \zeros_{l\times k}$, then for the smoother update it is necessary to augment the information by $\vx^{j+1}$ (and only), since $\vy^{j+1}$ no longer contains all the information about the future states on its own. We mention that a similar result is shown in the proof of Theorem 13.11 (specifically (13.126)) of Liptser \& Shiryaev \cite{liptser2001statisticsII} for the discrete-time case; specifically, under our notation, we have that $\pp\big(\vy^j\big|\vx^s,s\leq J, \vy^N,\ldots,\vy^J\big)=\pp\big(\vy^j\big|\vx^s,s\leq N, \vy^N\big)$, for $j<N\leq J$, which recovers \eqref{eq:markovproperty} for $N=j+1$ and $j+1=1,\ldots,J$. Here we instead prove the result in a martingale-free approach, and also obtain a pertinent condition for the continuous-time case in the process.

By virtue of block-matrix inversion through Schur complements, we see that if we define
\begin{align}
    \begin{split}
    \mathbf{C}^j_{22}&=\begin{pmatrix}
        \mr{\nf}^j+(\ml^{\vy,j}\mr{\nf}^j+\mr{\nf}^j(\ml^{\vy,j})^{\dagger}+(\ms^\vy\circ\ms^\vy)^j)\dt & (\mr{\nf}^j(\ml^{\vx,j})^{\dagger}+(\ms^\vy\circ\ms^\vx)^j)\dt\\
        (\ml^{\vx,j}\mr{\nf}^j+(\ms^\vx\circ\ms^\vy)^j)\dt & (\ms^\vx\circ\ms^\vx)^j\dt
    \end{pmatrix}\\
    &=:\begin{pmatrix}
        \left(\mathbf{C}^j_{22}\right)_{11} & \left(\mathbf{C}^j_{22}\right)_{12}\\
        \\
        \left(\mathbf{C}^j_{22}\right)^{\hspace{0.4em} \tran}_{12} & \left(\mathbf{C}^j_{22}\right)_{22}
    \end{pmatrix},
    \end{split} \label{eq:blockcovmatrix}
\end{align}
then
\begin{align}
\begin{split}
    (\mathbf{C}^j_{22})^{-1}&=\begin{pmatrix}
        \left((\mathbf{C}^j_{22})^{-1}\right)_{11} & \left((\mathbf{C}^j_{22})^{-1}\right)_{12}\\
        \left((\mathbf{C}^j_{22})^{-1}\right)_{21} & \left((\mathbf{C}^j_{22})^{-1}\right)_{22}
    \end{pmatrix}\\
    &:=\begin{pmatrix}
        \left(\mathbf{C}^j_{22}\right)^{\hspace{0.4em} -1}_{11}+\left(\mathbf{C}^j_{22}\right)^{\hspace{0.4em} -1}_{11}\left(\mathbf{C}^j_{22}\right)_{12} \mathbf{D}^j\left(\mathbf{C}^j_{22}\right)_{12}^{\hspace{0.4em} \tran}\left(\mathbf{C}^j_{22}\right)^{\hspace{0.4em} -1}_{11} & \hspace{0.5cm}-\left(\mathbf{C}^j_{22}\right)^{\hspace{0.4em} -1}_{11}\left(\mathbf{C}^j_{22}\right)_{12} \mathbf{D}^j\\
        \\
        -\mathbf{D}^j\left(\mathbf{C}^j_{22}\right)_{12}^{\hspace{0.4em} \tran}\left(\mathbf{C}^j_{22}\right)^{\hspace{0.4em} -1}_{11} & \mathbf{D}^j
    \end{pmatrix},
\end{split} \label{eq:blockcovmatrixinverse}
\end{align}
where
\begin{equation} \label{eq:schurcomplement}
    \mathbf{D}^j:=\left(\left(\mathbf{C}^j_{22}\right)_{22}-\left(\mathbf{C}^j_{22}\right)^{\hspace{0.4em} \tran}_{12}\left(\mathbf{C}^j_{22}\right)^{\hspace{0.4em} -1}_{11}\left(\mathbf{C}^j_{22}\right)_{12}\right)^{-1},
\end{equation}
is the inverse of the Schur complement of $\left(\mathbf{C}^j_{22}\right)_{11}$ in $\mathbf{C}^j_{22}$ (which is Hermitian and positive-definite since under our assumptions so is $\mathbf{C}^j_{22}$). Since $\pp\big(\vy^j\big|\vy^{j+1},\vx^s, s\leq J\big)\equiv \pp\big(\vy^j\big|\vy^{j+1},\vx^s, s\leq j+1\big)$ always holds, if we were to assume $\ms^\vy\circ\ms^\vx\equiv \zeros_{l\times k}$, then to establish that in this regime the process $\vy^j\big|\vy^{j+1}$ enjoys the first-order Markov property in continuous-time we just need to prove
\begin{equation*}
    \lim_{\substack{\dt\to0^+\\ J\dt=T}}\pp\big(\vy^j\big|\vy^{j+1},\vx^s, s\leq j\big)=\lim_{\substack{\dt\to0^+\\ J\dt=T}}\pp\big(\vy^j\big|\vy^{j+1},\vx^s, s\leq j+1\big).
\end{equation*}
Based on the expressions we have in \eqref{eq:backcond1}--\eqref{eq:auxiliarymat4} and due to the conditional normality of the distributions, we would just need to establish equality of the statistics up to the second-order moment, i.e., we would just need to prove that
\begin{equation*}
    \lim_{\substack{\dt\to0^+\\ J\dt=T}}\mathbf{m}^j_j=\lim_{\substack{\dt\to0^+\\ J\dt=T}}\mathbf{m}^j_{j+1} \quad \text{and}\quad\lim_{\substack{\dt\to0^+\\ J\dt=T}}\mathbf{P}^j_j=\lim_{\substack{\dt\to0^+\\ J\dt=T}}\mathbf{P}^j_{j+1},
\end{equation*}
or, equivalently and respectively, that
\begin{gather*}
    \lim_{\substack{\dt\to0^+\\ J\dt=T}}\mathbf{C}_j^j(\vy^{j+1}-(\mathbf{I}_{l\times l}+\ml^{\vy,j}\dt)\vm{\nf}^j-\vf^{\vy,j}\dt)=\lim_{\substack{\dt\to0^+\\ J\dt=T}}\mathbf{C}_{j+1}^j\begin{pmatrix}
        \vy^{j+1}-(\mathbf{I}_{l\times l}+\ml^{\vy,j}\dt)\vm{\nf}^j-\vf^{\vy,j}\dt\\
        \vx^{j+1}-\vx^{j}-(\ml^{\vx,j}\vm{\nf}^j+\vf^{\vx,j})\dt
    \end{pmatrix},\\
    \lim_{\substack{\dt\to0^+\\ J\dt=T}}\mathbf{C}_j^j\left(\mr{\nf}^j+\left(\ml^{\vy,j}\mr{\nf}^j+\mr{\nf}^j(\ml^{\vy,j})^{\dagger}+(\ms^\vy\circ\ms^\vy)^j\right)\dt\right)(\mathbf{C}_j^j)^\dagger=\lim_{\substack{\dt\to0^+\\ J\dt=T}}\mathbf{C}_{j+1}^j\mathbf{C}^j_{22}(\mathbf{C}_{j+1}^j)^\dagger\left(=\lim_{\substack{\dt\to0^+\\ J\dt=T}}\mathbf{C}_{j+1}^j(\mathbf{C}^j_{12})^\dagger\right).
\end{gather*}
We first expand and simplify, up to first-order terms (i.e., up to $O(\dt)$), the expressions of all the auxiliary matrices appearing in \eqref{eq:auxiliarymat1}, \eqref{eq:auxiliarymat2}, and \eqref{eq:blockcovmatrix}--\eqref{eq:schurcomplement}. These are not only needed for the proof of the Markovian properties here, but also required to prove the expressions of the discrete smoother mean and covariance matrix (see Lemma \ref{lem:discretesmoother}). We start by expanding the auxiliary matrix $\mathbf{C}_j^j$, where after suppressing terms of order higher than $O(\dt)$ when carrying out the series expansion, it is equivalently given by
\begin{align}
\begin{split} \label{eq:auxiliarymatexpanded1}
   \mathbf{C}_j^j&=\mr{\nf}^j(\mathbf{I}_{l\times l}+\ml^{\vy,j}\dt)^\dagger\big(\mr{\nf}^j+(\ml^{\vy,j}\mr{\nf}^j+\mr{\nf}^j(\ml^{\vy,j})^{\dagger}+(\ms^\vy\circ\ms^\vy)^j)\dt\big)^{-1} \\
   &=\mr{\nf}^j(\mathbf{I}_{l\times l}+\ml^{\vy,j}\dt)^\dagger\left[\mr{\nf}^j\big(\mathbf{I}_{l\times l}+(\mr{\nf}^j)^{-1}(\ml^{\vy,j}\mr{\nf}^j+\mr{\nf}^j(\ml^{\vy,j})^{\dagger}+(\ms^\vy\circ\ms^\vy)^j)\dt\big)\right]^{-1} \\
   &=\mr{\nf}^j(\mathbf{I}_{l\times l}+(\ml^{\vy,j})^\dagger\dt)\big(\mathbf{I}_{l\times l}+\big(\mr{\nf}^j\big)^{-1}(\ml^{\vy,j}\mr{\nf}^j+\mr{\nf}^j(\ml^{\vy,j})^{\dagger}+(\ms^\vy\circ\ms^\vy)^j)\dt\big)^{-1}\big(\mr{\nf}^j\big)^{-1} \\
   &= \mr{\nf}^j(\mathbf{I}_{l\times l}+(\ml^{\vy,j})^\dagger\dt)\big(\mathbf{I}_{l\times l}-\big(\mr{\nf}^j\big)^{-1}(\ml^{\vy,j}\mr{\nf}^j+\mr{\nf}^j(\ml^{\vy,j})^{\dagger}+(\ms^\vy\circ\ms^\vy)^j)\dt+O(\dt^2)\big)\big(\mr{\nf}^j\big)^{-1} \\
   &=\mr{\nf}^j\big(\mathbf{I}_{l\times l}+(\ml^{\vy,j})^\dagger\dt-\big(\mr{\nf}^j\big)^{-1}(\ml^{\vy,j}\mr{\nf}^j+\mr{\nf}^j(\ml^{\vy,j})^{\dagger}+(\ms^\vy\circ\ms^\vy)^j)\dt+O(\dt^2)\big)\big(\mr{\nf}^j\big)^{-1} \\
   &=\big(\mr{\nf}^j+\mr{\nf}^j(\ml^{\vy,j})^\dagger\dt-\ml^{\vy,j}\mr{\nf}^j\dt-\mr{\nf}^j(\ml^{\vy,j})^{\dagger}\dt-(\ms^\vy\circ\ms^\vy)^j\dt+O(\dt^2)\big)\big(\mr{\nf}^j\big)^{-1} \\
   &=\big(\mr{\nf}^j-\ml^{\vy,j}\mr{\nf}^j\dt-(\ms^\vy\circ\ms^\vy)^j\dt+O(\dt^2)\big)\big(\mr{\nf}^j\big)^{-1}  \\
   &= \mathbf{I}_{l\times l}-\big(\ml^{\vy,j}+(\ms^\vy\circ\ms^\vy)^j\big(\mr{\nf}^j\big)^{-1}\big)\dt+O(\dt^2),
\end{split}
\end{align}
where we have used Lemma \ref{lem:tailboundsspectral}, which concerns the spectrum of $\mathbf{H}$, as to establish the ($\pp$-a.s.) leading-order expansion of the inverse $(\mathbf{I}_{l\times l} + \mathbf{H}^j\dt)^{-1}$ through the Neumann series of $-\mathbf{H}^j\dt$, for sufficiently small $\dt>0$ \cite{suzuki1976convergence, engl1985successive, stewart1998matrix}. As such, from \eqref{eq:auxiliarymatexpanded1}, for $\mathbf{C}^j_j$ we have that
\begin{equation*}
    \mathbf{C}^j_j=\mathbf{I}_{l\times l}-\mathbf{G}^{\vy,j}\dt+O(\dt^2),
\end{equation*}
where $\mathbf{G}^{\vy,j}:=\ml^{\vy,j}+(\ms^\vy\circ\ms^\vy)^j\big(\mr{\nf}^j\big)^{-1}$.

Based on the definitions of the auxiliary matrices defined in \eqref{eq:blockcovmatrix}, it is immediate that
\begin{align*}
    \left(\mathbf{C}^j_{22}\right)_{11}&:=\mr{\nf}^j+(\ml^{\vy,j}\mr{\nf}^j+\mr{\nf}^j(\ml^{\vy,j})^{\dagger}+(\ms^\vy\circ\ms^\vy)^j)\dt\equiv\mr{\nf}^j(\mathbf{I}_{l\times l}+\mathbf{H}^j\dt),\\
    \left(\mathbf{C}^j_{22}\right)^{\hspace{0.4em} \tran}_{12}&:=(\ml^{\vx,j}\mr{\nf}^j+(\ms^\vx\circ\ms^\vy)^j)\dt\equiv\mathbf{G}^{\vx,j}\mr{\nf}^j\dt.
\end{align*}
Using then the same procedure as in \eqref{eq:auxiliarymatexpanded1}, we then have
\begin{align}
\begin{split} \label{eq:auxiliarymatexpanded2}
    \left(\mathbf{C}^j_{22}\right)^{\hspace{0.4em} \tran}_{12}\left(\mathbf{C}^j_{22}\right)^{\hspace{0.4em} -1}_{11}&=\mathbf{G}^{\vx,j}\mr{\nf}^j\left(\mathbf{I}_{l\times l}+\mathbf{H}^j\dt\right)^{-1}\big(\mr{\nf}^j\big)^{-1}\dt\\
    &=\mathbf{G}^{\vx,j}\mr{\nf}^j\left(\mathbf{I}_{l\times l}-\mathbf{H}^j\dt+O(\dt^2)\right)\big(\mr{\nf}^j\big)^{-1}\dt\\
    &=\mathbf{G}^{\vx,j}\dt-\mathbf{G}^{\vx,j}\mr{\nf}^j\mathbf{H}^j\big(\mr{\nf}^j\big)^{-1}\dt^2+O(\dt^3),
\end{split}
\end{align}
and
\begin{equation} \label{eq:auxiliarymatexpanded3}
    \left(\mathbf{C}^j_{22}\right)^{\hspace{0.4em} \tran}_{12}\left(\mathbf{C}^j_{22}\right)^{\hspace{0.4em} -1}_{11}\left(\mathbf{C}^j_{22}\right)_{12}=\mathbf{G}^{\vx,j}\mr{\nf}^j(\mathbf{G}^{\vx,j})^\dagger\dt^2+O(\dt^3),
\end{equation}
where we have again used Lemma \ref{lem:tailboundsspectral}. While our scheme only considers first-order terms as of leading-order, we nevertheless explicitly state the form of the quadratic terms in $\dt$ for these auxiliary matrices, as they are needed in the subsequent analysis. Before proceeding, we define for notational simplicity the auxiliary matrix $\mathbf{K}^j:=\big((\ms^\vx\circ\ms^\vx)^{j}\big)^{-1}\mathbf{G}^{\vx,j}$. For $\mathbf{D}^j$, which as aforementioned is the inverse of the Schur complement of $\left(\mathbf{C}^j_{22}\right)_{11}$ in $(\mathbf{C}^j_{22})$, we then have
\begin{align}
\begin{split} \label{eq:auxiliarymatexpanded4}
    \mathbf{D}^j&:=\left(\left(\mathbf{C}^j_{22}\right)_{22}-\left(\mathbf{C}^j_{22}\right)^{\hspace{0.4em} \tran}_{12}\left(\mathbf{C}^j_{22}\right)^{\hspace{0.4em} -1}_{11}\left(\mathbf{C}^j_{22}\right)_{12}\right)^{-1}\\
    &=\big((\ms^\vx\circ\ms^\vx)^j\dt - \mathbf{G}^{\vx,j}\mr{\nf}^j(\mathbf{G}^{\vx,j})^\dagger\dt^2\big)^{-1}\\
    &= \big(\mathbf{I}_{k\times k}- \underbrace{\big((\ms^\vx\circ\ms^\vx)^{j}\big)^{-1}\mathbf{G}^{\vx,j}\mr{\nf}^j(\mathbf{G}^{\vx,j})^\dagger}_{\displaystyle \mathclap{=\mathbf{K}^j\mr{\nf}^j(\mathbf{G}^{\vx,j})^\dagger\equiv\mathbf{L}^j}}\dt\big)^{-1}\big((\ms^\vx\circ\ms^\vx)^{j}\big)^{-1}/\dt\\
    &= \big(\mathbf{I}_{k\times k}+ \mathbf{L}^j\dt-(\mathbf{L}^j)^2\dt^2+O(\dt^3)\big)\big((\ms^\vx\circ\ms^\vx)^{j}\big)^{-1}/\dt\\
    &= \left(\frac{1}{\dt}\mathbf{I}_{k\times k}+ \mathbf{L}^j-(\mathbf{L}^j)^2\dt+O(\dt^2)\right)\big((\ms^\vx\circ\ms^\vx)^{j}\big)^{-1}\\
    &=\frac{1}{\dt}\big((\ms^\vx\circ\ms^\vx)^{j}\big)^{-1}+\mathbf{K}^j\mr{\nf}^j(\mathbf{K}^j)^\dagger+\mathbf{L}^j\mathbf{K}^j\mr{\nf}^j(\mathbf{K}^j)^\dagger\dt+O(\dt^2),
\end{split}
\end{align}
where we have used Lemma \ref{lem:tailboundsspectral}, which concerns the spectral properties of $\mathbf{L}$, as to establish the ($\pp$-a.s.) leading-order expansion of the inverse $(\mathbf{I}_{l\times l} - \mathbf{L}^j\dt)^{-1}$ through the Neumann series of $\mathbf{L}^j\dt$, for sufficiently small $\dt>0$, while for the last equality we used that $\mathbf{L}^j\big((\ms^\vx\circ\ms^\vx)^{j}\big)^{-1}=\mathbf{K}^j\mr{\nf}^j(\mathbf{K}^j)^\dagger$. Notice that since $\mathbf{L}^j\mathbf{K}^j\mr{\nf}^j(\mathbf{K}^j)^\dagger=\mathbf{Q}^j\big((\ms^\vx\circ\ms^\vx)^{j}\big)^{-1}(\mathbf{Q}^j)^\dagger$, for $\mathbf{Q}^j=\big((\ms^\vx\circ\ms^\vx)^{j}\big)^{-1}\mathbf{G}^{\vx,j}\mr{\nf}^j(\mathbf{G}^{\vx,j})^\dagger$, then the symmetry of $\mathbf{D}^j$ (as the Schur complement of a covariance matrix it is Hermitian and positive-definite) is also apparent by its leading-order expansion. By combining the expressions we have derived thus far, we end up with
\begin{align}
\begin{split} \label{eq:auxiliarymatexpanded5}
    \mathbf{D}^j\left(\mathbf{C}^j_{22}\right)^{\hspace{0.4em} \tran}_{12}\left(\mathbf{C}^j_{22}\right)^{\hspace{0.4em} -1}_{11}&=\left(\frac{1}{\dt}\big((\ms^\vx\circ\ms^\vx)^{j}\big)^{-1}+\mathbf{K}^j\mr{\nf}^j(\mathbf{K}^j)^\dagger+O(\dt)\right)\\
    &\hspace{2cm}\times\big(\mathbf{G}^{\vx,j}\dt-\mathbf{G}^{\vx,j}\mr{\nf}^j\mathbf{H}^j\big(\mr{\nf}^j\big)^{-1}\dt^2+O(\dt^3)\big)\\
    &=\big((\ms^\vx\circ\ms^\vx)^{j}\big)^{-1}\mathbf{G}^{\vx,j}\\
    &\hspace{2cm}+\big(\mathbf{K}^j\mr{\nf}^j(\mathbf{K}^j)^\dagger\mathbf{G}^{\vx,j}-(\ms^\vx\circ\ms^\vx)^j)^{-1}\mathbf{G}^{\vx,j}\mr{\nf}^j\mathbf{H}^j\big(\mr{\nf}^j\big)^{-1}\big)\dt+O(\dt^2)\\
    &= \mathbf{K}^j+\big(\mathbf{K}^j\mr{\nf}^j(\mathbf{K}^j)^\dagger\mathbf{G}^{\vx,j}-\mathbf{K}^j\mr{\nf}^j\mathbf{H}^j\big(\mr{\nf}^j\big)^{-1}\big)\dt+O(\dt^2).
\end{split}
\end{align}

As such, when we assume $\ms^\vy\circ\ms^\vx\equiv \zeros_{l\times k}$, and if in block-matrix form we have
\begin{equation} \label{eq:auxiliaryblocks}
    \mathbf{C}_{j+1}^j=\mathbf{C}^j_{12}(\mathbf{C}^j_{22})^{-1}=\begin{pmatrix}
        \mathbf{E}^j & \hspace{0.4em} \mathbf{F}^j
    \end{pmatrix}\in\mathbb{C}^{l\times(l+k)},
\end{equation}
for $\mathbf{E}^j\in\mathbb{C}^{l\times l}$ and $\mathbf{F}^j\in\mathbb{C}^{l \times k}$, then by block-matrix multiplication we have
\begin{align}
\begin{split} \label{eq:ejexpandedsimple}
    \mathbf{E}^j&=\mr{\nf}^j(\mathbf{I}_{l\times l}+(\ml^{\vy,j})^\dagger\dt)\left[\left(\mathbf{C}^j_{22}\right)^{\hspace{0.4em} -1}_{11}+\left(\mathbf{C}^j_{22}\right)^{\hspace{0.4em} -1}_{11}\left(\mathbf{C}^j_{22}\right)_{12} \mathbf{D}^j\left(\mathbf{C}^j_{22}\right)_{12}^{\hspace{0.4em} \tran}\left(\mathbf{C}^j_{22}\right)^{\hspace{0.4em} -1}_{11}\right]\\
    &\hspace{1.5cm}-\mr{\nf}^j(\ml^{\vx,j})^\dagger\mathbf{D}^j\left(\mathbf{C}^j_{22}\right)_{12}^{\hspace{0.4em} \tran}\left(\mathbf{C}^j_{22}\right)^{\hspace{0.4em} -1}_{11}\dt\\
    &=\mathbf{C}^j_j+\mathbf{C}^j_j\mr{\nf}^j(\ml^{\vx,j})^\dagger\mathbf{D}^j\left(\mathbf{C}^j_{22}\right)_{12}^{\hspace{0.4em} \tran}\left(\mathbf{C}^j_{22}\right)^{\hspace{0.4em} -1}_{11}\dt-\mr{\nf}^j(\ml^{\vx,j})^\dagger\mathbf{D}^j\left(\mathbf{C}^j_{22}\right)_{12}^{\hspace{0.4em} \tran}\left(\mathbf{C}^j_{22}\right)^{\hspace{0.4em} -1}_{11}\dt\\
    &=\mathbf{C}^j_j+\big(\mathbf{C}^j_j-\mathbf{I}_{l\times l}\big)\mr{\nf}^j(\ml^{\vx,j})^\dagger\mathbf{D}^j\left(\mathbf{C}^j_{22}\right)_{12}^{\hspace{0.4em} \tran}\left(\mathbf{C}^j_{22}\right)^{\hspace{0.4em} -1}_{11}\dt\\
    &=\mathbf{C}^j_j+O(\dt^2)\equiv \mathbf{C}^j_j,
\end{split}
\end{align}
where we have used \eqref{eq:auxiliarymatexpanded1} and \eqref{eq:auxiliarymatexpanded5} and suppressed higher order terms, while
\begin{align}
\begin{split} \label{eq:fjexpandedsimple}
    \mathbf{F}^j&=-\mr{\nf}^j(\mathbf{I}_{l\times l}+(\ml^{\vy,j})^\dagger\dt)\left(\mathbf{C}^j_{22}\right)^{\hspace{0.4em} -1}_{11}\left(\mathbf{C}^j_{22}\right)_{12} \mathbf{D}^j + \mr{\nf}^j(\ml^{\vx,j})^\dagger\mathbf{D}^j\dt\\
    &= -\mathbf{C}^j_j\mr{\nf}^j(\ml^{\vx,j})^\dagger\mathbf{D}^j\dt+ \mr{\nf}^j(\ml^{\vx,j})^\dagger\mathbf{D}^j\dt\\
    &= -\big(\mathbf{C}^j_j-\mathbf{I}_{l\times l}\big)\mr{\nf}^j(\ml^{\vx,j})^\dagger\mathbf{D}^j\dt\\
    &= \mathbf{G}^{\vy,j}\dt\mr{\nf}^j(\ml^{\vx,j})^\dagger\big(\big((\ms^\vx\circ\ms^\vx)^{j}\big)^{-1}+\mathbf{K}^j\mr{\nf}^j(\mathbf{K}^j)^\dagger\dt+O(\dt^2)\big)\\
    &=\mathbf{G}^{\vy,j}\mr{\nf}^j(\ml^{\vx,j})^\dagger\big((\ms^\vx\circ\ms^\vx)^{j}\big)^{-1}\dt+O(\dt^2)\\
    &\to \zeros_{l \times k}, \text{ as } \dt\to0^+,
\end{split}
\end{align}
where again we have used \eqref{eq:auxiliarymatexpanded1} and \eqref{eq:auxiliarymatexpanded4}, and suppressed higher order terms. We point to the proof of Lemma \ref{lem:discretesmoother} for the general form of $\mathbf{E}^j$ and $\mathbf{F}^j$ up to first-order terms when $\ms^\vy\circ\ms^\vx\not\equiv \zeros_{l\times k}$; see \eqref{eq:auxiliarymat5} and \eqref{eq:auxiliarymat6}, respectively.

As such, it is immediate from the everything thus far that if $\ms^\vy\circ\ms^\vx\equiv \zeros_{l\times k}$, then up to order $O(\dt)$ we have
\begin{align}
\begin{split} \label{eq:meanequality}
    \left(\mathbf{m}^j_{j+1}-\vm{\nf}^j=\right)&\mathbf{C}_{j+1}^j\begin{pmatrix}
        \vy^{j+1}-(\mathbf{I}_{l\times l}+\ml^{\vy,j}\dt)\vm{\nf}^j-\vf^{\vy,j}\dt\\
        \vx^{j+1}-\vx^{j}-(\ml^{\vy,j}\vm{\nf}^j+\vf^{\vx,j})\dt
    \end{pmatrix}\\
    &=\begin{pmatrix}
        \mathbf{E}^j & \hspace{0.4em}\mathbf{F}^j
    \end{pmatrix}\begin{pmatrix}
        \vy^{j+1}-(\mathbf{I}_{l\times l}+\ml^{\vy,j}\dt)\vm{\nf}^j-\vf^{\vy,j}\dt\\
        \vx^{j+1}-\vx^{j}-(\ml^{\vy,j}\vm{\nf}^j+\vf^{\vx,j})\dt
    \end{pmatrix}\\
    &=\mathbf{C}_j^j(\vy^{j+1}-(\mathbf{I}_{l\times l}+\ml^{\vy,j}\dt)\vm{\nf}^j-\vf^{\vy,j}\dt)+O(\dt^{3/2})+O(\dt^2)\\
    &\equiv\mathbf{C}_j^j(\vy^{j+1}-(\mathbf{I}_{l\times l}+\ml^{\vy,j}\dt)\vm{\nf}^j-\vf^{\vy,j}\dt)\left(=\mathbf{m}^j_{j}-\vm{\nf}^j\right),
\end{split}
\end{align}
where we used \eqref{eq:auxiliaryblocks}--\eqref{eq:fjexpandedsimple} and \eqref{eq:discretecondgauss1}, which makes the equality in the mean true up to order $O(\dt)$ in discrete time, as well as in the limit $\dt\to0^+$ for continuous time, and
\begin{align}
    \left(\mr{\nf}^j-\mathbf{P}^j_{j+1}=\mathbf{C}_{j+1}^j(\mathbf{C}^j_{12})^\dagger=\right)&\mathbf{C}_{j+1}^j\mathbf{C}^j_{22}(\mathbf{C}_{j+1}^j)^\dagger \nonumber \\
    &=\begin{pmatrix}
        \mathbf{C}_j^j & \hspace{0.4em} O(\dt)
    \end{pmatrix}\mathbf{C}^j_{22}\begin{pmatrix}
        (\mathbf{C}_j^j)^\dagger \\ O(\dt)
    \end{pmatrix} \nonumber\\
    &=\mathbf{C}_j^j\left(\mr{\nf}^j+\left(\ml^{\vy,j}\mr{\nf}^j+\mr{\nf}^j(\ml^{\vy,j})^{\dagger}+(\ms^\vy\circ\ms^\vy)^j\right)\dt\right)(\mathbf{C}_j^j)^\dagger+O(\dt^2) \label{eq:covarianceequality}\\
    &\equiv \mathbf{C}_j^j\left(\mr{\nf}^j+\left(\ml^{\vy,j}\mr{\nf}^j+\mr{\nf}^j(\ml^{\vy,j})^{\dagger}+(\ms^\vy\circ\ms^\vy)^j\right)\dt\right)(\mathbf{C}_j^j)^\dagger\left(=\mr{\nf}^j-\mathbf{P}^j_j\right), \nonumber
\end{align}
where we again used \eqref{eq:auxiliaryblocks}--\eqref{eq:fjexpandedsimple} as well as \eqref{eq:blockcovmatrix}, which makes the equality in the covariance true up to order $O(\dt)$ in discrete time, as well as in the limit $\dt\to0^+$ for continuous time.

Combining \eqref{eq:meanequality} and \eqref{eq:covarianceequality}, are then enough to yield the desired result in continuous time,
\begin{equation*}
    \lim_{\substack{\dt\to0^+\\ J\dt=T}}\pp\big(\vy^j\big|\vy^{j+1},\vx^s, s\leq j\big)=\lim_{\substack{\dt\to0^+\\ J\dt=T}}\pp\big(\vy^j\big|\vy^{j+1},\vx^s, s\leq j+1\big)\left(=\lim_{\substack{\dt\to0^+\\ J\dt=T}}\pp\big(\vy^j\big|\vy^{j+1},\vx^s, s\leq J\big)\right),
\end{equation*}
due to the conditional Gaussianity of these distributions. This ends the proof of Lemma \ref{lem:backwardscond}.
\end{proof}

Here we make a note about the result concerning the Markov property of $\vy^j\big|\vy^{j+1}$ with respect to natural filtration of $\vx$, i.e., the fact that in general we have
\begin{equation*}
    p(\vy^j\big|\vy^{j+1},\vx^s, s\leq j+1)=p(\vy^j\big|\vy^{j+1},\vx^s, s\leq J)\neq p(\vy^j\big|\vy^{j+1},\vx^s, s\leq j), \ \forall j=0,1,\ldots,J-1,
\end{equation*}
but in the absence of noise cross-interaction we have
\begin{equation*}
    \lim_{\substack{\dt\to0^+\\ J\dt=T}}p(\vy^j\big|\vy^{j+1},\vx^s, s\leq j+1)=\lim_{\substack{\dt\to0^+\\ J\dt=T}}p(\vy^j\big|\vy^{j+1},\vx^s, s\leq j),
\end{equation*}
with equality up to $O(\dt)$ in the discrete-time case. The condition $\ms^\vy\circ\ms^\vx\equiv \zeros_{l\times k}$ is only sufficient but not necessary for going from the second-order backwards in time Markov property w.r.t.\ the natural filtration of $\vx$ (in continuous-time) to the first-order one (i.e., the stronger property). Indeed, the algebra in the proof of Lemma \ref{lem:backwardscond}, as well as the results from Lemma \ref{lem:discretesmoother} in the sequel, reveal that we may have the same result if instead we assume that $\pp$-a.s.\ over $[0,T]$ we have:
\begin{gather*}
    \rmd \vx-(\ml^\vx\vm{\nf}+\vf^\vx)\rmd t\in \mathrm{Ker}(\mathbf{F}) \text{ and } \mathbf{G}^{\vx}\equiv \zeros_{k\times l} \text{ for } \lim_{\substack{\dt\to0^+\\ J\dt=T}}\mathbf{m}^j_j=\lim_{\substack{\dt\to0^+\\ J\dt=T}}\mathbf{m}^j_{j+1},\\
    \text{and } \mathbf{F}(\ml^\vx\mr{\nf}+(\ms^\vx\circ\ms^\vy))\mathbf{E}^\dagger+\mathbf{E}(\mr{\nf}(\ml^\vx)^\dagger+(\ms^\vy\circ\ms^\vx))\mathbf{F}^\dagger+\mathbf{F}(\ms^\vx\circ\ms^\vx)^{-1}\mathbf{F}^\dagger \equiv \zeros_{l\times l} \text{ for } \lim_{\substack{\dt\to0^+\\ J\dt=T}}\mathbf{P}^j_j=\lim_{\substack{\dt\to0^+\\ J\dt=T}}\mathbf{P}^j_{j+1},
\end{gather*}
or any other equivalent (to these) relations.

The third lemma concerns itself with the discrete equations of the optimal nonlinear smoother, derived using the discrete smoother distribution.
\begin{lem}[\textbf{Discrete Optimal Nonlinear Smoother}] \label{lem:discretesmoother}
    Let the assumptions in Theorem \ref{thm:smoothing} hold. Denote by $\vm{\normalfont{s}}^{j,J}$ and $\mr{\normalfont{s}}^{j,J}$ the conditional mean and conditional covariance of the conditional distribution of the discrete smoother $\pp\big(\vy^j\big|\vx^s,s\leq J\big)$. In other words,
    \begin{equation*}
        \vm{\normalfont{s}}^{j,J}=\ee{\vy^j\big|\vx^s,s\leq J},\quad\mr{\normalfont{s}}^{j,J}=\mathrm{Cov}\big(\vy^j,\vy^j\big|\vx^s,s\leq J\big)=\ee{(\vy^j-\vm{\normalfont{s}}^{j,J})(\vy^j-\vm{\normalfont{s}}^{j,J})^\dagger\big|\vx^s,s\leq J}.
    \end{equation*}
    Given a realization of the observable $\vx^j$ for $j=0,\ldots,J$, the optimal smoother estimate $\pp\big(\vy^j\big|\vx^s, s=0,\ldots,J\big)$ is conditionally Gaussian,
    \begin{equation*}
        \pp\big(\vy^j\big|\vx^s, s=0,\ldots,J\big)\overset{\rmd}{\sim}\mathcal{N}_l(\vm{\normalfont{s}}^{j,J},\mr{\normalfont{s}}^{j,J}),
    \end{equation*}
    where the conditional mean $\vm{\normalfont{s}}^{j,J}$ and conditional covariance $\mr{\normalfont{s}}^{j,J}$ of the smoother at time step $t_j$ satisfy the following (backward) recursive equations:
    \begin{align}
        \vm{\normalfont{s}}^{j,J}&=\vm{\nf}^j+\mathbf{E}^j\left(\vm{\normalfont{s}}^{j+1,J}-(\mathbf{I}_{l\times l}+\ml^{\vy,j}\dt)\vm{\nf}^j-\vf^{\vy,j}\dt\right)+\mathbf{F}^j\left(\vx^{j+1}-\vx^{j}-(\ml^{\vx,j}\vm{\nf}^j+\vf^{\vx,j})\dt\right), \label{eq:discretesmoother1} \\
        \begin{split}
            \mr{\normalfont{s}}^{j,J}&=\mr{\nf}^j-\mathbf{C}_{j+1}^j\mathbf{C}^j_{22}(\mathbf{C}_{j+1}^j)^\dagger+\mathbf{E}^j\mr{\normalfont{s}}^{j+1,J}(\mathbf{E}^j)^\dagger\\
        &=\mr{\nf}^j+\mathbf{E}^j\left(\mr{\normalfont{s}}^{j+1,J}(\mathbf{E}^j)^{\dagger}-(\mathbf{I}_{l\times l}+\ml^{\vy,j}\dt)\mr{\nf}^j\right)-\mathbf{F}^j\ml^{\vx,j}\mr{\nf}^j\dt,
        \end{split} \label{eq:discretesmoother2}
    \end{align}
    where the auxiliary matrices $\mathbf{E}^j\in\mathbb{C}^{l\times l}$ and $\mathbf{F}^j\in\mathbb{C}^{l\times k}$ are the $(1,1)$ and $(1,2)$ blocks of $\mathbf{C}_{j+1}^j$, respectively, as shown in \eqref{eq:auxiliaryblocks}, and are explicitly given up to leading-order $O(\dt)$ by
    \begin{align}
    \begin{split} \label{eq:auxiliarymat5}
        \mathbf{E}^j&:=\mr{\nf}^j(\mathbf{I}_{l\times l}+(\ml^{\vy,j})^\dagger\dt)\left[\left(\mathbf{C}^j_{22}\right)^{\hspace{0.4em} -1}_{11}+\left(\mathbf{C}^j_{22}\right)^{\hspace{0.4em} -1}_{11}\left(\mathbf{C}^j_{22}\right)_{12} \mathbf{D}^j\left(\mathbf{C}^j_{22}\right)_{12}^{\hspace{0.4em} \tran}\left(\mathbf{C}^j_{22}\right)^{\hspace{0.4em} -1}_{11}\right]\\
        & \hspace{2cm}-\mr{\nf}^j(\ml^{\vx,j})^\dagger\mathbf{D}^j\left(\mathbf{C}^j_{22}\right)_{12}^{\hspace{0.4em} \tran}\left(\mathbf{C}^j_{22}\right)^{\hspace{0.4em} -1}_{11}\dt\\
        &=\mathbf{C}_j^j+(\ms^\vy\circ\ms^\vx)^j\big((\ms^\vx\circ\ms^\vx)^{j}\big)^{-1}\mathbf{G}^{\vx,j}\dt+O(\dt^2),
    \end{split}
    \end{align}
    and
    \begin{align}
    \begin{split} \label{eq:auxiliarymat6}
        \mathbf{F}^j&:=-\mr{\nf}^j(\mathbf{I}_{l\times l}+(\ml^{\vy,j})^\dagger\dt)\left(\mathbf{C}^j_{22}\right)^{\hspace{0.4em} -1}_{11}\left(\mathbf{C}^j_{22}\right)_{12} \mathbf{D}^j + \mr{\nf}^j(\ml^{\vx,j})^\dagger\mathbf{D}^j\dt\\
        &=-\mr{\nf}^j\Big[(\mathbf{K}^j)^{\dagger}+\left((\mathbf{G}^{\vx,j})^\dagger\mathbf{K}^j\mr{\nf}^j(\mathbf{K}^j)^{\dagger}-\big(\mr{\nf}^j\big)^{-1}(\mathbf{H}^{j})^\dagger\mr{\nf}^j(\mathbf{K}^j)^\dagger+(\ml^{\vy,j})^{\dagger}(\mathbf{K}^{j})^\dagger\right)\dt\\
        &\hspace{2cm} -(\ml^{\vx,j})^\dagger\left(\big((\ms^\vx\circ\ms^\vx)^{j}\big)^{-1}+\mathbf{K}^j\mr{\nf}^j(\mathbf{K}^j)^\dagger\dt\right)\Big]+O(\dt^2),
    \end{split}
    \end{align}
    with $\mathbf{C}_j^j\in\mathbb{C}^{l\times l}$ defined in \eqref{eq:auxiliarymat1}, $\left(\mathbf{C}_{22}^j\right)_{11}\in\mathbb{C}^{l\times l}$ and $\left(\mathbf{C}_{22}^j\right)^{\hspace{0.4em} \tran}_{12}\in\mathbb{C}^{l\times k}$ given in \eqref{eq:blockcovmatrix}, and $\mathbf{D}^j\in\mathbb{C}^{k\times k}$ provided in \eqref{eq:schurcomplement}, and are expressed up to first-order terms as
    \begin{gather*}
         \mathbf{C}^j_j:=\mr{\nf}^j(\mathbf{I}_{l\times l}+(\ml^{\vy,j})^\dagger\dt)\left(\mathbf{C}^j_{22}\right)^{\hspace{0.4em} -1}_{11}=\mathbf{I}_{l\times l}-\mathbf{G}^{\vy,j}\dt+O(\dt^2),\\
         \left(\mathbf{C}^j_{22}\right)_{11}=\mr{\nf}^j\left(\mathbf{I}_{l\times l}+\mathbf{H}^j\dt\right),\quad \left(\mathbf{C}^j_{22}\right)^{\hspace{0.4em}\tran}_{12}=\mathbf{G}^{\vx,j}\mr{\nf}^j\dt,\\
         \mathbf{D}^j=\left(\frac{1}{\dt}\mathbf{I}_{k\times k}+\mathbf{L}^j-(\mathbf{L}^j)^2\dt+O(\dt^2)\right)\big((\ms^\vx\circ\ms^\vx)^{j}\big)^{-1} \\
         \mspace{110mu} = \frac{1}{\dt}\big((\ms^\vx\circ\ms^\vx)^{j}\big)^{-1}+\mathbf{K}^j\mr{\nf}^j(\mathbf{K}^j)^\dagger+\mathbf{L}^j\mathbf{K}^j\mr{\nf}^j(\mathbf{K}^j)^\dagger\dt+O(\dt^2),
    \end{gather*}
    where the auxiliary matrices $\mathbf{G}^{\vx,j}\in\mathbb{C}^{k\times l}$, $\mathbf{G}^{\vy,j}\in\mathbb{C}^{l\times l}$, $\mathbf{H}^j\in\mathbb{C}^{l\times l}$, $\mathbf{L}^j\in\mathbb{C}^{k\times k}$, $\mathbf{K}^j\in\mathbb{C}^{k\times l}$ are defined as
    \begin{gather*}
        \mathbf{G}^{\vx,j}:=\ml^{\vx,j}+(\ms^\vx\circ\ms^\vy)^j\big(\mr{\nf}^j\big)^{-1},\quad \mathbf{G}^{\vy,j}:=\ml^{\vy,j}+(\ms^\vy\circ\ms^\vy)^j\big(\mr{\nf}^j\big)^{-1},\\
        \mathbf{H}^j:=(\mr{\nf})^{-1}(\ml^{\vy,j}\mr{\nf}^j+\mr{\nf}^j(\ml^{\vy,j})^{\dagger}+(\ms^\vy\circ\ms^\vy)^j),\\
        \mathbf{K}^j:=\big((\ms^\vx\circ\ms^\vx)^{j}\big)^{-1}\mathbf{G}^{\vx,j},\quad \mathbf{L}^j:=\big((\ms^\vx\circ\ms^\vx)^{j}\big)^{-1}\mathbf{G}^{\vx,j}\mr{\nf}^j(\mathbf{G}^{\vx,j})^\dagger=\mathbf{K}^j\mr{\nf}^j(\mathbf{G}^{\vx,j})^\dagger.
    \end{gather*}
    Finally, by definition we have $\vm{\ns}^{J,J}=\vm{\nf}^J$ and $\mr{\ns}^{J,J}=\mr{\nf}^J$, i.e., the smoother Gaussian statistics coincide with the filter ones on the end point $t_J=T$.
\end{lem}
\begin{proof}[\textbf{Proof of Lemma \ref{lem:discretesmoother}}]
Note that although all discrete estimates we have seen thus far do depend on $J$ (and by extend on $\dt$), since they are evaluated on $t_j=j\frac{T}{J}$ (e.g., $\vm{\nf}^j$, $\mr{\nf}^j$, $\mathbf{m}_j^j$, $\mathbf{P}_j^j$, etc.), for the Gaussian statistics of the discrete smoother this dependence is noted explicitly since we condition over the whole observation period, and as such on all $J+1$ observations of $\vx$; $\vx^0,\ldots,\vx^J$.

Note that we have already showed the desired expressions for $\mathbf{C}_j^j$, $\left(\mathbf{C}_{22}^j\right)_{11}$, $\left(\mathbf{C}_{22}^j\right)^{\hspace{0.4em} \tran}_{12}$, and $\mathbf{D}^j$ in the proof of Lemma \ref{lem:backwardscond}, in \eqref{eq:auxiliarymatexpanded1}, \eqref{eq:auxiliarymatexpanded2}, \eqref{eq:auxiliarymatexpanded3}, and \eqref{eq:auxiliarymatexpanded4}, respectively. We have also recovered the expressions of the block parts of $\mathbf{C}^j_{j+1}$, $\mathbf{E}^j$ and $\mathbf{F}^j$, up to $O(\dt)$ in the absence of noise cross-interaction (i.e., when $\ms^\vy\circ\ms^\vx\equiv \zeros_{l\times k}$), in \eqref{eq:ejexpandedsimple} and \eqref{eq:fjexpandedsimple}. As such, we just need to prove the discrete smoother formulae \eqref{eq:discretesmoother1}--\eqref{eq:discretesmoother2}, and the form of the block matrices \eqref{eq:auxiliarymat5}--\eqref{eq:auxiliarymat6} for the general case (as well as to cross-check to see that when we assume absence of noise cross-interaction we recover the expressions we already have).

Having already established in the proof of Lemma \ref{lem:backwardscond} the auxiliary expressions required here, we can now fully determine the blocks of $\mathbf{C}_{j+1}^j$, $\mathbf{E}^j$ and $\mathbf{F}^j$, up to leading-order terms, for the general case where noise cross-interaction could be present. Specifically, using \eqref{eq:auxiliarymatexpanded1}--\eqref{eq:auxiliarymatexpanded5}, we have
\begin{align}
\begin{split} \label{eq:ejexpanded}
        \mathbf{E}^j&:=\mr{\nf}^j(\mathbf{I}_{l\times l}+(\ml^{\vy,j})^\dagger\dt)\left[\left(\mathbf{C}^j_{22}\right)^{\hspace{0.4em} -1}_{11}+\left(\mathbf{C}^j_{22}\right)^{\hspace{0.4em} -1}_{11}\left(\mathbf{C}^j_{22}\right)_{12} \mathbf{D}^j\left(\mathbf{C}^j_{22}\right)_{12}^{\hspace{0.4em} \tran}\left(\mathbf{C}^j_{22}\right)^{\hspace{0.4em} -1}_{11}\right]\\
        &\hspace{2cm}-\mr{\nf}^j(\ml^{\vx,j})^\dagger\mathbf{D}^j\left(\mathbf{C}^j_{22}\right)_{12}^{\hspace{0.4em} \tran}\left(\mathbf{C}^j_{22}\right)^{\hspace{0.4em} -1}_{11}\dt\\
        &=\mathbf{C}_j^j+(\mathbf{C}_j^j\mr{\nf}(\mathbf{G}^{\vx,j})^\dagger-\mr{\nf}^j(\ml^{\vx,j})^\dagger)\mathbf{K}^j\dt+O(\dt^2)\\
        &=\mathbf{C}_j^j+((\mathbf{I}_{l\times l}-\mathbf{G}^{\vy,j}\dt)\mr{\nf}(\mathbf{G}^{\vx,j})^\dagger-\mr{\nf}^j(\ml^{\vx,j})^\dagger)\mathbf{K}^j\dt+O(\dt^2)\\
        &=\mathbf{C}_j^j+\mr{\nf}(\mathbf{G}^{\vx,j}-\ml^{\vx,j})^\dagger\mathbf{K}^j\dt+O(\dt^2)\\
        &=\mathbf{C}_j^j+\mr{\nf}(\mr{\nf})^{-1}(\ms^\vy\circ\ms^\vx)^j\mathbf{K}^j\dt+O(\dt^2)\\
        &=\mathbf{C}_j^j+(\ms^\vy\circ\ms^\vx)^j\big((\ms^\vx\circ\ms^\vx)^{j}\big)^{-1}\mathbf{G}^{\vx,j}\dt+O(\dt^2),
\end{split}
\end{align}
and
\begin{align}
    \mathbf{F}^j&:=-\mr{\nf}^j(\mathbf{I}_{l\times l}+(\ml^{\vy,j})^\dagger\dt)\left(\mathbf{C}^j_{22}\right)^{\hspace{0.4em} -1}_{11}\left(\mathbf{C}^j_{22}\right)_{12} \mathbf{D}^j + \mr{\nf}^j(\ml^{\vx,j})^\dagger\mathbf{D}^j\dt\nonumber \\
    &=-\mr{\nf}^j(\mathbf{I}_{l\times l}+(\ml^{\vy,j})^\dagger\dt)\big((\mathbf{K}^j)^\dagger+\big((\mathbf{G}^{\vx,j})^\dagger\mathbf{K}^j\mr{\nf}^j(\mathbf{K}^j)^\dagger-\big(\mr{\nf}^j\big)^{-1}(\mathbf{H}^j)^\dagger\mr{\nf}^j(\mathbf{K}^j)^\dagger\big)\dt+O(\dt^2)\big) \nonumber\\
    &\quad +\mr{\nf}^j(\ml^{\vx,j})^\dagger\left(\frac{1}{\dt}\big((\ms^\vx\circ\ms^\vx)^{j}\big)^{-1}+\mathbf{K}^j\mr{\nf}^j(\mathbf{K}^j)^\dagger+O(\dt)\right)\dt+O(\dt^2) \nonumber\\
    &=-\mr{\nf}^j\Big((\mathbf{K}^j)^\dagger-(\ml^{\vx,j})^\dagger\big((\ms^\vx\circ\ms^\vx)^{j}\big)^{-1} \label{eq:fjexpanded} \\
    &\quad +\big((\mathbf{G}^{\vx,j})^\dagger\mathbf{K}^j\mr{\nf}^j(\mathbf{K}^j)^\dagger-\big(\mr{\nf}^j\big)^{-1}(\mathbf{H}^j)^\dagger\mr{\nf}^j(\mathbf{K}^j)^\dagger+(\ml^{\vy,j})^\dagger(\mathbf{K}^j)^\dagger-(\ml^{\vx,j})^\dagger\mathbf{K}^j\mr{\nf}^j(\mathbf{K}^j)^\dagger\big)\dt\Big)+O(\dt^2) \nonumber\\
    &=-\mr{\nf}^j\Big[(\mathbf{K}^j)^{\dagger}+\big((\mathbf{G}^{\vx,j})^\dagger\mathbf{K}^j\mr{\nf}^j(\mathbf{K}^j)^{\dagger}-\big(\mr{\nf}^j\big)^{-1}(\mathbf{H}^{j})^\dagger\mr{\nf}^j(\mathbf{K}^j)^\dagger+(\ml^{\vy,j})^{\dagger}(\mathbf{K}^{j})^\dagger\big)\dt \nonumber \\
    &\hspace{2cm} -(\ml^{\vx,j})^\dagger\big(\big((\ms^\vx\circ\ms^\vx)^{j}\big)^{-1}+\mathbf{K}^j\mr{\nf}^j(\mathbf{K}^j)^\dagger\dt\big)\Big]+O(\dt^2). \nonumber
\end{align}
Notice how in the absence of cross-interaction noise (i.e., $\ms^\vy\circ\ms^\vx\equiv\zeros_{l\times k})$, $\mathbf{E}^j$ reduces down to $\mathbf{C}^j_j+O(\dt^2)$, as already shown in \eqref{eq:ejexpandedsimple}. As for $\mathbf{F}^j$, observe that when $\ms^\vy\circ\ms^\vx\equiv\zeros_{l\times k}$ then we have $\mathbf{G}^{\vx,j}=\ml^{\vx,j}$ and $\mathbf{K}^j=\big((\ms^\vx\circ\ms^\vx)^{j}\big)^{-1}\ml^{\vx,j}$. With these reductions we get
\begin{align*}
    \mathbf{F}^j&=-\mr{\nf}^j\Big[(\mathbf{K}^j)^{\dagger}+\big((\mathbf{G}^{\vx,j})^\dagger\mathbf{K}^j\mr{\nf}^j(\mathbf{K}^j)^{\dagger}-\big(\mr{\nf}^j\big)^{-1}(\mathbf{H}^{j})^\dagger\mr{\nf}^j(\mathbf{K}^j)^\dagger+(\ml^{\vy,j})^{\dagger}(\mathbf{K}^{j})^\dagger\big)\dt\\
    &\hspace{2cm} -(\ml^{\vx,j})^\dagger\big(\big((\ms^\vx\circ\ms^\vx)^{j}\big)^{-1}+\mathbf{K}^j\mr{\nf}^j(\mathbf{K}^j)^\dagger\dt\big)\Big]+O(\dt^2)\\
    &=-\mr{\nf}^j\Big[(\ml^{\vx,j})^{\dagger}\big((\ms^\vx\circ\ms^\vx)^{j}\big)^{-1}+(\ml^{\vx,j})^{\dagger}\mathbf{K}^j\mr{\nf}^j(\mathbf{K}^j)^{\dagger}\dt-\big(\mr{\nf}^j\big)^{-1}(\mathbf{H}^{j})^\dagger\mr{\nf}^j(\mathbf{K}^j)^\dagger\dt+(\ml^{\vy,j})^{\dagger}(\mathbf{K}^{j})^\dagger\dt\\
    &\hspace{2cm} -(\ml^{\vx,j})^\dagger\big((\ms^\vx\circ\ms^\vx)^{j}\big)^{-1}-(\ml^{\vx,j})^\dagger\mathbf{K}^j\mr{\nf}^j(\mathbf{K}^j)^\dagger\dt\Big]+O(\dt^2)\\
    &=\big((\mathbf{H}^{j})^\dagger\mr{\nf}^j-\mr{\nf}^j(\ml^{\vy,j})^{\dagger}\big)(\ml^{\vx,j})^{\dagger}\big((\ms^\vx\circ\ms^\vx)^{j}\big)^{-1}\dt+O(\dt^2)\\
    &= \big(\ml^{\vy,j}\mr{\nf}^j+\mr{\nf}^j(\ml^{\vy,j})^{\dagger}+(\ms^\vy\circ\ms^\vy)^j-\mr{\nf}^j(\ml^{\vy,j})^{\dagger}\big)(\ml^{\vx,j})^{\dagger}\big((\ms^\vx\circ\ms^\vx)^{j}\big)^{-1}\dt+O(\dt^2)\\
    &=\mathbf{G}^{\vy,j}\mr{\nf}^j(\ml^{\vx,j})^{\dagger}\big((\ms^\vx\circ\ms^\vx)^{j}\big)^{-1}\dt+O(\dt^2),
\end{align*}
which is exactly the expression that we had in \eqref{eq:fjexpandedsimple}.

Having established the expressions for all auxiliary matrices, we can now finally proceed with the inspection of the conditional joint distribution of $p(\vy^j,\vy^{j+1}|\vx^s,s\leq J)$ and the recovery of the discrete smoother formulae for the posterior (smoother) mean and covariance. By definition of the conditional distribution, we have
\begin{align*}
    p(\vy^j,\vy^{j+1}|\vx^s,s\leq J)&=p(\vy^{j}|\vy^{j+1},\vx^s,s\leq J)p(\vy^{j+1}|\vx^s,s\leq J)\\
    &=p(\vy^{j}|\vy^{j+1},\vx^s,s\leq j+1)p(\vy^{j+1}|\vx^s,s\leq J) \ (\because \text{ Markov property}).
\end{align*}
The first conditional on the right is Gaussian, as shown in Lemma \ref{lem:backwardscond}, $\pp\big(\vy^j\big|\vy^{j+1},\vx^s, s\leq j+1\big)\overset{\rmd}{\sim} \mathcal{N}_l(\mathbf{m}^j_{j+1},\mathbf{P}^j_{j+1})$, with the conditional mean and covariance given by \eqref{eq:backcond3}--\eqref{eq:backcond4}. As for the second conditional, after marginalizing the discrete-time result of Theorem \ref{thm:condgaussianity} through Lemma \ref{lem:affinity}, we have that it also follows a normal distribution, which we denote by $\mathcal{N}_l(\vm{\normalfont{s}}^{j+1,J},\mr{\normalfont{s}}^{j+1,J})$. As such, by marginalizing this equation through Lemma \ref{lem:bayesian} (over $\vy^{j+1}$) and by using \eqref{eq:backcond3}--\eqref{eq:backcond4}, we arrive at the following result (where again note that we are conditioning on $\vx^s,s\leq J$):
\begin{align*}
    p(\vy^{j}|\vx^s,s\leq J)&=\int_{\mathbb{C}^l}p(\vy^j,\vy^{j+1}|\vx^s,s\leq J)\rmd\vy^{j+1}\\
    &=\int_{\mathbb{C}^l}p(\vy^{j}|\vy^{j+1},\vx^s,s\leq j+1)p(\vy^{j+1}|\vx^s,s\leq J)\rmd\vy^{j+1}\\
    &= \int_{\mathbb{C}^l}\mathcal{N}_l(\vy^{j};\mathbf{E}^j\vy^{j+1}+\mathbf{b}^j,\mathbf{P}^j_{j+1})\mathcal{N}_l(\vy^{j+1};\vm{\normalfont{s}}^{j+1,J},\mr{\normalfont{s}}^{j+1,J})\rmd\vy^{j+1}, \ \mathbf{b}^j:=\mathbf{m}^j_{j+1}-\mathbf{E}^j\vy^{j+1}\\
    &=\mathcal{N}_l\big(\vy^{j};\mathbf{E}^j\vm{\normalfont{s}}^{j+1,J}+\mathbf{b}^j,\mathbf{E}^j\mr{\normalfont{s}}^{j+1,J}(\mathbf{E}^j)^{\dagger}+\mathbf{P}^j_{j+1}\big).
\end{align*}
The explicit expression of $\mathbf{b}^j$ is
\begin{equation} \label{eq:residualvector}
    \mathbf{b}^j=\vm{\nf}^j-\mathbf{E}^j\big((\mathbf{I}_{l\times l}+\ml^{\vy,j}\dt)\vm{\nf}^j+\vf^{\vy,j}\dt\big)+\mathbf{F}^j\big(\vx^{j+1}-\vx^{j}-(\ml^{\vx,j}\vm{\nf}^j+\vf^{\vx,j})\dt\big),
\end{equation}
where $\mathbf{E}^j$ and $\mathbf{F}^j$ up to order $O(\dt)$ are given in \eqref{eq:ejexpanded}--\eqref{eq:fjexpanded}. As a result, with the help of \eqref{eq:backcond3}--\eqref{eq:backcond4} and the previous expressions, we see that the mean and covariance of the discrete smoother are
\begin{align*}
    \vm{\normalfont{s}}^{j,J}&=\mathbf{E}^j\vm{\normalfont{s}}^{j+1,J}+\mathbf{b}^j=\vm{\nf}^j+\mathbf{C}_{j+1}^j\begin{pmatrix}
        \vm{\normalfont{s}}^{j+1,J}-(\mathbf{I}_{l\times l}+\ml^{\vy,j}\dt)\vm{\nf}^j-\vf^{\vy,j}\dt\\
        \vx^{j+1}-\vx^{j}-(\ml^{\vx,j}\vm{\nf}^j+\vf^{\vx,j})\dt
        \end{pmatrix}\\
    &=\vm{\nf}^j+\mathbf{E}^j\big(\vm{\normalfont{s}}^{j+1,J}-(\mathbf{I}_{l\times l}+\ml^{\vy,j}\dt)\vm{\nf}^j-\vf^{\vy,j}\dt\big)+\mathbf{F}^j\big(\vx^{j+1}-\vx^{j}-(\ml^{\vx,j}\vm{\nf}^j+\vf^{\vx,j})\dt\big),\\
    \mr{\normalfont{s}}^{j,J}&=\mathbf{P}^j_{j+1}+\mathbf{E}^j\mr{\normalfont{s}}^{j+1,J}(\mathbf{E}^j)^{\dagger}=\mr{\nf}^j-\mathbf{C}_{j+1}^j\mathbf{C}^j_{22}(\mathbf{C}_{j+1}^j)^\dagger+\mathbf{E}^j\mr{\normalfont{s}}^{j+1,J}(\mathbf{E}^j)^{\dagger}\\
    &=\mr{\nf}^j-\mathbf{E}^j(\mathbf{I}_{l\times l}+\ml^{\vy,j}\dt)\mr{\nf}^j-\mathbf{F}^j\ml^{\vx,j}\mr{\nf}^j\dt+\mathbf{E}^j\mr{\normalfont{s}}^{j+1,J}(\mathbf{E}^j)^{\dagger}\\
    &=\mr{\nf}^j+\mathbf{E}^j\big(\mr{\normalfont{s}}^{j+1,J}(\mathbf{E}^j)^{\dagger}-(\mathbf{I}_{l\times l}+\ml^{\vy,j}\dt)\mr{\nf}^j\big)-\mathbf{F}^j\ml^{\vx,j}\mr{\nf}^j\dt,
\end{align*}
since it is easy to see that
\begin{equation} \label{eq:residualmatrixauxiliary}
    \mathbf{C}_{j+1}^j\mathbf{C}^j_{22}(\mathbf{C}_{j+1}^j)^\dagger=\mathbf{C}_{j+1}^j(\mathbf{C}^j_{12})^\dagger=\mathbf{E}^j(\mathbf{I}_{l\times l}+\ml^{\vy,j}\dt)\mr{\nf}^j+\mathbf{F}^j\ml^{\vx,j}\mr{\nf}^j\dt,
\end{equation}
where we note that this recovers the previous expression we had in \eqref{eq:covarianceequality}, when we have assumed the absence of noise cross-interaction. This finishes the proof of Lemma \ref{lem:discretesmoother}.
\end{proof}

With these Lemmas, we are ready to prove the optimal nonlinear smoother state estimation backward equations.
\begin{proof}[\textbf{Proof of Theorem \ref{thm:smoothing}}]
    This is Theorem 12.10 in Liptser \& Shiryaev \cite{liptser2001statisticsII}. For the analogous result in the case of discrete time, see Theorem 13.12, which outlines the corresponding optimal recursive nonlinear smoother backward difference equations, with the respective sufficient assumptions given in Subchapters 13.2.1 and 13.3.8. Here we instead provide a martingale-free proof by working over the time discretization of our model and then transitioning to the continuous-time case via the usual limit.

Recall from Lemma \ref{lem:discretesmoother}, that
    \begin{align*}
    \vm{\normalfont{s}}^{j,J}&=\vm{\nf}^j+\mathbf{E}^j\big(\vm{\normalfont{s}}^{j+1,J}-(\mathbf{I}_{l\times l}+\ml^{\vy,j}\dt)\vm{\nf}^j-\vf^{\vy,j}\dt\big)+\mathbf{F}^j\big(\vx^{j+1}-\vx^{j}-(\ml^{\vx,j}\vm{\nf}^j+\vf^{\vx,j})\dt\big),\\
    \mr{\normalfont{s}}^{j,J}&=\mr{\nf}^j+\mathbf{E}^j\big(\mr{\normalfont{s}}^{j+1,J}(\mathbf{E}^j)^{\dagger}-(\mathbf{I}_{l\times l}+\ml^{\vy,j}\dt)\mr{\nf}^j\big)-\mathbf{F}^j\ml^{\vx,j}\mr{\nf}^j\dt,
    \end{align*}
    for every $j=0,1,\ldots,J$, where the auxiliary matrices appearing in these expressions are provided up to leading-order terms in Lemma \ref{lem:discretesmoother} and its proof.

    We start off with the smoother mean. Subtracting $\vm{\normalfont{s}}^{j+1,J}$ on both sides of \eqref{eq:discretesmoother1} yields
    \begin{align*}
        \vm{\normalfont{s}}^{j,J}-\vm{\normalfont{s}}^{j+1,J}&=\vm{\nf}^j-\vm{\normalfont{s}}^{j+1,J}\\
        &\hspace{0.4cm}+\mathbf{E}^j\big(\vm{\normalfont{s}}^{j+1,J}-(\mathbf{I}_{l\times l}+\ml^{\vy,j}\dt)\vm{\nf}^j-\vf^{\vy,j}\dt\big)+\mathbf{F}^j\big(\vx^{j+1}-\vx^{j}-(\ml^{\vx,j}\vm{\nf}^j+\vf^{\vx,j})\dt\big).
    \end{align*}
    We now try to account for all the terms on the right-hand side of this mean-difference equation. Recalling the expressions for $\mathbf{E}^j$ and $\mathbf{F}^j$ proved in Lemma \ref{lem:discretesmoother} and by letting $\vm{\normalfont{s}}^{j,J}-\vm{\normalfont{s}}^{j+1,J}=\boldsymbol{\gamma}_1+\boldsymbol{\gamma}_2+\boldsymbol{\gamma}_3+\boldsymbol{\gamma}_4$, for $\boldsymbol{\gamma}_i$'s to be determined based on the expansions of these auxiliary matrices, then for each term, by suppressing terms of order higher that $O(\dt)$, we have:
    \begin{itemize}
        \item[$\bullet \ \boldsymbol{\gamma}_1$:]\hspace{0.2em} Letting $\boldsymbol{\gamma}_1:=\vm{\nf}^j-\vm{\normalfont{s}}^{j+1,J}+\mathbf{C}_j^j\big(\vm{\normalfont{s}}^{j+1,J}-(\mathbf{I}_{l\times l}+\ml^{\vy,j}\dt)\vm{\nf}^j-\vf^{\vy,j}\dt\big)$, we have
        \begin{align*}
            \boldsymbol{\gamma}_1&=\vm{\nf}^j-\vm{\normalfont{s}}^{j+1,J}+(\mathbf{I}_{l\times l}-\mathbf{G}^{\vy,j}\dt)\big(\vm{\normalfont{s}}^{j+1,J}-(\mathbf{I}_{l\times l}+\ml^{\vy,j}\dt)\vm{\nf}^j-\vf^{\vy,j}\dt\big)\\
            &=\vm{\nf}^j-\vm{\normalfont{s}}^{j+1,J}+\vm{\normalfont{s}}^{j+1,J}-(\mathbf{I}_{l\times l}+\ml^{\vy,j}\dt)\vm{\nf}^j-\vf^{\vy,j}\dt+\mathbf{G}^{\vy,j}(\vm{\nf}^j-\vm{\normalfont{s}}^{j+1,J})\dt+O(\dt^2)\\
            &=(-\ml^{\vy,j}\vm{\normalfont{s}}^{j+1,J}-\vf^{\vy,j}+(\ms^\vy\circ\ms^\vy)^j(\mr{\nf}^j)^{-1}(\vm{\nf}^j-\vm{\normalfont{s}}^{j+1,J}))\dt+O(\dt^2).
        \end{align*}

        \item[$\bullet \ \boldsymbol{\gamma}_2$:]\hspace{0.2em} Expanding terms, we have
        \begin{align*}
             \mathbf{E}^j\big(\vm{\normalfont{s}}^{j+1,J}-(\mathbf{I}_{l\times l}+\ml^{\vy,j}\dt)\vm{\nf}^j-\vf^{\vy,j}\dt\big)&=\mathbf{C}_j^j\big(\vm{\normalfont{s}}^{j+1,J}-(\mathbf{I}_{l\times l}+\ml^{\vy,j}\dt)\vm{\nf}^j-\vf^{\vy,j}\dt\big)\\
             &\hspace*{-5cm}
             +(\ms^\vy\circ\ms^\vx)^j\big((\ms^\vx\circ\ms^\vx)^{j}\big)^{-1}\mathbf{G}^{\vx,j}\big(\vm{\normalfont{s}}^{j+1,J}-(\mathbf{I}_{l\times l}+\ml^{\vy,j}\dt)\vm{\nf}^j-\vf^{\vy,j}\dt\big)\dt+O(\dt^2),
        \end{align*}
        where we used the fact that $\mathbf{E}^j=\mathbf{C}_j^j+(\ms^\vy\circ\ms^\vx)^j\big((\ms^\vx\circ\ms^\vx)^{j}\big)^{-1}\mathbf{G}^{\vx,j}\dt+O(\dt^2)$. The first term is already accounted for in $\boldsymbol{\gamma}_1$, and so we let
        \begin{equation*}
            \boldsymbol{\gamma}_2:=(\ms^\vy\circ\ms^\vx)^j\big((\ms^\vx\circ\ms^\vx)^{j}\big)^{-1}\mathbf{G}^{\vx,j}\big(\vm{\normalfont{s}}^{j+1,J}-(\mathbf{I}_{l\times l}+\ml^{\vy,j}\dt)\vm{\nf}^j-\vf^{\vy,j}\dt\big)\dt.
        \end{equation*}
        We then have
        \begin{align*}
            \boldsymbol{\gamma}_2&=(\ms^\vy\circ\ms^\vx)^j\big((\ms^\vx\circ\ms^\vx)^{j}\big)^{-1}\mathbf{G}^{\vx,j}(\vm{\normalfont{s}}^{j+1,J}-\vm{\nf}^j)\dt+O(\dt^2)\\
            &=(\ms^\vy\circ\ms^\vx)^j\big((\ms^\vx\circ\ms^\vx)^{j}\big)^{-1}(\ml^{\vx,j}+(\ms^\vx\circ\ms^\vy)^j\big(\mr{\nf}^j\big)^{-1})(\vm{\normalfont{s}}^{j+1,J}-\vm{\nf}^j)\dt+O(\dt^2)\\
            &=-(\ms^\vy\circ\ms^\vx)^j\big((\ms^\vx\circ\ms^\vx)^{j}\big)^{-1}\ml^{\vx,j}(\vm{\nf}^j-\vm{\normalfont{s}}^{j+1,J})\dt\\
            &\hspace{1cm} -(\ms^\vy\circ\ms^\vx)^j\big((\ms^\vx\circ\ms^\vx)^{j}\big)^{-1}(\ms^\vx\circ\ms^\vy)^j\big(\mr{\nf}^j\big)^{-1}(\vm{\nf}^j-\vm{\normalfont{s}}^{j+1,J})\dt+O(\dt^2).
        \end{align*}

        \item[$\bullet \ \boldsymbol{\gamma}_3$:]\hspace{0.2em} Letting $\boldsymbol{\gamma}_3:=-\mathbf{F}^j\big(\ml^{\vx,j}\vm{\nf}^j+\vf^{\vx,j}\big)\dt$, we have
        \begin{align*}
            \boldsymbol{\gamma}_3&=\mr{\nf}^j\big((\mathbf{K}^j)^\dagger-(\ml^{\vx,j})^\dagger\big((\ms^\vx\circ\ms^\vx)^{j}\big)^{-1}\big)\big(\ml^{\vx,j}\vm{\nf}^j+\vf^{\vx,j}\big)\dt+O(\dt^2)\\
            &=\mr{\nf}^j\big((\mr{\nf}^j)^{-1}(\ms^\vy\circ\ms^\vx)^j+(\ml^{\vx,j})^\dagger-(\ml^{\vx,j})^\dagger\big)\big((\ms^\vx\circ\ms^\vx)^{j}\big)^{-1}\big(\ml^{\vx,j}\vm{\nf}^j+\vf^{\vx,j}\big)\dt+O(\dt^2)\\
            &=(\ms^\vy\circ\ms^\vx)^j\big((\ms^\vx\circ\ms^\vx)^{j}\big)^{-1}\big(\ml^{\vx,j}\vm{\nf}^j+\vf^{\vx,j}\big)\dt+O(\dt^2).
        \end{align*}

        \item[$\bullet \ \boldsymbol{\gamma}_4$:]\hspace{0.2em} First, we let $\boldsymbol{\gamma}_4:=\mathbf{F}^j\left(\vx^{j+1}-\vx^j\right)$. But, observe from \eqref{eq:discretecondgauss1} that
        \begin{equation*}
            \Delta\vx^j:=\vx^{j+1}-\vx^j=\left(\ml^{\vy,j}\vy^j+\vf^{\vy,j}\right)\dt+\ms_1^{\vy,j}\sqrt{\dt} \ve_1^j+\ms_2^{\vy,j}\sqrt{\dt} \ve_2^j=O(\sqrt{\dt}),
        \end{equation*}
        and by linearity $\ee{\Delta\vx^j\big|\vx^s,s\leq J}=O(\dt)$. As such, in a similar manner as with $\boldsymbol{\gamma}_3$, we see that
        \begin{equation*}
            \boldsymbol{\gamma}_4=-(\ms^\vy\circ\ms^\vx)^j\big((\ms^\vx\circ\ms^\vx)^{j}\big)^{-1}\left(\vx^{j+1}-\vx^j\right)+O(\dt^{3/2}).
        \end{equation*}
    \end{itemize}

    Adding all of these together based on the definitions of the $\boldsymbol{\gamma}_i$'s and expansions of the auxiliary matrices as shown in Lemma \ref{lem:discretesmoother}, and carrying out the necessary algebra, after taking the limit as $\dt\to0^+$, or equivalently as $J\to \infty$, we end up with
    \begin{equation*}
        \smooth{\rmd  \vm{\ns}}(t)=-\big(\ml^\vy\vm{\ns}+\vf^\vy-\mb\mr{\nf}^{-1}(\vm{\nf}-\vm{\ns})\big)\rmd t+ (\ms^\vy\circ \ms^\vx)(\ms^\vx\circ \ms^\vx)^{-1}\big(\smooth{ \rmd\vx}+(\ml^\vx\vm{\ns}+\vf^\vx)\rmd t \big),
    \end{equation*}
    for $T\geq t\geq 0$ and where the $\ma,\mb$ matrices are given as in \eqref{eq:auxiliarymata}--\eqref{eq:auxiliarymatb}. The backward-arrow notation on the left-hand side, and $\smooth{\rmd \vx}$ on right-hand side, denotes
    \begin{equation*}
        \smooth{\rmd \vm{\ns}}(t):=\lim_{\Delta t\to 0} \left(\vm{\ns}(t)-\vm{\ns}(t+\Delta t)\right),\quad \smooth{\rmd \vx}(t):=\lim_{\Delta t\to 0} \left(\vx(t)-\vx(t+\Delta t)\right),
    \end{equation*}
    for $\vm{\ns}(t)$ being the conditional mean of the smoother distribution $\pp\big(\vyt\big|\vx(s), s\in[0,T]\big)$ at time $t$, where on the right-hand side the smoother Gaussian mean is evaluated at $t+\dt$, while all other terms and the filter mean are being evaluated at $t$.

    We now look into the smoother covariance tensor. Subtracting $\mr{\normalfont{s}}^{j,J}$ on both sides of \eqref{eq:discretesmoother2} yields
    \begin{align*}
        \hspace*{-1.2cm}
        \mr{\normalfont{s}}^{j,J}-\mr{\normalfont{s}}^{j+1,J}&=\mr{\nf}^j-\mr{\normalfont{s}}^{j+1,J}-\mathbf{C}_{j+1}^j\mathbf{C}^j_{22}(\mathbf{C}_{j+1}^j)^\dagger+\mathbf{E}^j\mr{\normalfont{s}}^{j+1,J}(\mathbf{E}^j)^{\dagger}\\
        \hspace*{-1.2cm}
        &=\mr{\nf}^j-\mr{\normalfont{s}}^{j+1,J}-\mathbf{E}^j\left(\mathbf{C}_{22}^j\right)_{11}(\mathbf{E}^j)^\dagger-\mathbf{F}^j\left(\mathbf{C}_{22}^j\right)^{\hspace{0.4em} \tran}_{12}(\mathbf{E}^j)^\dagger\\
        &\hspace{2.35cm}-\mathbf{E}^j\left(\mathbf{C}_{22}^j\right)_{12}(\mathbf{F}^j)^\dagger-\mathbf{F}^j\left(\mathbf{C}_{22}^j\right)_{22}(\mathbf{F}^j)^\dagger+\mathbf{E}^j\mr{\normalfont{s}}^{j+1,J}(\mathbf{E}^j)^{\dagger},
    \end{align*}
    where we have used \eqref{eq:blockcovmatrix}, \eqref{eq:auxiliaryblocks}, and block-matrix algebra to expand the term $\mathbf{C}_{j+1}^j\mathbf{C}^j_{22}(\mathbf{C}_{j+1}^j)^\dagger$. We now try to account for all the terms on the right-hand side of the covariance-difference equation by letting
    \begin{equation*}
        \mr{\normalfont{s}}^{j,J}-\mr{\normalfont{s}}^{j+1,J}=\sum_{i=1}^6\boldsymbol{\Theta}_i,
    \end{equation*}
    for $\boldsymbol{\Theta}_i$'s to be determined based on the expansions of auxiliary matrices up to leading-order terms as outlined in Lemmas \ref{lem:backwardscond} and \ref{lem:discretesmoother}. Specifically, for each term, by suppressing terms of order higher than $O(\dt)$, we have:
    \begin{itemize}
        \item[$\bullet \ \boldsymbol{\Theta}_1$:]\hspace{0.4em} Letting $\boldsymbol{\Theta}_1:=\mr{\nf}^j-\mr{\normalfont{s}}^{j+1,J}+\mathbf{C}_j^j\Big(\mr{\normalfont{s}}^{j+1,J}-\left(\mathbf{C}_{22}^j\right)_{11}\big)(\mathbf{C}_j^j)^\dagger$, we have
        \begin{align*}
            \boldsymbol{\Theta}_1&=\mr{\nf}^j-\mr{\normalfont{s}}^{j+1,J}+(\mathbf{I}_{l\times l}-\mathbf{G}^{\vy,j}\dt)\big(\mr{\normalfont{s}}^{j+1,J}-\mr{\nf}^j(\mathbf{I}_{l\times l}+\mathbf{H}^j\dt)\big)(\mathbf{I}_{l\times l}-\mathbf{G}^{\vy,j}\dt)^\dagger+O(\dt^2)\\
            &=-\big(\mathbf{G}^{\vy,j}\mr{\normalfont{s}}^{j+1,J}+\mr{\normalfont{s}}^{j+1,J}(\mathbf{G}^{\vy,j})^\dagger-(\ms^\vy\circ\ms^\vy)^j\big)\dt+O(\dt^2).
        \end{align*}

        \item[$\bullet \Bigg\{\substack{\displaystyle \boldsymbol{\Theta}_2\\ \\ \displaystyle \boldsymbol{\Theta}_3}$:]\hspace{1em} Expanding terms, we have
        \begin{align*}
             \mathbf{E}^j\Big(\mr{\normalfont{s}}^{j+1,J}-\left(\mathbf{C}_{22}^j\right)_{11}\Big)(\mathbf{E}^j)^\dagger&=\mathbf{C}_j^j\Big(\mr{\normalfont{s}}^{j+1,J}-\left(\mathbf{C}_{22}^j\right)_{11}\Big)(\mathbf{C}_j^j)^\dagger\\
             &+\mathbf{C}_j^j\big(\mr{\normalfont{s}}^{j+1,J}-\mr{\nf}^j\big)(\mathbf{G}^{\vx,j})^\dagger\big((\ms^\vx\circ\ms^\vx)^{j}\big)^{-1}(\ms^\vx\circ\ms^\vy)^j\dt\\
             &+(\ms^\vy\circ\ms^\vx)^j\big((\ms^\vx\circ\ms^\vx)^{j}\big)^{-1}\mathbf{G}^{\vx,j}\big(\mr{\normalfont{s}}^{j+1,J}-\mr{\nf}^j\big)(\mathbf{C}_j^j)^\dagger\dt+O(\dt^2),
        \end{align*}
        since $\left(\mathbf{C}_{22}^j\right)_{11}=\mr{\nf}^j+O(\dt)$ and $\mathbf{E}^j=\mathbf{C}_j^j+(\ms^\vy\circ\ms^\vx)^j\big((\ms^\vx\circ\ms^\vx)^{j}\big)^{-1}\mathbf{G}^{\vx,j}\dt+O(\dt^2)$. The first term is already accounted for in $\boldsymbol{\Theta}_1$, and so we let
        \begin{equation*}
            \boldsymbol{\Theta}_2=\mathbf{C}_j^j\big(\mr{\normalfont{s}}^{j+1,J}-\mr{\nf}^j\big)(\mathbf{G}^{\vx,j})^\dagger\big((\ms^\vx\circ\ms^\vx)^{j}\big)^{-1}(\ms^\vx\circ\ms^\vy)^j\dt,
        \end{equation*}
        and $\boldsymbol{\Theta}_3=(\boldsymbol{\Theta}_2)^\dagger$. Using then the fact that $\mathbf{C}_j^j=\mathbf{I}_{l\times l} +O(\dt)$, it is immediate that
        \begin{equation*}
            \boldsymbol{\Theta}_2=\big(\mr{\normalfont{s}}^{j+1,J}-\mr{\nf}^j\big)(\mathbf{G}^{\vx,j})^\dagger\big((\ms^\vx\circ\ms^\vx)^{j}\big)^{-1}(\ms^\vx\circ\ms^\vy)^j\dt+O(\dt^2),
        \end{equation*}
        and likewise for $\boldsymbol{\Theta}_3=(\boldsymbol{\Theta}_2)^\dagger=(\ms^\vy\circ\ms^\vx)^j\big((\ms^\vx\circ\ms^\vx)^{j}\big)^{-1}\mathbf{G}^{\vx,j}\big(\mr{\normalfont{s}}^{j+1,J}-\mr{\nf}^j\big)\dt+O(\dt^2)$.

        \item[$\bullet \Bigg\{\substack{\displaystyle \boldsymbol{\Theta}_4\\ \\ \displaystyle \boldsymbol{\Theta}_5}$:]\hspace{1em} Letting $\boldsymbol{\Theta}_4:=-\mathbf{F}^j\left(\mathbf{C}_{22}^j\right)^{\hspace{0.4em} \tran}_{12}(\mathbf{E}^j)^\dagger$, we have
        \begin{align*}
            \boldsymbol{\Theta}_4&=\mr{\nf}^j((\mathbf{G}^{\vx,j})^\dagger-(\ml^{\vx,j})^\dagger+O(\dt))\big((\ms^\vx\circ\ms^\vx)^{j}\big)^{-1}\mathbf{G}^{\vx,j}\mr{\nf}^j((\mathbf{C}_j^j)^\dagger+O(\dt))\dt\\
            &=\mr{\nf}^j((\mathbf{G}^{\vx,j})^\dagger-(\ml^{\vx,j})^\dagger)\big((\ms^\vx\circ\ms^\vx)^{j}\big)^{-1}\mathbf{G}^{\vx,j}\mr{\nf}^j(\mathbf{C}_j^j)^\dagger\dt+O(\dt^2)\\
            &=(\ms^\vy\circ\ms^\vx)^j\big((\ms^\vx\circ\ms^\vx)^{j}\big)^{-1}\mathbf{G}^{\vx,j}\mr{\nf}^j\dt+O(\dt^2),
        \end{align*}
        and likewise for
        \begin{equation*}
        \boldsymbol{\Theta}_5:=(\boldsymbol{\Theta}_4)^\dagger=-\mathbf{E}^j\left(\mathbf{C}_{22}^j\right)_{12}(\mathbf{F}^j)^\dagger=\mr{\nf}^j(\mathbf{G}^{\vx,j})^\dagger\big((\ms^\vx\circ\ms^\vx)^{j}\big)^{-1}(\ms^\vx\circ\ms^\vy)^j\dt+O(\dt^2).
        \end{equation*}
        
        \item[$\bullet \ \boldsymbol{\Theta}_6$:]\hspace{0.4em} Letting $\boldsymbol{\Theta}_6:=-\mathbf{F}^j\left(\mathbf{C}_{22}^j\right)_{22}(\mathbf{F}^j)^\dagger$, we have
        \begin{align*}
            \boldsymbol{\Theta}_6&=-\mr{\nf}^j((\mathbf{G}^{\vx,j})^\dagger-(\ml^{\vx,j})^\dagger+O(\dt))\big((\ms^\vx\circ\ms^\vx)^{j}\big)^{-1}(\ms^\vx\circ\ms^\vx)^j\\
            &\hspace{2cm}\times\big((\ms^\vx\circ\ms^\vx)^{j}\big)^{-1}(\mathbf{G}^{\vx,j}-\ml^{\vx,j}+O(\dt))\mr{\nf}^j\dt\\
            &=-\mr{\nf}^j((\mathbf{G}^{\vx,j})^\dagger-(\ml^{\vx,j})^\dagger)\big((\ms^\vx\circ\ms^\vx)^{j}\big)^{-1}(\mathbf{G}^{\vx,j}-\ml^{\vx,j})\mr{\nf}^j\dt+O(\dt^2)\\
            &=-(\ms^\vy\circ\ms^\vx)^j\big((\ms^\vx\circ\ms^\vx)^{j}\big)^{-1}(\ms^\vx\circ\ms^\vy)^j\dt+O(\dt^2).
        \end{align*}
    \end{itemize}

    Like we did for the smoother mean, adding all of these together, and based on the definitions of the $\boldsymbol{\Theta}_i$'s and expansions of the auxiliary matrices as shown in Lemmas \ref{lem:backwardscond} and \ref{lem:discretesmoother}, carrying out the necessary algebra and afterwards taking the limit as $\dt\to0^+$, or equivalently as $J\to \infty$, yields
    \begin{equation*}
        \smooth{\rmd \mr{\ns}}(t)= -\big((\ma+\mb\mr{\nf}^{-1})\mr{\ns}+\mr{\ns}(\ma+\mb\mr{\nf}^{-1})^\dagger-\mb\big)\rmd t,
    \end{equation*}
    for $T\geq t\geq 0$, where the $\ma,\mb$ matrices are given as in \eqref{eq:auxiliarymata}--\eqref{eq:auxiliarymatb}, and with the backward-arrow notation on the left-hand side denoting
    \begin{equation*}
        \smooth{\rmd \mr{\ns}}(t):=\lim_{\Delta t\to 0} \left(\mr{\ns}(t)-\mr{\ns}(t+\Delta t)\right),
    \end{equation*}
    for $\mr{\ns}(t)$ being the conditional covariance of the smoother distribution $\pp\big(\vyt\big|\vx(s), s\in[0,T]\big)$ at time $t$, where on the right-hand side the smoother covariance is evaluated at $t+\dt$, while all other terms and the filter estimate are being evaluated at $t$.

    Finally, the existence and uniqueness of continuous solutions to \eqref{eq:revbackinter1}--\eqref{eq:revbackinter2}, that continuously depend on the initial distribution and additional model parameters, can be established by the fact that $\vm{\ns}(t)$ and $\mr{\ns}(t)$ solve backward random ordinary differential equation that are linear in $\vm{\ns}(t)$ and $\mr{\ns}(t)$, respectively, with matrix-valued coefficients and forcings that satisfy the conditions of the random variant of the Picard-Lindel\"of theorem (under our assumed regularity conditions) \cite{strand1970random, han2017random, soong1973random}; adaptations are of course needed from the parallel approach in establishing these properties for the filter posterior Gaussian statistics, since \eqref{eq:revbackinter1}--\eqref{eq:revbackinter2} now run backwards in time (with a terminating instead of an initial condition, specifically the identification of the smoother statistics through the filter ones) \cite{lu2021mathematical, lim2001linear, sun2021linear}. An alternative approach to the existence, uniqueness, and continuity of the smoother Gaussian statistics is provided in the corresponding results from Liptser \& Shiryayev \cite{liptser2001statisticsII}.
\end{proof}

\subsection{Proof of Theorem \ref{thm:forwardsample}} \label{sec:Appendix_D}

Before proceeding with the derivation and proof of the formula for conditional optimal nonlinear forward sampling of the unobserved process in a CGNS, we first establish the following lemma.
\begin{lem}[\textbf{Cross-Covariance of $\boldsymbol{(\vy^{j},\vy^{j+1})}$ Conditioned on $\vx^s,s\leq J$}] \label{lem:cross-covariance}
        The cross-covariance matrix between $\vy^j$ and $\vy^{j+1}$ when conditioned on the whole observed time series of $\vx^s$, $s\leq J$, for $j+1\leq J$, is given by
        \begin{equation*}
            \mathrm{Cov}\big(\vy^j,\vy^{j+1}\big| \vx^s, s\leq J\big)=\mathbf{E}^j\mr{\normalfont{s}}^{j+1,J},
        \end{equation*}
        where $\mathbf{E}^j$ is defined by \eqref{eq:auxiliarymat6}.
\end{lem}
\begin{proof}[\textbf{Proof of Lemma \ref{lem:cross-covariance}}] The standard proof would start from writing down the details of the joint distribution of $\vy^j$ and $\vy^{j+1}$ conditioned on the observations $\vx^s$, $s\leq J$, granted $j+1\leq J$. Here, we instead utilize a shortcut proxy method to arrive at the form of this cross-covariance matrix. A detailed proof which derives the form of this cross-covariance matrix by solving for the joint Gaussian distribution in detail, in the absence of cross-interacting noise feedbacks, can be found in Theorem A.2 of Chen \cite{chen2020learning} (i.e., for $\ms^\vy\circ\ms^\vx\equiv \zeros_{l\times k}$); the general case is an easy adaptation of that approach, by using the full $2\times 2$ block covariance matrix instead of a diagonal one.

We know from the discrete-time variant of Theorem \ref{thm:condgaussianity}, that the following joint distribution is Gaussian:
\begin{equation*}
    \pp\big(\vy^j,\vy^{j+1}\big|\vx^s, s\leq J\big)\overset{\rmd}{\sim}\mathcal{N}_{2l}\left(\begin{pmatrix}
        \vm{\ns}^{j,J}\\ \vm{\ns}^{j+1,J}
    \end{pmatrix},\begin{pmatrix}
        \mr{\ns}^{j,J} & \mathrm{Cov}\big(\vy^j,\vy^{j+1}\big| \vx^s, s\leq J\big)\\
        \mathrm{Cov}\big(\vy^{j+1},\vy^j\big| \vx^s, s\leq J\big) & \mr{\ns}^{j+1,J}
    \end{pmatrix}\right).
\end{equation*}
Then, in light of Lemma \ref{lem:conditional}, the conditional mean of $\vy^j$ given $\vy^{j+1}$ and $\vx^s$, $s\leq J$, is given by
\begin{equation*}
    \ee{\vy^j\big|\vy^{j+1},\vx^s,s\leq J}=\vm{\ns}^{j,J}+\mathrm{Cov}\big(\vy^j,\vy^{j+1}\big| \vx^s, s\leq J\big)\big(\mr{\ns}^{j+1,J}\big)^{-1}(\vy^{j+1}-\vm{\ns}^{j+1,J}).
\end{equation*}
But, in the general case where noise cross-interaction is present, this mean is exactly given by $\mathbf{m}_{j+1}^j$ in \eqref{eq:backcond3}, due to the Markovian property illustrated in \eqref{eq:markovproperty}. As such, by matching the terms appearing in front of $\vy^{j+1}$, after recalling that $\mathbf{C}_{j+1}^j=\begin{pmatrix} \mathbf{E}^j & \hspace{0.4em} \mathbf{F}^j \end{pmatrix}$), it is immediate that $\mathrm{Cov}\left(\vy^j,\vy^{j+1}\big| \vx^s, s\leq J\right)=\mathbf{E}^j\mr{\ns}^{j+1,J}$.
\end{proof}

We now proceed with the algorithm for forward sampling of the hidden process conditioned on the current observations.
\begin{proof}[\textbf{Proof of Theorem \ref{thm:forwardsample}}] By the discrete-time variant of Theorem \ref{thm:condgaussianity}, $\pp\big(\vy^{j+1},\vy^j\big|\vx^s,s\leq j+1\big)$ is normal (after using the exchangeability (in law) property of jointly distributed Gaussian random vectors by virtue of Lemma \ref{lem:affinity} and permutation matrices). Likewise, by $\pp\big(\vy^j,\vx^{j+1}\big|\vx^s,s\leq j\big)=\pp\big(\vx^{j+1}\big|\vy^j,\vx^s,s\leq j\big)\pp\big(\vy^j\big|,\vx^s,s\leq j\big)$ and Lemma \ref{lem:bayesian}, as a result of \eqref{eq:discretecondgauss1}, Lemma \ref{lem:affinity}, and Theorem \ref{thm:filtering}, so is $\pp\big(\vy^j,\vx^{j+1}\big|\vx^s,s\leq j\big)$. Also, note that in discrete time, for the forward sampling procedure, observations are given up to the $(j+1)$-st time instant which we consider to be the end point of the observable time series, and as such only at this point, by definition, the smoother and filter posterior Gaussian statistics coincide, i.e.,
\begin{equation} \label{eq:smootherfiltercoincide}
    \vm{\ns}^{j+1,j+1}=\vm{\nf}^{j+1} \text{ and } \mr{\ns}^{j+1,j+1}=\mr{\nf}^{j+1}.
\end{equation}
This observation is utilized in the sequel.

Now, we turn our attention to the joint distribution $p(\vy^j,\vx^{j+1}|\vx^s,s\leq j)$. As mentioned, this distribution is Gaussian ($\pp$-a.s.), and by \eqref{eq:cg1} and the optimal nonlinear filter estimate, it is explicitly expressed as
\begin{equation*}
    \pp\big(\vy^j,\vx^{j+1}|\vx^s,s\leq j\big)\overset{\rmd}{\sim}\mathcal{N}_{l+k}\left(\begin{pmatrix}
        \vm{\nf}^j\\\vx^j+(\ml^{\vx,j}\vm{\nf}^j+\vf^{\vx,j})\dt
    \end{pmatrix}, \begin{pmatrix}
        \mr{\nf}^j & \mr{\nf}^j(\ml^{\vx,j})^\dagger\dt \\
        \ml^{\vx,j}\mr{\nf}^j\dt & (\ms^\vx\circ\ms^\vx)^j\dt
    \end{pmatrix}\right).
\end{equation*}
As such, by virtue of Lemma \ref{lem:conditional}, we have that $\pp\big(\vy^j\big|\vx^s,s\leq j+1\big)\overset{\rmd}{\sim}\mathcal{N}_l(\mathbf{m}_{\text{f}}^{j,j+1},\mathbf{P}_{\text{f}}^{j,j+1})$, where the first superscript denotes the time step we are currently on for the unobserved process, while the latter denotes up to which time step we condition on, and the subscript ``$\,\text{f}\,$" denotes that this is a forward- or filter-based distribution. Note how we also condition on $\vx^{j+1}$ now, and as such this distribution differs from the filter posterior distribution. If we now define $\mathbf{M}^j:=\mr{\nf}^j(\ml^{\vx,j})^{\tran}$, then
\begin{align}
\begin{split}
    \mathbf{m}_{\text{f}}^{j,j+1}&=\vm{\nf}^j+\mathbf{M}^j\big((\ms^\vx\circ\ms^\vx)^{j}\big)^{-1}\big(\vx^{j+1}-\vx^j-(\ml^{\vx,j}\vm{\nf}^j+\vf^{\vx,j})\dt\big)=\vm{\nf}^j+O\big(\sqrt{\dt}\big),\\
    \mathbf{P}_{\text{f}}^{j,j+1}&=\mr{\nf}^j-\mathbf{M}^j\big((\ms^\vx\circ\ms^\vx)^{j}\big)^{-1}\ml^{\vx,j}\mr{\nf}^j\dt=\mr{\nf}^j+O(\dt).
\end{split} \label{eq:auxiliaryconditionalstatistics}
\end{align}
While at the end point, $t_{j+1}$, we have that the filtering and smoothing Gaussian statistics coincide, at $j$ we see that this is not necessarily the case. But at order $O(1)$, the filtered distribution is also recovered for the $j$-th time step. This observation is utilized in what follows.

In a very similar manner, and by recalling the observation in \eqref{eq:smootherfiltercoincide} and the normality of $\pp\big(\vy^{j+1},\vy^j\big|\vx^s,s\leq j+1\big)$, we have
\begin{equation*}
    \pp\big(\vy^{j+1},\vy^j\big|\vx^s,s\leq j+1\big)\overset{\rmd}{\sim}\mathcal{N}_{2l}\left(\begin{pmatrix}
        \vm{\nf}^{j+1}\\ \mathbf{m}_{\text{f}}^{j,j+1}
    \end{pmatrix}, \begin{pmatrix}
        \mr{\nf}^{j+1} & \mr{\nf}^{j+1}(\mathbf{E}^j)^\dagger \\
        \mathbf{E}^j\mr{\nf}^{j+1} & \mathbf{P}_{\text{f}}^{j,j+1}
    \end{pmatrix}\right),
\end{equation*}
where the form of the cross-covariance matrix comes from Lemma \ref{lem:cross-covariance} for $J=j+1$. Then Lemma \ref{lem:conditional} once again gives
\begin{align*}
    \ee{\vy^{j+1}|\vy^j,\vx^s,s\leq j+1}:=\mathbf{m}_{\text{f}}^{j+1,j+1}&=\vm{\nf}^{j+1}+\mr{\nf}^{j+1}(\mathbf{E}^j)^\dagger(\mathbf{P}_{\text{f}}^{j,j+1})^{-1}(\vy^j-\mathbf{m}_{\text{f}}^{j,j+1}),\\
    \mathrm{Var}\big(\vy^{j+1}\big|\vy^j,\vx^s,s\leq j+1\big):=\mathbf{P}_{\text{f}}^{j+1,j+1}&=\mr{\nf}^{j+1}-\mr{\nf}^{j+1}(\mathbf{E}^j)^\dagger(\mathbf{P}_{\text{f}}^{j,j+1})^{-1}\mathbf{E}^j\mr{\nf}^{j+1}.
\end{align*}
We now express these Gaussian statistics up to leading-order terms. But before that, it is important to observe that our end-goal is to produce a forward sampling formula, meaning our focus is extracting a discrete equation for the difference $\mathbf{m}_{\text{f}}^{j+1,j+1}-\vy^j$ up to leading-order $O(\dt)$, and then passing to the limit as to retrieve an SDE that can be used to forward sample the unobserved process based on the observations up until $\vx^{j+1}$. But exactly because our focus is the difference $\mathbf{m}_{\text{f}}^{j+1,j+1}-\vy^j$, and not $\vy^{j+1}-\vy^j$, it stands to reason that the statistics at point $j$, for arbitrarily small $\dt$, should necessarily coincide with the corresponding Gaussian filtering statistics. As such, we essentially have that
\begin{align*}
    \mathbf{m}_{\text{f}}^{j+1,j+1}&=\vm{\nf}^{j+1}+\mr{\nf}^{j+1}(\mathbf{E}^j)^\dagger(\mr{\nf}^j)^{-1}(\vy^j-\vm{\nf}^{j}),\\
    \mathbf{P}_{\text{f}}^{j+1,j+1}&=\mr{\nf}^{j+1}-\mr{\nf}^{j+1}(\mathbf{E}^j)^\dagger(\mr{\nf}^j)^{-1}\mathbf{E}^j\mr{\nf}^{j+1},
\end{align*}
for $\dt\ll 1$, where we note the appearance of $\vm{\nf}^j$ and $\mr{\nf}^j$ on the right-hand side instead of $\mathbf{m}_{\text{f}}^{j,j+1}$ and $\mathbf{P}_{\text{f}}^{j,j+1}$, respectively; this can also be argued rigorously through \eqref{eq:auxiliaryconditionalstatistics}.

As such, by using the results from Lemma \ref{lem:discretesmoother}, where we showed that
\begin{align*}
    \mathbf{E}^j&=\mathbf{C}_j^j+(\ms^\vy\circ\ms^\vx)^j\big((\ms^\vx\circ\ms^\vx)^{j}\big)^{-1}\mathbf{G}^{\vx,j}\dt+O(\dt^2)\\
    &=\mathbf{I}_{l\times l}+\big((\ms^\vy\circ\ms^\vx)^j\big((\ms^\vx\circ\ms^\vx)^{j}\big)^{-1}\mathbf{G}^{\vx,j}-\mathbf{G}^{\vy,j}\big)\dt+O(\dt^2),
\end{align*}
we have
\begin{align*}
    \mathbf{m}_{\text{f}}^{j+1,j+1}&=\vm{\nf}^{j+1}+\mr{\nf}^{j+1}(\mathbf{E}^j)^\dagger(\mr{\nf}^j)^{-1}(\vy^j-\vm{\nf}^{j})\\
    &=\vm{\nf}^{j+1}+\mr{\nf}^{j+1}(\mr{\nf}^j)^{-1}(\vy^j-\vm{\nf}^{j})\\
    &\hspace{1.7cm}+\mr{\nf}^{j+1}\big((\mathbf{G}^{\vx,j})^\dagger\big((\ms^\vx\circ\ms^\vx)^{j}\big)^{-1}(\ms^\vx\circ\ms^\vy)^j-(\mathbf{G}^{\vy,j})^\dagger\big)(\mr{\nf}^j)^{-1}(\vy^j-\vm{\nf}^{j})\dt+O(\dt^2).
\end{align*}
We now substitute $\mr{\nf}^{j+1}=\mr{\nf}^j+(\mathbf{N}_1^j+\mathbf{N}_2^j)\dt$ back into the previous equation, where by the discretization of \eqref{eq:filter2}, \eqref{eq:discretefilter2}), $\mathbf{N}_1^j$ and $\mathbf{N}_2^j$ are defined by
\begin{align*}
    \mathbf{N}_1^j&:=\ml^{\vy,j}\mr{\nf}^j+\mr{\nf}^j(\ml^{\vy,j})^{\dagger}+(\ms^\vy\circ\ms^\vy)^j=\mr{\nf}^j\mathbf{H}^j,\\
    \mathbf{N}_2^j&:=-\big(\mr{\nf}^j(\ml^{\vx,j})^{\dagger}+(\ms^\vy\circ\ms^\vx)^j\big)\big((\ms^\vx\circ\ms^\vx)^{j}\big)^{-1}\big(\ml^{\vx,j}\mr{\nf}^j+(\ms^\vx\circ\ms^\vy)^j\big)\\
    &=-\mr{\nf}^j(\mathbf{G}^{\vx,j})^\dagger\big((\ms^\vx\circ\ms^\vx)^{j}\big)^{-1}\mathbf{G}^{\vx,j}\mr{\nf}^j,
\end{align*}
with the auxiliary matrices $\mathbf{H}^j$ and $\mathbf{G}^{\vx,j}$ given in Lemma \ref{lem:discretesmoother}. Doing so, we end up with
\begin{align*}
    \mathbf{m}_{\text{f}}^{j+1,j+1}&=\vm{\nf}^{j+1}+(\vy^j-\vm{\nf}^{j})+(\mathbf{N}_1^j+\mathbf{N}_2^j)(\mr{\nf}^j)^{-1}(\vy^j-\vm{\nf}^{j})\dt\\
    &\hspace{1.7cm}+\mr{\nf}^{j}\big((\mathbf{G}^{\vx,j})^\dagger\big((\ms^\vx\circ\ms^\vx)^{j}\big)^{-1}(\ms^\vx\circ\ms^\vy)^j-(\mathbf{G}^{\vy,j})^\dagger\big)(\mr{\nf}^j)^{-1}(\vy^j-\vm{\nf}^{j})\dt+O(\dt^2),
\end{align*}
or equivalently, based on the definitions of $\mathbf{N}_1^j$, $\mathbf{N}_2^j$, $\mathbf{G}^{\vx,j}$, and $\mathbf{G}^{\vy,j}$, we have
\begin{align}
\begin{split}
    &\mathbf{m}_{\text{f}}^{j+1,j+1}-\vy^j=\vm{\nf}^{j+1}-\vm{\nf}^{j}+\ml^{\vy,j}(\vy^j-\vm{\nf}^j)\dt+\mr{\nf}^{j}(\mathbf{G}^{\vy,j})^\dagger(\mr{\nf}^{j})^{-1}(\vy^j-\vm{\nf}^{j})\dt\\
    &\hspace{0.1cm}-\mr{\nf}^{j}(\mathbf{G}^{\vx,j})^\dagger\big((\ms^\vx\circ\ms^\vx)^{j}\big)^{-1}\ml^{\vx,j}(\vy^j-\vm{\nf}^{j})\dt\mr{\nf}^{j}(\mathbf{G}^{\vx,j})^\dagger\big((\ms^\vx\circ\ms^\vx)^{j}\big)^{-1}(\ms^\vx\circ\ms^\vy)^j(\mr{\nf}^{j})^{-1}(\vy^j-\vm{\nf}^{j})\dt\\
    &\hspace{0.2cm}+\mr{\nf}^{j}(\mathbf{G}^{\vx,j})^\dagger\big((\ms^\vx\circ\ms^\vx)^{j}\big)^{-1}(\ms^\vx\circ\ms^\vy)^j(\mr{\nf}^{j})^{-1}(\vy^j-\vm{\nf}^{j})\dt-\mr{\nf}^{j}(\mathbf{G}^{\vy,j})^\dagger(\mr{\nf}^{j})^{-1}(\vy^j-\vm{\nf}^{j})\dt+O(\dt^2)\\
    \\
    &\hspace{2.3cm}\equiv \vm{\nf}^{j+1}-\vm{\nf}^{j} + (\ma^j-\mr{\nf}^{j}\mc^j)(\vy^j-\vm{\nf}^{j})\dt+O(\dt^2).
\end{split} \label{eq:discreteforwardsample}
\end{align}
This exactly accounts for the drift term in \eqref{eq:forwardsample} when passing to the limit. We just need to get a hold of the diffusion coefficient as well.

Recall from elementary SDE theory that for an Itô diffusion process with noise coefficient $\boldsymbol{\sigma}$, the covariance matrix for its infinitesimal increments is exactly given by $\boldsymbol{\sigma}\boldsymbol{\sigma}^\dagger\dt$. As such, accounting for the self-adjointness, the noise coefficient for the forward sampling formula is the square root of the positive-definite matrix $\mathrm{Var}\big(\vy^{j+1}|\vy^j,\vx^s,s\leq j+1\big)=\mathbf{P}_{\text{f}}^{j+1,j+1}$. We now proceed to find the form of this matrix up to $O(\dt)$ and take its unique square root. Going along this path, and by again using $\mr{\nf}^{j+1}=\mr{\nf}^j+(\mathbf{N}_1^j+\mathbf{N}_2^j)\dt$ and the expression of $\mathbf{E}^j$ from before, we have
\begin{align*}
    \mathbf{P}_{\text{f}}^{j+1,j+1}&=\mr{\nf}^{j+1}-\mr{\nf}^{j+1}(\mathbf{E}^j)^\dagger(\mr{\nf}^j)^{-1}\mathbf{E}^j\mr{\nf}^{j+1}\\
    &=\mr{\nf}^{j+1}\Big[\mathbf{I}_{l\times l}-\left(\mathbf{I}_{l\times l}+((\mathbf{G}^{\vx,j})^\dagger\big((\ms^\vx\circ\ms^\vx)^{j}\big)^{-1}(\ms^\vx\circ\ms^\vy)^j-(\mathbf{G}^{\vy,j})^\dagger)\dt\right)\\
    &\hspace{0.5cm}\times (\mr{\nf}^j)^{-1}\big(\mathbf{I}_{l\times l}+((\ms^\vy\circ\ms^\vx)^j\big((\ms^\vx\circ\ms^\vx)^{j}\big)^{-1}\mathbf{G}^{\vx,j}-\mathbf{G}^{\vy,j})\dt\big)(\mr{\nf}^j+(\mathbf{N}_1^j+\mathbf{N}_2^j)\dt)\Big]\\
    &=\mr{\nf}^{j+1}\Big[\mathbf{I}_{l\times l}-\mathbf{I}_{l\times l}-(\mr{\nf}^j)^{-1}((\ms^\vy\circ\ms^\vx)^j\big((\ms^\vx\circ\ms^\vx)^{j}\big)^{-1}\mathbf{G}^{\vx,j}-\mathbf{G}^{\vy,j})\mr{\nf}^j\dt\\
    &\hspace{0.5cm}-(\mr{\nf}^j)^{-1}(\mathbf{N}_1^j+\mathbf{N}_2^j)\dt-((\mathbf{G}^{\vx,j})^\dagger\big((\ms^\vx\circ\ms^\vx)^{j}\big)^{-1}(\ms^\vx\circ\ms^\vy)^j-(\mathbf{G}^{\vy,j})^\dagger)\dt\Big]+O(\dt^2)\\
    &=(\mathbf{G}^{\vy,j}-(\ms^\vy\circ\ms^\vx)^j\big((\ms^\vx\circ\ms^\vx)^{j}\big)^{-1}\mathbf{G}^{\vx,j})\mr{\nf}^j\dt-(\mathbf{N}_1^j+\mathbf{N}_2^j)\dt\\
    &\hspace{0.5cm}+\mr{\nf}^j((\mathbf{G}^{\vy,j})^\dagger-(\mathbf{G}^{\vx,j})^\dagger\big((\ms^\vx\circ\ms^\vx)^{j}\big)^{-1}(\ms^\vx\circ\ms^\vy)^j)\dt+O(\dt^2)\\
    &=\ml^{\vy,j}\mr{\nf}^j\dt+(\ms^\vy\circ\ms^\vy)^j\dt-(\ms^\vy\circ\ms^\vx)^j\big((\ms^\vx\circ\ms^\vx)^{j}\big)^{-1}\ml^{\vx,j}\mr{\nf}^j\dt\\
    &\hspace{0.5cm}-(\ms^\vy\circ\ms^\vx)^j\big((\ms^\vx\circ\ms^\vx)^{j}\big)^{-1}(\ms^\vx\circ\ms^\vy)^j\dt+\mr{\nf}^j(\ml^{\vy,j})^\dagger\dt+(\ms^\vy\circ\ms^\vy)^j\dt\\
    &\hspace{0.75cm}-\mr{\nf}^j(\ml^{\vx,j})^\dagger\big((\ms^\vx\circ\ms^\vx)^{j}\big)^{-1}(\ms^\vx\circ\ms^\vy)^j\dt-(\ms^\vy\circ\ms^\vx)^j\big((\ms^\vx\circ\ms^\vx)^{j}\big)^{-1}(\ms^\vx\circ\ms^\vy)^j\dt\\
    &\hspace{1cm}-\ml^{\vy,j}\mr{\nf}^j\dt-\mr{\nf}^j(\ml^{\vy,j})^\dagger\dt-(\ms^\vy\circ\ms^\vy)^j\dt+\mr{\nf}^j(\mathbf{G}^{\vx,j})^\dagger\big((\ms^\vx\circ\ms^\vx)^{j}\big)^{-1}\mathbf{G}^{\vx,j}\mr{\nf}^j+O(\dt^2)\\
    &=(\ms^\vy\circ\ms^\vy)^j\dt-(\ms^\vy\circ\ms^\vx)^j\big((\ms^\vx\circ\ms^\vx)^{j}\big)^{-1}\ml^{\vx,j}\mr{\nf}^j\dt\\
    &\hspace{0.5cm}-2(\ms^\vy\circ\ms^\vx)^j\big((\ms^\vx\circ\ms^\vx)^{j}\big)^{-1}(\ms^\vx\circ\ms^\vy)^j\dt-\mr{\nf}^j(\ml^{\vx,j})^\dagger\big((\ms^\vx\circ\ms^\vx)^{j}\big)^{-1}(\ms^\vx\circ\ms^\vy)^j\dt\\
    &\hspace{0.75cm}+\mr{\nf}^j\mc^j\mr{\nf}^j+(\ms^\vy\circ\ms^\vx)^j\big((\ms^\vx\circ\ms^\vx)^{j}\big)^{-1}\ml^{\vx,j}\mr{\nf}^j\dt+\mr{\nf}^j(\ml^{\vx,j})^\dagger\big((\ms^\vx\circ\ms^\vx)^{j}\big)^{-1}(\ms^\vx\circ\ms^\vy)^j\dt\\
    &\hspace{1cm}+(\ms^\vy\circ\ms^\vx)^j\big((\ms^\vx\circ\ms^\vx)^{j}\big)^{-1}(\ms^\vx\circ\ms^\vy)^j\dt+O(\dt^2)\\
    &\equiv (\mb^j+\mr{\nf}^j\mc^j\mr{\nf}^j)\dt+O(\dt^2).
\end{align*}
Since $\mathbf{P}_{\text{f}}^{j+1,j+1}=(\mb^j+\mr{\nf}^j\mc^j\mr{\nf}^j)\dt+O(\dt^2)$, this establishes that $\mb+\mr{\nf}\mc\mr{\nf}$ is nonnegative-definite $\pp$-a.s., or at least the infimum over time of its smallest eigenvalue is nonnegative $\pp$-a.s., where \eqref{eq:positivedefiniteforwardsample} can also be seen by the fact that $\mb\succeq \mathbf{0}_{l\times}$ $\pp$-a.s., as shown in the proof of Theorem \ref{thm:filtering}, and by the assumption on the positive-definiteness of the filter covariance matrices and $(\ms^\vx\circ\ms^\vx)^{-1}$ over time. As such its unique square root exists. So, by collecting the results thus far, using \eqref{eq:discreteforwardsample}, and taking the limit $\dt\to0^+$, leads to
\begin{equation*}
     \rmd\hat{\vy}_{\text{\nf}} = \rmd\vm{\nf}+(\ma-\mr{\nf}\mc)(\hat{\vy}_{\text{\nf}}-\vm{\nf})\rmd t+(\mb+\mr{\nf}\mc\mr{\nf})^{1/2}\rmd \vw_{\hat{\vy}_{\text{\nf}}},
\end{equation*}
which is exactly what we wanted to prove.
\end{proof}

\subsection{Proof of Theorem \ref{thm:forwardsamplefilter}} \label{sec:Appendix_E}

\begin{proof}[\textbf{Proof of Theorem \ref{thm:forwardsamplefilter}}]
    In this proof, the overline notation denotes taking the expectation of the process in question conditioned on the sub-$\sigma$-algebra $\cF_t^{\vx,\vy(t)}$. We plug-in the mean-fluctuation decomposition \eqref{eq:reynoldsdecomp} into \eqref{eq:forwardsample}, or equivalently by linearity, take the ensemble average of \eqref{eq:forwardsample} conditioned on $\cF_t^{\vx,\vy(t)}$, to get
    \begin{equation}
        \rmd \overline{\hat{\vy}}_{\text{f}}=\rmd\vm{\nf}+(\ma-\mr{\nf}\mc)(\overline{\hat{\vy}}_{\text{f}}-\vm{\nf})\rmd t,\label{eq:ensembleavgeqn}
    \end{equation}
    where $\ee{\vm{\nf}(t)\Big|\cF_t^{\vx,\vy(t)}}=\ee{\ee{\vyt|\cF^\vx_t}\Big|\cF_t^{\vx,\vy(t)}}=\ee{\vyt\big|\cF_t^{\vx}}=\vm{\nf}(t)$ by the stability property of conditional expectations and since $\ee{\vyt\big|\cF_t^{\vx}}$ is a $\cF_t^{\vx}$-measurable process with $\cF_t^{\vx}\subseteq\cF_t^{\vx,\vy(t)}$.

    Now, if $\overline{\hat{\vy}}_{\text{f}}$ at the initial time instant equals the filtering mean, $\vm{\nf}$, then $\overline{\hat{\vy}}_{\text{f}}\equiv\vm{\nf}$ at all time instants, due to the uniqueness of the solution to \eqref{eq:ensembleavgeqn} by the underlying assumptions considered thus far, regardless of whether we are talking about its weak or strong solution. On the other hand, if $\overline{\hat{\vy}}_{\text{f}}$ differs from the filter mean $\vm{\nf}$ at the initial time instant, i.e., in the case of misidentified initial conditions, then $\overline{\hat{\vy}}_{\text{f}}\to\vm{\nf}$, as time evolves, instead, with this convergence being exponentially fast. This is an immediate consequence of the assumption that the real parts of all the eigenvalues of $\mathbf{A}-\mr{\nf}\boldsymbol{\Gamma}$ are negative ($\pp$-a.s.) and by the linearity of \eqref{eq:ensembleavgeqn} \cite{soong1973random, neckel2013random, khasminskii2012stochastic, arnold2014random, crauel2015nonautonomous, han2017random}, where the exponential convergence is due to the mean-square exponential stability of the filter posterior mean which follows again by this condition on the spectrum of $\mathbf{A}-\mr{\nf}\boldsymbol{\Gamma}$. This establishes \textbf{(a)}.

    Now, by subtracting \eqref{eq:ensembleavgeqn} from \eqref{eq:forwardsample}, yields an SDE for the fluctuation part of $\hat{\vy}$,
    \begin{equation} \label{eq:forwresidualeqn}
        \rmd\hat{\vy}_{\text{\nf}}'=(\ma-\mr{\nf}\mc)\hat{\vy}_{\text{\nf}}'\rmd t +(\mb+\mr{\nf}\mc\mr{\nf})^{1/2}\rmd \vw_{\hat{\vy}_{\text{\nf}}}.
    \end{equation}
    The covariance is then given by $\overline{\hat{\vy}_{\text{\nf}}'(\hat{\vy}_{\text{\nf}}')^\dagger}$, which by Itô's lemma \cite{mao2008stochastic, evans2012introduction, oksendal2003stochastic, gardiner2009stochastic, steele2001stohastic} and linearity, has a differential given by
    \begin{equation} \label{eq:forwitolemmacov}
        \rmd \left(\overline{\hat{\vy}_{\text{\nf}}'(\hat{\vy}_{\text{\nf}}')^\dagger}\right)=\overline{\hat{\vy}_{\text{\nf}}'(\rmd\hat{\vy}_{\text{\nf}}')^\dagger}+\overline{\rmd\hat{\vy}_{\text{\nf}}'(\hat{\vy}_{\text{\nf}}')^\dagger}+\overline{\rmd\hat{\vy}_{\text{\nf}}'(\rmd\hat{\vy}_{\text{\nf}}')^\dagger}.
    \end{equation}
    As such, by plugging-in \eqref{eq:forwresidualeqn} into \eqref{eq:forwitolemmacov}, and using the facts
    \begin{gather*}
        \overline{\rmd \vw_{\hat{\vy}_{\text{\nf}}}(\hat{\vy}_{\text{\nf}}')^\dagger}=\overline{\hat{\vy}_{\text{\nf}}'(\rmd \vw_{\hat{\vy}_{\text{\nf}}})^\dagger}=\mathbf{0}_{l\times l},\\
        \overline{\rmd \vw_{\hat{\vy}_{\text{\nf}}}(\rmd \vw_{\hat{\vy}_{\text{\nf}}})^\dagger}=\mathbf{I}_{l\times l}\rmd t,
    \end{gather*}
    then up to leading-order term $O(\rmd t)$ we have,
    \begin{align*}
         \rmd \left(\overline{\hat{\vy}_{\text{\nf}}'(\hat{\vy}_{\text{\nf}}')^\dagger}\right)&=\big((\ma-\mr{\nf}\mc)\overline{\hat{\vy}_{\text{\nf}}'(\hat{\vy}_{\text{\nf}}')^\dagger}+\overline{\hat{\vy}_{\text{\nf}}'(\hat{\vy}_{\text{\nf}}')^\dagger}(\ma^\dagger-\mc\mr{\nf})+\mb+\mr{\nf}\mc\mr{\nf}\big)\rmd t\\
         &=\big(\ma\mr{\nf}+\mr{\nf}\ma^\dagger+\mb- \mr{\nf}\mc \mr{\nf}\big)\rmd t,
    \end{align*}
    which is exactly the same evolution equation as the for the filter covariance \eqref{eq:filter2} after using the definitions of $\ma$, $\mb$, and $\mc$ (\eqref{eq:auxiliarymata}, \eqref{eq:auxiliarymatb}, and \eqref{eq:auxiliarymatc}, respectively), with the last equality being a result of the fact that either $\overline{\hat{\vy}}_{\text{f}}=\vm{\nf}$ or $\overline{\hat{\vy}}_{\text{f}}\to\vm{\nf}$ exponentially fast in time, depending on the initial conditions, and as such, by continuity, $\overline{\hat{\vy}_{\text{\nf}}'(\hat{\vy}_{\text{\nf}}')^{\tran}}=\mr{\nf}$ or $\overline{\hat{\vy}_{\text{\nf}}'(\hat{\vy}_{\text{\nf}}')^{\tran}}\to\mr{\nf}$ exponentially fast in time, correspondingly. As such, the last equality should be understood either as a genuine equality in the case of correct initial conditions, or as a convergence result as $t\uparrow$, with the convergence being exponentially fast, when the initial conditions are misidentified. This establishes \textbf{(b)}, which completes the proof.
\end{proof}

\subsection{Proof of Theorem \ref{thm:backwardsample}} \label{sec:Appendix_F}

\begin{proof}[\textbf{Proof of Theorem \ref{thm:backwardsample}}]
    We proceed in a similar manner as in the proof of Theorem \ref{thm:forwardsample}. We start with the distribution
\begin{equation*}
    \pp\big(\vy^j\big|\vy^{j+1},\vx^s, s\leq J\big)=\pp\big(\vy^j\big|\vy^{j+1},\vx^s, s\leq j+1\big),
\end{equation*}
where in Lemma \ref{lem:backwardscond} we showed that it is Gaussian, $\mathcal{N}_l(\mathbf{m}_{j+1}^j,\mathbf{P}_{j+1}^j)$, with mean and covariance given by \eqref{eq:backcond3}--\eqref{eq:backcond4}. Looking at the covariance matrix and using the definitions of the matrices appearing in its expression (up to leading-order $O(\dt)$), we have
\begin{align*}
     \mathbf{P}_{j+1}^j&=\mr{\nf}^j-\mathbf{C}_{j+1}^j\mathbf{C}^j_{22}(\mathbf{C}_{j+1}^j)^\dagger=\mr{\nf}^j-\mathbf{C}_{j+1}^j(\mathbf{C}^j_{12})^\dagger\\
     &=\mr{\nf}^j-\underbrace{\mathbf{E}^j}_{\displaystyle\mathclap{=\mathbf{C}_j^j+(\ms^\vy\circ\ms^\vx)^j\big((\ms^\vx\circ\ms^\vx)^{j}\big)^{-1}\mathbf{G}^{\vx,j}\dt+O(\dt^2)}}(\mathbf{I}_{l\times l}+\ml^{\vy,j}\dt)\mr{\nf}^j-\overbrace{\mathbf{F}^j}^{\displaystyle\mathclap{=-\mr{\nf}^j\big((\mathbf{K}^j)^\dagger-(\ml^{\vx,j})^\dagger\big((\ms^\vx\circ\ms^\vx)^{j}\big)^{-1}\big)+O(\dt)}}\ml^{\vx,j}\mr{\nf}^j\dt\\
     &=\mr{\nf}^j-\overbrace{\mathbf{C}_j^j}^{\displaystyle\mathclap{\hspace{2.4cm}=\mathbf{I}_{l\times l}-\mathbf{G}^{\vy,j}\dt}}(\mathbf{I}_{l\times l}+\ml^{\vy,j}\dt)\mr{\nf}^j-(\ms^\vy\circ\ms^\vx)^j\big((\ms^\vx\circ\ms^\vx)^{j}\big)^{-1}\mathbf{G}^{\vx,j}\mr{\nf}^j\dt\\
     &\hspace{1cm}+\mr{\nf}^j\big((\mathbf{K}^j)^\dagger-(\ml^{\vx,j})^\dagger\big((\ms^\vx\circ\ms^\vx)^{j}\big)^{-1}\big)\ml^{\vx,j}\mr{\nf}^j\dt+O(\dt^2)\\
     &=\mr{\nf}^j-\mr{\nf}^j+\underbrace{(\mathbf{G}^{\vy,j}-\ml^{\vy,j})\mr{\nf}^j}_{\displaystyle\mathclap{=(\ms^\vy\circ\ms^\vy)^j}}\dt-(\ms^\vy\circ\ms^\vx)^j\big((\ms^\vx\circ\ms^\vx)^{j}\big)^{-1}\mathbf{G}^{\vx,j}\mr{\nf}^j\dt\\
     &\hspace{1cm}+\mr{\nf}^j(\mathbf{K}^j)^\dagger\ml^{\vx,j}\mr{\nf}^j\dt-\mr{\nf}^j(\ml^{\vx,j})^\dagger\big((\ms^\vx\circ\ms^\vx)^{j}\big)^{-1}\ml^{\vx,j}\mr{\nf}^j\dt+O(\dt^2)\\
     &=(\ms^\vy\circ\ms^\vy)^j\dt-(\ms^\vy\circ\ms^\vx)^j\big((\ms^\vx\circ\ms^\vx)^{j}\big)^{-1}\mathbf{G}^{\vx,j}\mr{\nf}^j\dt\\
     &\hspace{0.5cm}+((\ms^\vy\circ\ms^\vx)^j+\mr{\nf}^j(\ml^{\vx,j})^\dagger)\big((\ms^\vx\circ\ms^\vx)^{j}\big)^{-1}\ml^{\vx,j}\mr{\nf}^j\dt\\
     &\hspace{1cm}-\mr{\nf}^j(\ml^{\vx,j})^\dagger\big((\ms^\vx\circ\ms^\vx)^{j}\big)^{-1}\ml^{\vx,j}\mr{\nf}^j\dt+O(\dt^2)\\
     &=(\ms^\vy\circ\ms^\vy)^j\dt-(\ms^\vy\circ\ms^\vx)^j\big((\ms^\vx\circ\ms^\vx)^{j}\big)^{-1}(\ml^{\vx,j}\mr{\nf}^j+(\ms^\vx\circ\ms^\vy)^j)\dt\\
     &\hspace{1cm}+(\ms^\vy\circ\ms^\vx)^j\big((\ms^\vx\circ\ms^\vx)^{j}\big)^{-1}\ml^{\vx,j}\mr{\nf}^j\dt+O(\dt^2)\\
     &=(\ms^\vy\circ\ms^\vy)^j\dt-(\ms^\vy\circ\ms^\vx)^j\big((\ms^\vx\circ\ms^\vx)^{j}\big)^{-1}(\ms^\vx\circ\ms^\vy)^j\dt+O(\dt^2)\\
     &=\mb^j\dt+O(\dt^2),
\end{align*}
which shows that up to leading-order $\mathbf{P}_{j+1}^j=\mb^j\dt+O(\dt^2)$. This fully accounts for the diffusion tensor appearing in \eqref{eq:backwardsample}, and as such we now need to just account for the drift term through $\mathbf{m}_{j+1}^j-\vy^{j+1}$.

Using \eqref{eq:backcond3}, we have
\begin{align*}
    \mathbf{m}_{j+1}^j-\vy^{j+1}&=\vm{\nf}^j+(\mathbf{E}^j-\mathbf{I}_{l\times l})\vy^{j+1}-\mathbf{E}^j\big((\mathbf{I}_{l\times l}+\ml^{\vy,j}\dt)\vm{\nf}^j+\vf^{\vy,j}\dt\big)\\
    &\hspace{2cm}+\mathbf{F}^j\big(\vx^{j+1}-\vx^{j}-(\ml^{\vx,j}\vm{\nf}^j+\vf^{\vx,j})\dt\big).
\end{align*}
We analyze each term in this expression. We already know from the proof of Theorem \ref{thm:smoothing} (see derivations for the $\boldsymbol{\gamma}_3$ and $\boldsymbol{\gamma}_4$ terms) that
\begin{equation*}
    \mathbf{F}^j\big(\vx^{j+1}-\vx^{j}-(\ml^{\vx,j}\vm{\nf}^j+\vf^{\vx,j})\dt\big)=-(\ms^\vy\circ\ms^\vx)^j\big((\ms^\vx\circ\ms^\vx)^{j}\big)^{-1}\big(\vx^{j+1}-\vx^{j}-(\ml^{\vx,j}\vm{\nf}^j+\vf^{\vx,j})\dt\big).
\end{equation*}
As for the other terms, we have
\begin{align*}
    \vm{\nf}^j-\mathbf{E}^j\big((\mathbf{I}_{l\times l}+\ml^{\vy,j}\dt)\vm{\nf}^j+\vf^{\vy,j}\dt\big)&=\vm{\nf}^j-(\mathbf{I}_{l\times l}-\mathbf{G}^{\vy,j}\dt)\big((\mathbf{I}_{l\times l}+\ml^{\vy,j}\dt)\vm{\nf}^j+\vf^{\vy,j}\dt\big)\\
    &\hspace{2cm}-(\ms^\vy\circ\ms^\vx)^j\big((\ms^\vx\circ\ms^\vx)^{j}\big)^{-1}\mathbf{G}^{\vx,j}\vm{\nf}^j\dt\\
    &=-\ml^{\vy,j}\vm{\nf}^j\dt-\vf^{\vy,j}\dt+\mathbf{G}^{\vy,j}\vm{\nf}^j\dt\\
    &\hspace{2cm}-(\ms^\vy\circ\ms^\vx)^j\big((\ms^\vx\circ\ms^\vx)^{j}\big)^{-1}\mathbf{G}^{\vx,j}\vm{\nf}^j\dt,
\end{align*}
and
\begin{equation*}
    (\mathbf{E}^j-\mathbf{I}_{l\times l})\vy^{j+1}=((\ms^\vy\circ\ms^\vx)^j\big((\ms^\vx\circ\ms^\vx)^{j}\big)^{-1}\mathbf{G}^{\vx,j}-\mathbf{G}^{\vy,j})\vy^{j+1}\dt.
\end{equation*}
Adding these last two term then gives
\begin{align*}
    &\vm{\nf}^j+(\mathbf{E}^j-\mathbf{I}_{l\times l})\vy^{j+1}-\mathbf{E}^j\big((\mathbf{I}_{l\times l}+\ml^{\vy,j}\dt)\vm{\nf}^j+\vf^{\vy,j}\dt\big)\\
    &\hspace{2cm}=\big(\mathbf{G}^{\vy,j}-(\ms^\vy\circ\ms^\vx)^j\big((\ms^\vx\circ\ms^\vx)^{j}\big)^{-1}\mathbf{G}^{\vx,j}\big)(\vm{\nf}^j-\vy^{j+1})\dt-\ml^{\vy,j}\vm{\nf}^j\dt-\vf^{\vy,j}\dt\\
    &\hspace{2cm}=(\mb^j(\mr{\nf}^j)^{-1}+\ma^j)(\vm{\nf}^j-\vy^{j+1})\dt-\ml^{\vy,j}\vm{\nf}^j\dt-\vf^{\vy,j}\dt.
\end{align*}
By then summing up these results we recover
\begin{align*}
    \mathbf{m}_{j+1}^j-\vy^{j+1}&=-\ml^{\vy,j}\vm{\nf}^j\dt-\vf^{\vy,j}\dt+(\mb^j(\mr{\nf}^j)^{-1}+\ma^j)(\vm{\nf}^j-\vy^{j+1})\dt\\
    &\hspace{1cm}+(\ms^\vy\circ\ms^\vx)^j\big((\ms^\vx\circ\ms^\vx)^{j}\big)^{-1}\big(\vx^{j}-\vx^{j+1}+(\ml^{\vx,j}\vm{\nf}^j+\vf^{\vx,j})\dt\big),
\end{align*}
and so by combining all of these and taking the limit $\dt\to0^+$, leads to
\begin{equation*}
        \smooth{\rmd \hat{\vy}_{\text{\ns}}} = (-\ml^\vy \vm{\nf}-\vf^\vy +(\mb\mr{\nf}^{-1}+\ma)(\vm{\nf}-\hat{\vy}_{\text{\ns}}))\rmd t+ (\ms^\vy\circ \ms^\vx)(\ms^\vx\circ \ms^\vx)^{-1}\big(\smooth{\rmd \vx}+(\ml^\vx\vm{\nf}+\vf^\vx)\rmd t \big) +\mb^{1/2}\rmd \vw_{\hat{\vy}_{\text{\ns}}},
\end{equation*}
which is exactly what we wanted to prove.
\end{proof}

\subsection{Proof of Theorem \ref{thm:backwardsamplesmoother}} \label{sec:Appendix_G}

\begin{proof}[\textbf{Proof of Theorem \ref{thm:backwardsamplesmoother}}]
    In this proof, the overline notation denotes taking the expectation of the quantity in question conditioned on the $\sigma$-algebra $\cF_T^{\vx}$. Furthermore, since we are conditioning on the whole observational period, intricacies like those that appeared in the proof of Theorem \ref{thm:forwardsamplefilter} should not arise here. We plug-in the mean-fluctuation decomposition \eqref{eq:reynoldsdecomp} into \eqref{eq:backwardsample}, or equivalently by linearity, take the ensemble average of \eqref{eq:alternativebacksample} conditioned on the complete observational $\sigma$-algebra $\cF_T^{\vx}$, to get
    \begin{equation}
        \overline{\smooth{\rmd \hat{\vy}}}_{\text{s}}=\smooth{\rmd \vm{\ns}}-(\ma+\mb\mr{\nf}^{-1})(\overline{\hat{\vy}}_{\text{s}}-\vm{\ns})\rmd t,\label{eq:backensembleavgeqn}
    \end{equation}
    since $\ee{\vm{\ns}\big|\cF_T^{\vx}}=\ee{\ee{\vyt\big|\cF_T^{\vx}}\big|\cF_T^{\vx}}=\vm{\ns}$, by the stability property of conditional expectations, and $\vw_{\hat{\vy}_{\ns}}$ is a zero expectation source term. 
    
    Now, in a similar argument to the one found in the proof of Theorem \ref{thm:forwardsamplefilter}, if $\overline{\hat{\vy}}_{\text{s}}$ coincides with the smoother mean at the endpoint $t=T$, then by uniqueness of the solution to \eqref{eq:backensembleavgeqn} (again under all the underlying assumptions considered thus far), they should do so at all time instants, regardless of whether we are talking about its weak or strong solution. Note here that unlike in the filter-based sampling strategy (see proof of Theorem \ref{thm:forwardsamplefilter}), where it was possible to obtain exponentially fast converging consistency even in the case of incorrect or misidentified initial conditions, here it is required for the consistency to hold that the backward samples match the smoother mean at $t=T$, which in turn matches the filter mean. This is because now time moves backwards, so in the case of misidentified ``terminal" conditions, i.e., $\overline{\hat{\vy}}_{\text{s}}\neq\vm{\ns}$ at the end point $t=T$, even though there the filter and smoother statistics necessarily coincide, it could be that the relaxation to the smoother mean is slow as time ``devolves" from $T$ down to $0$, or equivalently as time runs backwards. Even though the assumption on the spectrum of $\ma+\mb\mr{\nf}^{-1}$ provides the exponentially mean-square stability of the smoother covariance statistics by the theory of backward random ordinary differential equations (see \eqref{eq:revbackinter2}), it might still fail to provide fast-converging consistency between the Gaussian statistics of the backward samples and smoother state estimation by virtue of the linearity in \eqref{eq:backensembleavgeqn}; this can also be observed by the behavior of the expected spectrum of the damping feedback $\ma+\mb\mr{\nf}^{-1}$ in Figure \ref{fig:case_study_fig_3}. This establishes \textbf{(a)}.

    Now, subtracting \eqref{eq:backensembleavgeqn} from \eqref{eq:alternativebacksample} yields an SDE for the fluctuation part of $\hat{\vy}_{\text{\ns}}$:
    \begin{equation} \label{eq:backresidualeqn}
        \smooth{\rmd\hat{\vy}}_{\text{s}}'=-(\ma+\mb\mr{\nf}^{-1})(\hat{\vy}_{\text{\ns}}-\vm{\ns})\rmd t+\mb^{1/2}\rmd \vw_{\hat{\vy}_{\text{\ns}}}.
    \end{equation}
    The covariance is given by $\overline{\hat{\vy}_{\text{\ns}}'\left(\hat{\vy}_{\text{\ns}}'\right)^\dagger}$, which by Itô's lemma \cite{mao2008stochastic, evans2012introduction, oksendal2003stochastic, gardiner2009stochastic, steele2001stohastic} and linearity, has a differential given by
    \begin{equation} \label{eq:backitolemmacov}
        \smooth{\rmd \big(\overline{\hat{\vy}_{\text{\ns}}'\left(\hat{\vy}_{\text{\ns}}'\right)^\dagger}\big)}=\overline{\hat{\vy}_{\text{\ns}}'\Big(\smooth{\rmd\hat{\vy}}_{\text{s}}'\Big)^\dagger}+\overline{\smooth{\rmd\hat{\vy}}_{\text{s}}'\Big(\hat{\vy}_{\text{\ns}}'\Big)^\dagger}-\overline{\smooth{\rmd\hat{\vy}}_{\text{s}}'\Big(\smooth{\rmd\hat{\vy}}_{\text{s}}'\Big)^\dagger}.
    \end{equation}
    As such, by plugging-in \eqref{eq:backresidualeqn} into \eqref{eq:backitolemmacov}, similar to the proof of Theorem \ref{thm:forwardsamplefilter}, up to leading-order term $O(\rmd t)$ we end up with
    \begin{equation} \label{eq:quadformcovsmooth}
         \smooth{\rmd \big(\overline{\hat{\vy}_{\text{\ns}}'\left(\hat{\vy}_{\text{\ns}}'\right)^\dagger}\big)}=-\big((\ma+\mb\mr{\nf}^{-1})\mr{\ns}+\mr{\ns}(\ma+\mb\mr{\nf}^{-1})^\dagger-\mb\big)\rmd t,
    \end{equation}
    since we have established in the proof of \textbf{(a)} that $\overline{\hat{\vy}_{\text{\ns}}}=\vm{\ns}$ necessarily, for all time instants $T\geq t\geq 0$ if they match at the end point (or $\overline{\hat{\vy}_{\text{\ns}}}$ converges to the smoother mean as $t$ decreases from $T$ towards $0$ exponentially fast, in the case where they differ at $t=T$), and so
    \begin{equation*}
        \ee{\hat{\vy}_{\text{\ns}}'(\hat{\vy}_{\text{\ns}}-\vm{\ns})^\dagger\Big|\cF_T^{\vx}}=\ee{(\hat{\vy}_{\text{\ns}}-\vm{\ns})\left(\hat{\vy}_{\text{\ns}}'\right)^\dagger\Big|\cF_T^{\vx}}=\mr{\ns},
    \end{equation*}
    with \eqref{eq:quadformcovsmooth} being exactly the same as the smoother covariance equation given in \eqref{eq:revbackinter2}. This establishes \textbf{(b)}, thus ending the proof.
\end{proof}

\bibliographystyle{unsrt}
\bibliography{references2}

\end{document}